\documentclass[11pt,leqno]{article}

\usepackage{amsthm,amsfonts,amssymb,amsmath,oldgerm}
\usepackage{fullpage}
\usepackage{graphicx}
\usepackage{mathrsfs}
\usepackage{xcolor}
\usepackage[colorinlistoftodos,disable]{todonotes}
\usepackage[hidelinks]{hyperref}
\hypersetup{
    colorlinks=false,
    pdfborder={0 0 0},
   pdfauthor={B. de Rijk},
   pdftitle={Nonlinear stability and asymptotic behavior of periodic wave trains in reaction-diffusion systems against C\_{ub}-perturbations}
}
\usepackage{enumitem}

\makeatletter
\def\namedlabel#1#2{\begingroup
    #2%
    \def\@currentlabel{#2}%
    \phantomsection\label{#1}\endgroup
}
\makeatother

\numberwithin{equation}{section}

\newcommand{\Z}{\mathbb Z}
\newcommand{\R}{\mathbb R}
\newcommand{\C}{\mathbb C}
\newcommand{\vt}{\widetilde{v}}
\newcommand{\ri}{\mathrm{i}}
\newcommand{\re}{\mathrm{e}}
\newcommand{\de}{\mathrm{d}}
\newcommand{\El}{\mathcal{L}}
\newcommand{\xx}{{\zeta}}
\newcommand{\xt}{{\bar{\zeta}}}
\newcommand{\NT}{\widetilde{\mathcal{N}}}
\newcommand{\Non}{\mathcal{N}}
\newtheorem{theorem}{Theorem}[section]
\newtheorem{proposition}[theorem]{Proposition}
\newtheorem{corollary}[theorem]{Corollary}
\newtheorem{lemma}[theorem]{Lemma}
\newtheorem{remark}[theorem]{Remark}
\theoremstyle{definition}

\allowdisplaybreaks[3]

\title{Nonlinear stability and asymptotic behavior of periodic wave trains in reaction-diffusion systems against $C_{\mathrm{ub}}$-perturbations}

\author{Bj\"orn de Rijk\thanks{Department of Mathematics, Karlsruhe Institute of Technology, Englerstra\ss e 2, 76131 Karlsruhe, Germany; \texttt{bjoern.de-rijk@kit.edu}}}

\date{\today}

\begin{document}

\maketitle
\begin{abstract}
We present a nonlinear stability theory for periodic wave trains in reaction-diffusion systems, which relies on pure $L^\infty$-estimates only. Our analysis shows that localization or periodicity requirements on perturbations, as present in the current literature, can be completely lifted. Inspired by previous works considering localized perturbations, we decompose the semigroup generated by the linearization about the wave train and introduce a spatio-temporal phase modulation to capture the most critical dynamics, which is governed by a viscous Burgers' equation. We then aim to close a nonlinear stability argument by iterative estimates on the corresponding Duhamel formulation, where, hampered by the lack of localization, we must rely on diffusive smoothing to render decay of the semigroup. Yet, this decay is not strong enough to control all terms in the Duhamel formulation. We address this difficulty by applying the Cole-Hopf transform to eliminate the critical Burgers'-type nonlinearities. Ultimately, we establish nonlinear stability of diffusively spectrally stable wave trains against $C_{\mathrm{ub}}$-perturbations. Moreover, we show that the perturbed solution converges to a modulated wave train, whose phase and wavenumber are approximated by solutions to the associated viscous Hamilton-Jacobi and Burgers' equation, respectively.
\end{abstract}

\section{Introduction}

A paradigmatic class of pattern-forming systems, which are characterized by their rich dynamics, yet have a simple form, is the class of reaction-diffusion systems
\begin{equation}\label{RD}
\partial_t u = Du_{xx} + f(u), \qquad x \in\R,\, t \geq 0,\, u\in\R^n,
\end{equation}
where $n \in \mathbb N$, $D \in \R^{n \times n}$ is a symmetric, positive-definite matrix, and $f \colon \R^n \to \R^n$ is a smooth nonlinearity. The well-known Turing mechanism, describing pattern-forming processes in developmental biology, fluid mechanics, materials science and more~\cite{CrossHohenberg}, asserts that periodic traveling waves, or \emph{wave trains}, are typically the first patterns to arise after a homogeneous background state in~\eqref{RD} becomes unstable, and serve as the building blocks for more complicated structures. Wave trains are solutions to~\eqref{RD} of the form $u_0(x,t) = \phi_0(k_0 x -\omega_0 t)$ with wavenumber $k_0 \in \R \setminus \{0\}$, temporal frequency $\omega_0 \in \R$, propagation speed $c = \omega_0/k_0$ and $1$-periodic profile function $\phi_0(\zeta)$. Thus, by switching to the comoving frame $\zeta = k_0x-\omega_0 t$, we find that $\phi_0$ is a stationary solution to 
\begin{equation}\label{RD2}
\partial_t u = k_0^2Du_{\zeta\zeta} + \omega_0 u_\zeta + f(u).
\end{equation}

Despite the structural importance and apparent simplicity of wave trains, their nonlinear stability on spatially extended domains has proven to be challenging. The main obstruction is that the linearization of~\eqref{RD2} about the wave train is a periodic differential operator with continuous spectrum touching the imaginary axis at the origin due to translational invariance. The absence of a spectral gap prohibits exponential convergence of the perturbed solution towards a translate of the original profile, as in the case of finite domains with periodic boundary conditions. Instead, one expects diffusive behavior on the linear level, so that only \emph{algebraic} decay of the associated semigroup can be realized by giving up localization or by exploiting its smoothing action, see Remark~\ref{rem:lindecay}. In particular, stability could strongly depend on the nonlinearity and the selected space of perturbations, see Remark~\ref{rem:nonlheat}.

The 30-year open problem of nonlinear stability of wave trains on spatially extended domains was first resolved in~\cite{CEE92} for the real Ginzburg-Landau equation considering localized perturbations. The extension to pattern-forming systems without gauge symmetry, such as~\eqref{RD}, resisted many attempts, until it was tackled for the Swift-Hohenberg equation in~\cite{SCH} by incorporating mode filters in Bloch frequency domain, see also~\cite{SCHN}. Yet, a breakthrough in the understanding of perturbed wave trains in reaction-diffusion systems~\eqref{RD} was the introduction of a spatio-temporal \emph{phase modulation} in~\cite{DSSS}. Instead of controlling the difference $u(\xx,t) - \phi_0(\xx)$ between the perturbed solution $u(\xx,t)$ to~\eqref{RD2} and the wave train $\phi_0(\xx)$, one considers the \emph{modulated perturbation}
\begin{align} v(\zeta,t) = u(\xx - \gamma(\xx,t),t) - \phi_0(\xx), \label{e:defv}\end{align}
and aims to capture the most critical dynamics, which arises due to translational invariance of the wave train, by the phase modulation $\gamma(\xx,t)$. It is shown in~\cite{DSSS} that the leading-order dynamics of $\gamma(\xx,t)$ is described by the viscous Hamilton-Jacobi equation
\begin{align}
\partial_t \breve\gamma = d\breve\gamma_{\xx\xx} + a \breve\gamma_{\xx} + \nu \breve\gamma_\zeta^2, \label{e:HamJac}
\end{align}
whose parameters $d > 0$, $a, \nu \in \R$ can be determined explicitly in terms of linear and nonlinear dispersion relations, see~\S\ref{sec:statement} for details. The modulational ansatz~\eqref{e:defv} led to fast developments in the nonlinear stability theory of wave trains against localized perturbations. Using a rich blend of methods such as the renormalization group approach~\cite{IYSA,SAN3}, iterative $L^1$-$H^k$-estimates~\cite{JONZNL,JONZW,JONZ} or pointwise estimates~\cite{RS21,JUN,JUNNL}, one was able to each time streamline, sharpen and extend previous results, e.g.~by deriving the asymptotic behavior of the perturbation in varying degrees of detail, or by allowing for a large, or nonlocalized, initial phase modulation, see~\S\ref{sec:modnonlocal}-\S\ref{sec:approx}. 

These works confirm that wave trains in reaction-diffusion systems are more robust than one would initially expect from the fact that the linearization possesses continuous spectrum up to the imaginary axis, see Remark~\ref{rem:nonlheat}. This is strongly related to the fact that the critical phase dynamics is governed by the viscous Hamilton-Jacobi equation~\eqref{e:HamJac}, which reduces to a linear heat equation after applying the Cole-Hopf transform. This begs the question of whether extension to larger classes of perturbations is possible and, in particular, whether localization conditions on perturbations can be completely lifted; after all, bounded solutions to~\eqref{e:HamJac} can be controlled using well-known tools such as the maximum principle or the Cole-Hopf transform.

In this paper, we answer this question in the affirmative, shedding new light on the persistence and resilience of periodic structures in reaction-diffusion systems. Here, we take inspiration from a recent study~\cite{HDRS22}, establishing nonlinear stability of periodic roll waves in the real Ginzburg-Landau equation in a pure $L^\infty$-framework. Despite the slower decay rates due to the loss of localization, the nonlinear argument can be closed in~\cite{HDRS22} relying on diffusive smoothing only. Yet, the real Ginzburg-Landau equation is a special equation for which the analysis simplifies significantly due to its gauge symmetry and for which the most critical nonlinear terms do not appear due to its reflection symmetry. Therefore, the analysis in this paper requires fundamentally new ideas. Before explaining these in~\S\ref{sec:outlineproof} below, we first formulate our main result in~\S\ref{sec:statement}.

\begin{remark} \label{rem:lindecay}
{\upshape
The heat semigroup $\smash{\re^{\partial_x^2 t}}$ is only bounded as an operator acting on $L^p(\R)$ for $1 \leq p \leq \infty$. Nevertheless, algebraic decay can be obtained by giving up localization, or by exploiting its smoothing action. Indeed, $\smash{\re^{\partial_x^2 t}}$ decays at rate $\smash{t^{-\frac{1}{2} + \frac{1}{2p}}}$ as an operator from $L^1(\R)$ into $L^p(\R)$, and $\partial_x \smash{\re^{\partial_x^2 t}}$ decays on $L^p(\R)$ at rate $\smash{t^{-\frac{1}{2}}}$. In a nonlinear argument the lost localization can be regained through the nonlinearity, e.g.~if a function $u$ is $L^2$-localized then its square $u^2$ is $L^1$-localized. Moreover, if the nonlinearity admits derivatives, they can often be moved onto the semigroup in the Duhamel formulation facilitating decay due to diffusive smoothing. We refer to~\cite[Section~14.1.3]{SU17book} and~\cite[Section~2]{HDRS22} for simple examples illustrating how a nonlinear argument can be closed through these principles. 
}\end{remark}

\begin{remark} \label{rem:nonlheat}
{\upshape
The stability of solutions to reaction-diffusion systems could heavily depend on the nonlinearity and the selected space of perturbations in case the linearization possesses continuous spectrum up to the imaginary axis. To illustrate the latter, we consider the scalar heat equation
\begin{align*} \partial_t u = u_{xx} + g(u,u_x),\end{align*}
with smooth nonlinearity $g \colon \R^2 \to \R$. For $g(u,u_x) = u^4$ one proves with the aid of localization-induced decay, cf.~\cite[Section~14.1.3]{SU17book}, that the rest state $u = 0$ is stable against perturbations from $L^1(\R) \cap L^\infty(\R)$, whereas the comparison principle shows that $u = 0$ is unstable against perturbations from $L^\infty(\R)$ (they blow up in finite time). On the other hand, taking $g(u,u_x) = u^2$ the rest state is even unstable against perturbations from $L^1(\R) \cap L^\infty(\R)$, cf.~\cite{FUJI}, whereas for $g(u,u_x) = u_x^3$ its stability against perturbations from $L^\infty(\R)$ can be shown with the aid of diffusive smoothing~\cite[Section~2]{HDRS22}.
}\end{remark}

\subsection{Statement of main result} \label{sec:statement}

In this paper we establish the nonlinear stability and describe the asymptotic behavior of wave-train solutions to~\eqref{RD} subject to bounded and uniformly continuous perturbations.

We formulate the hypotheses for our main result. First, we assume the existence of a wave train:
\begin{itemize}
\item[\namedlabel{assH1}{\upshape (H1)}] There exist a wavenumber $k_0 \in \R \setminus \{0\}$ and a temporal frequency $\omega_0 \in \R$ such that~\eqref{RD} admits a wave-train solution $u_0(x,t)=\phi_0(k_0x - \omega_0 t)$, where the profile function $\phi_0 \colon \R \to \R^n$ is nonconstant, smooth and $1$-periodic.
\end{itemize}

Next, we pose spectral stability assumptions on the wave train $u_0(x,t)$. Linearizing~\eqref{RD2} about its stationary solution $\phi_0$, we obtain the $1$-periodic differential operator
\begin{align*}
\El_0 u = k_0^2Du_{\zeta\zeta} + \omega_0 u_\zeta + f'(\phi_0(\zeta))u,
\end{align*}
acting on $C_{\mathrm{ub}}(\R)$ with domain $D(\El_0) = C_{\mathrm{ub}}^2(\R)$, where $C_{\mathrm{ub}}^m(\R)$, $m \in \mathbb{N}_0$, denotes the space of bounded and uniformly continuous functions, which are $m$ times differentiable and whose $m$ derivatives are also bounded and uniformly continuous. We equip $C_{\mathrm{ub}}^m(\R)$ with the standard $\smash{W^{m,\infty}}$-norm, so that it is a Banach space. An important reason to consider bounded and uniformly continuous perturbations, instead of perturbations in $L^\infty(\R)$, is that the operator $\El_0$ is densely defined on the space $C_{\mathrm{ub}}(\R)$, but not on $L^\infty(\R)$, see~\cite[Theorem~3.1.7 and Corollary~3.1.9]{LUN}. We refer to Remark~\ref{rem:linfty} for further details.

Applying the Floquet-Bloch transform to $\El_0$ yields the family of operators
\begin{align*}
\El(\xi) u = k_0^2D\left(\partial_\zeta + \ri \xi\right)^2 u + \omega_0 \left(\partial_\zeta + \ri \xi\right) u + f'(\phi_0(\zeta))u,
\end{align*}
posed on $L_{\mathrm{per}}^2(0,1)$ with domain $D(\El(\xi)) = H_{\mathrm{per}}^2(0,1)$ parameterized by the Bloch frequency variable $\xi \in [-\pi,\pi)$. It is well-known that the spectrum decomposes as
$$\sigma(\El_0) = \bigcup_{\xi \in [-\pi,\pi)} \sigma(\El(\xi)).$$
Here, we note that the Bloch operator $\El(\xi)$ has compact resolvent, and thus discrete spectrum, for each $\xi \in [-\pi,\pi)$. The conditions for \emph{diffusive spectral stability} can now be formulated as follows:
\begin{itemize}
\setlength\itemsep{0em}
\item[\namedlabel{assD1}{\upshape (D1)}] It holds $\sigma(\El_0)\subset\{\lambda\in\C:\Re(\lambda)<0\}\cup\{0\}$;
\item[\namedlabel{assD2}{\upshape (D2)}] There exists $\theta>0$ such that for any $\xi\in[-\pi,\pi)$ we have $\Re\,\sigma(\El(\xi))\leq-\theta \xi^2$;
\item[\namedlabel{assD3}{\upshape (D3)}] $0$ is a simple eigenvalue of $\El(0)$.
\end{itemize}
We emphasize that these diffusive spectral stability conditions, which were first introduced in~\cite{SCH}, are standard in all nonlinear stability analyses of periodic wave trains in reaction-diffusion systems, cf.~\cite{IYSA,JONZNL,JONZW,JONZ,JUN,JUNNL,SAN3,SCHN,SCH,SCHT}. Recalling that the spectrum of $\El_0$ touches the origin due to translational invariance, they resemble the most stable nondegenerate spectral configuration. Examples of reaction-diffusion systems, in which diffusively spectrally stable wave trains have been shown to exist, include the complex Ginzburg-Landau equation~\cite{vanH}, the Gierer-Meinhardt system~\cite{PLO} and the Brusselator model~\cite{SUKH}. Typically, spectral stability analyses of wave trains rely on perturbative arguments, which for instance exploit that the wave trains are constructed close to a homogeneous rest state undergoing a Turing bifurcation~\cite{MIEL,SCHN}, employ the stability of a nearby traveling pulse solution~\cite{SAS} or take advantage of the slow-fast structure of the system~\cite{BDR2}.

By translational invariance and Hypothesis~\ref{assD3} the kernel of the Bloch operator $\El(0)$ is spanned by the derivative $\phi_0' \in H^2_{\mathrm{per}}(0,1)$ of the wave train. Thus, $0$ must also be a simple eigenvalue of its adjoint $\El(0)^*$. We denote by $\smash{\widetilde{\Phi}_0} \in H^2_{\mathrm{per}}(0,1)$ the corresponding eigenfunction satisfying 
\begin{align} \label{e:adjoint} \big\langle \widetilde{\Phi}_0,\phi_0'\big\rangle_{L^2(0,1)} = 1.\end{align}

It is a direct consequence of the implicit function theorem that $\phi_0$ is part of a family $u_k(x,t) = \phi(k x - \omega(k) t;k)$ of smooth $1$-periodic wave-trains solutions to~\eqref{RD} for an open range of wavenumbers $k$ around $k_0$, such that $\omega(k_0) = \omega_0$ and $\phi(\cdot;k_0) = \phi_0$, cf.~Proposition~\ref{prop:family}. The function $\omega(k)$ is the so-called \emph{nonlinear dispersion relation}, describing the dependency of the frequency on the wavenumber. Note that by translational invariance we can always arrange for
\begin{align} \label{e:gauge} \big\langle \widetilde{\Phi}_0,\partial_k \phi(\cdot,k_0)\big\rangle_{L^2(0,1)} = 0,\end{align}
cf.~\cite[Section~4.2]{DSSS}. The coefficients $d>0$ and $a,\nu \in \R$ of the viscous Hamilton-Jacobi equation~\eqref{e:HamJac}, describing the leading-order phase dynamics, can now be expressed as
\begin{align} \label{e:defad}
 a = \omega_0 - k_0\omega'(k_0), \qquad d = k_0^2\big\langle \widetilde{\Phi}_0,D \phi_0' + 2k_0D \partial_{\zeta k} \phi(\cdot;k_0)\big\rangle_{L^2(0,1)}, \qquad \nu = -\frac{1}{2}k_0^2 \omega''(k_0). 
\end{align}

We are in the position to state our main result, which establishes nonlinear stability of diffusively spectrally stable wave trains against $C_{\mathrm{ub}}$-perturbations and yields convergence of the modulated perturbed solution towards the wave train, where the modulation is approximated by a solution to the viscous Hamilton-Jacobi equation~\eqref{e:HamJac}.

\begin{theorem} \label{t:mainresult}
Assume~\ref{assH1} and~\ref{assD1}-\ref{assD3}. Then, there exist constants $\epsilon, M > 0$ such that, whenever $v_0 \in C_{\mathrm{ub}}(\R)$ satisfies
\[
E_0:=\left\|v_0\right\|_\infty <\epsilon,
\]
there exist a scalar function $\gamma \in C^\infty\big([0,\infty) 
\times \R,\R\big)$ with $\gamma(0) = 0$ and $\gamma(t) \in C_{\mathrm{ub}}^m(\R)$ for each $m \in \mathbb{N}_0$ and $t \geq 0$, and a unique classical global solution 
\begin{align} \label{regu}
 u \in X := C\big([0,\infty),C_{\mathrm{ub}}(\R)\big) \cap C\big((0,\infty),C_{\mathrm{ub}}^2(\R)\big) \cap C^1\big((0,\infty),C_{\mathrm{ub}}(\R)\big),
\end{align}
to~\eqref{RD2} with initial condition $u(0)=\phi_0 + v_0$, which enjoy the estimates
\begin{align} 
\label{e:mtest10}
\left\|u(t)-\phi_0\right\|_\infty + \frac{\sqrt{t}}{\sqrt{1+t}}\|u_\xx(t)-\phi_0'\|_{\infty} &\leq ME_0, \\
\label{e:mtest11}
\left\|u(\cdot-\gamma(\cdot,t),t)-\phi_0\right\|_\infty &\leq \frac{ME_0}{\sqrt{1+t}},\\
\label{e:mtest12}
\left\|u(\cdot-\gamma(\cdot,t),t)-\phi(\cdot\,;k_0(1+\gamma_\xx(\cdot,t))\right\|_\infty &\leq \frac{ME_0 \log(2+t)}{1+t},
\end{align}
and
\begin{align}  \label{e:mtest2}
\|\gamma(t)\|_\infty \leq ME_0, \qquad \left\|\gamma_\xx(t)\right\|_\infty, \|\partial_t \gamma(t)\|_\infty 
\leq \frac{ME_0}{\sqrt{1+t}}, \qquad \left\|\gamma_{\xx\xx}(t)\right\|_\infty \leq \frac{ME_0 \log(2+t)}{1+t},
\end{align}
for all $t \geq 0$. Moreover, there exists a unique classical global solution $\breve{\gamma} \in X$ with initial condition $\breve{\gamma}(0) = \widetilde{\Phi}_0^*v_0$ to the viscous Hamilton-Jacobi equation~\eqref{e:HamJac}, with coefficients~\eqref{e:defad}, such that we have the approximation
\begin{align}
\label{e:mtest3}
t^{\frac{j}{2}}\left\|\partial_\xx^j \left(\gamma(t) - \breve{\gamma}(t)\right)\right\|_\infty \leq M \left(E_0^2 + \frac{E_0 \log(2+t)}{\sqrt{1+t}}\right),
\end{align}
for $j = 0,1$ and $t \geq 0$.
\end{theorem}

Theorem~\ref{t:mainresult} establishes Lyapunov stability of the wave train $\phi_0$ as a solution to~\eqref{RD2} in $C_{\mathrm{ub}}(\R)$, cf.~estimate~\eqref{e:mtest10}. Naturally, asymptotic stability cannot be expected due to translational invariance of the wave train and the fact that any sufficiently small translate is a $C_{\mathrm{ub}}$-perturbation. In fact, the temporal decay rates presented in Theorem~\ref{t:mainresult} are sharp (up to possibly a logarithm), see~\S\ref{sec:optimal} for details. Yet, Theorem~\ref{t:mainresult} does imply asymptotic convergence of the perturbed solution $u(t)$ towards a \emph{modulated} wave train.

\begin{corollary} \label{cor:psit}
Assume~\ref{assH1} and~\ref{assD1}-\ref{assD3}. Then, there exist constants $\epsilon, M > 0$ such that, whenever $v_0 \in C_{\mathrm{ub}}(\R)$ satisfies $E_0:=\left\|v_0\right\|_\infty <\epsilon$, the solution $u(\xx,t)$ to~\eqref{RD2} and the phase function $\gamma(\xx,t)$, both established in Theorem~\ref{t:mainresult}, satisfy
\begin{align} 
\begin{split}
\left\|u\left(\cdot,t\right)-\phi_0\left(\cdot + \gamma(\cdot,t)\right)\right\|_\infty &\leq \frac{M E_0}{\sqrt{1+t}},\\
\left\|u\left(\cdot,t\right)-\phi\left(\cdot + \gamma(\cdot,t)\left(1+\gamma_\xx(\cdot,t)\right);k_0\left(1+\gamma_\xx(\cdot,t)\right)\right)\right\|_\infty &\leq \frac{M E_0 \log(2+t)}{1+t},
\end{split} 
\label{e:mtest6}\end{align}
for $t \geq 0$.
\end{corollary}
\begin{proof}
Take $E_0 > 0$ so small that estimate~\eqref{e:mtest2} implies that $\|\gamma_\xx(t)\|_\infty \leq1$. Then, the map $\psi_t \colon \R \to \R$ given by $\psi_t(\xx) = \xx - \gamma(\xx,t)$ is invertible for each $t \geq 0$. We rewrite $\psi_t(\psi_t^{-1}(\xx)) = \xx$ as $\psi_t^{-1}(\xx) = \xx + \gamma(\psi_t^{-1}(\xx),t)$ to obtain
\begin{align*}\psi_t^{-1}(\xx) - \xx - \gamma(\xx,t)\left(1+\gamma_\xx(\xx,t)\right) &= \gamma_\xx(\xx,t) \left(\gamma\left(\xx+\gamma\left(\psi_t^{-1}(\xx),t\right),t\right) - \gamma(\xx,t)\right)\\ &\qquad + \gamma\left(\xx+\gamma\left(\psi_t^{-1}(\xx),t\right),t\right) - \gamma(\xx,t) - \gamma_\xx(\xx,t) \gamma\left(\psi_t^{-1}(\xx),t\right).\end{align*}
Next, we apply Taylor's theorem to the latter and find
\begin{align}
\sup_{\xx \in \R} \left|\psi_t^{-1}(\xx) - \xx - \gamma(\xx,t)\left(1 + \gamma_\xx(\xx,t)\right)\right| \leq \left(\left\|\gamma_\xx(t)\right\|_\infty^2 + \frac{1}{2}\left\|\gamma_{\xx\xx}(t)\right\|_\infty\left\|\gamma(t)\right\|_\infty\right) \left\|\gamma(t)\right\|_\infty, \label{e:mtest61}
\end{align}
for $t \geq 0$. Similarly, it holds
\begin{align}
\sup_{\xx \in \R} \left|\gamma_\xx\left(\psi_t^{-1}(\xx),t\right) - \gamma_\xx(\xx,t)\right| \leq \left\|\gamma_{\xx\xx}(t)\right\|_\infty \left\|\gamma(t)\right\|_\infty, \label{e:mtest62}
\end{align}
for $t \geq 0$. Thus, upon substituting $\psi_t^{-1}(\cdot)$ for $\cdot$ in~\eqref{e:mtest11} and~\eqref{e:mtest12}, and using estimates~\eqref{e:mtest2},~\eqref{e:mtest61} and~\eqref{e:mtest62}, we arrive at~\eqref{e:mtest6}.
\end{proof}

Upon comparing the two estimates in~\eqref{e:mtest6}, one notes that modulating the wavenumber of the wave train, in accordance with the phase modulation, leads to a sharper approximation result. In fact, it is natural that phase and wavenumber modulations are directly linked. Indeed, one readily observes that, for the modulated wave train $\phi_0(k_0x + \gamma_0(x))$, the local wavenumber, i.e.~the number of waves per unit interval near $x$, is $k_0 + \gamma_0'(x)$. Thus, as the phase modulation $\gamma(t)$ is approximated by a solution $\breve \gamma(t)$ to the viscous Hamilton-Jacobi equation~\eqref{e:HamJac}, cf.~estimate~\eqref{e:mtest3}, one finds that the associated wavenumber modulation is approximated by the solution $\breve k(t) = \breve \gamma_\xx(t)$ to the viscous Burgers' equation
\begin{align} \partial_t \breve k = d\partial_\xx^2 \breve k + a\partial_\xx \breve k + \nu \partial_\xx \big(\breve{k}^2\big). \label{e:Burgers}\end{align}

\begin{remark}
{\upshape Going back to the original $(x,t)$-variables in Theorem~\ref{t:mainresult} and Corollary~\ref{cor:psit}, we obtain that, for each $v_0 \in C_{\mathrm{ub}}(\R)$ with $E_0 = \|v_0\|_\infty < \epsilon$, there exists a unique classical global solution $u \in X$ to the reaction-diffusion system~\eqref{RD} with initial condition $u(0) = \phi_0 + v_0$ obeying the estimates
 \begin{align*} 
\begin{split}
\sup_{x \in \R} \left|u(x,t)-\phi_0\left(k_0 x - \omega_0 t\right)\right| &\leq M E_0,\\
\sup_{x \in \R} \left|u(x,t)-\phi_0\left(k_0 x - \omega_0 t + \mathring{\gamma}(x,t)\right)\right| &\leq \frac{M E_0}{\sqrt{1+t}},\\
\sup_{x \in \R} \left|u(x,t)-\phi\left(k_0 x - \omega_0 t + \mathring{\gamma}(x,t)\left(1+\tfrac1{k_0} \mathring{\gamma}_x(x,t)\right);k_0 + \mathring{\gamma}_x(x,t)\right)\right| &\leq \frac{M E_0 \log(2+t)}{1+t},
\end{split} 
\end{align*}
for $t \geq 0$, where the phase modulation $\mathring{\gamma} \in C^\infty\big([0,\infty) \times \R,\R\big)$ is given by $\mathring{\gamma}(x,t) = \gamma(k_0x - \omega_0 t,t)$. 
}\end{remark}


\subsection{Strategy of proof} \label{sec:outlineproof}

In order to establish nonlinear stability of the wave-train solution $\phi_0$ to~\eqref{RD2}, a naive approach would be to control the perturbation $\vt(t) = u(t) - \phi_0$ over time, which satisfies the parabolic semilinear equation
\begin{align}
\left(\partial_t - \El_0\right)\vt = \NT(\vt), \label{e:umodpert}
\end{align}
where the nonlinearity
\begin{align*}
\NT(\vt) &= f(\phi_0+\vt) - f(\phi_0) - f'(\phi_0) \vt,
\end{align*}
is quadratic in $\vt$. However, the bounds on $\re^{\El_0 t}$ are the same as those on the heat semigroup $\smash{\re^{\partial_x^2 t}}$, cf.~Proposition~\ref{prop:full}, and are therefore not strong enough to close the nonlinear argument through iterative estimates on the Duhamel formulation of~\eqref{e:umodpert}.\footnote{As a matter of fact, in the heat equation $\partial_t u = u_{xx} + u^2$ with quadratic nonlinearity, each nontrivial nonnegative solution  blows up in finite time~\cite{FUJI}.}

Instead, we isolate the most critical behavior, which arises through translational invariance of the wave train and is manifested by spectrum of $\El_0$ touching the origin, by introducing a spatio-temporal phase modulation $\gamma(t)$. As in previous works~\cite{JONZNL,JONZ} considering localized perturbations, we then aim to control the associated modulated perturbation $v(t)$, see~\eqref{e:defv}, which satisfies a quasilinear equation of the form
\begin{align} \label{e:pertbeq}
(\partial_t - \El_0)\left(v + \phi_0'\gamma - \gamma_\xx v\right) = N\left(v, v_\xx,v_{\xx\xx},\gamma_\xx, \partial_t \gamma,\gamma_{\xx\xx}, \gamma_{\xx\xx\xx}\right),
\end{align}
where $N$ is nonlinear in its variables. By decomposing the semigroup $\re^{\El_0 t}$ in a principal part of the form $\phi_0'S_p^0(t)$, where $S_p^0(t)$ decays diffusively, and a residual part exhibiting higher order temporal decay, the phase modulation $\gamma(t)$ in~\eqref{e:defv} can be chosen in such a way that it compensates for the most critical contributions in the Duhamel formulation of~\eqref{e:pertbeq}. We then expect, as in the case of localized perturbations~\cite{JONZW}, that the leading-order dynamics of the phase is governed by the viscous Hamilton-Jacobi equation~\eqref{e:HamJac}, where the coefficients are given by~\eqref{e:defad}. We stress that the nonlinearities in both~\eqref{e:HamJac} and~\eqref{e:pertbeq} only depend on \emph{derivatives} of the phase, whose leading-order dynamics is thus described by the viscous Burgers' equation~\eqref{e:Burgers}, obtained by differentiating~\eqref{e:HamJac}.

It is well-known that small, sufficiently localized initial data in the viscous Burgers' equation decay diffusively and perturbations by higher-order nonlinearities do not influence this decay, see for instance~\cite[Theorem~1]{UECN} or~\cite[Theorem~4]{BKL}. Thus, in the previous works~\cite{JONZNL,JONZ}, a nonlinear iteration scheme in the  variables $v$, $\gamma_\zeta$ and $\gamma_t$ could be closed. Here, one could allow for a nonlocalized initial phase modulation, cf.~\S\ref{sec:modnonlocal}, because only derivatives of the phase $\gamma$ enter in the nonlinear iteration and, thus, need to be localized. Furthermore, the loss of regularity arising in the quasilinear equation~\eqref{e:pertbeq} was addressed with the aid of $L^2$-energy estimates (so-called \emph{nonlinear damping estimates}). 

Our idea is to replace the semigroup decomposition and associated $L^1$-$H^k$-estimates in~\cite{JONZNL,JONZ} by a Green's function decomposition and associated pointwise bounds. These pointwise Green's function bounds, which have partly been established in~\cite{RS21,JUN} and are partly new, then yield pure $L^\infty$-estimates on the components of the semigroup $\re^{\El_0 t}$. Of course, the loss of localization leads to weaker decay rates, which complicates the nonlinear stability argument. Here, we take inspiration from~\cite{HDRS22}, where nonlinear stability against $C_{\mathrm{ub}}$-perturbations has been obtained in the special case of periodic roll solutions in the real Ginzburg-Landau equation by fully exploiting diffusive smoothing.\footnote{It seems more natural to follow the analysis in~\cite{HDRS22}, instead of~\cite{JONZNL,JONZ}. However, the analysis in~\cite{HDRS22} is tailored for the real Ginzburg-Landau equation. Indeed, by factoring out its gauge symmetry, periodic roll waves reduce to homogeneous steady states, yielding a linearization with constant coefficients, which can be diagonalized in Fourier space, so that diffusive modes can be explicitly separated from exponentially damped modes.} Still, we are confronted with various challenges. 

The first challenge is to control the dynamics of the phase $\gamma(t)$ and, more importantly, its derivative $\gamma_\zeta(t)$, which satisfy perturbed viscous Hamilton-Jacobi and Burgers' equations, respectively. In contrast to the case of localized initial data, the nonlinearities in~\eqref{e:HamJac} and~\eqref{e:Burgers} are decisive for the leading-order asymptotics of solutions with bounded initial data, see~\cite[Section~2.4]{SAN3} and Remark~\ref{rem:fronts}, and cannot be controlled through iterative estimates on the associated Duhamel forumlation.\footnote{We emphasize that this issue did not arise in~\cite{HDRS22}, since in the special case of periodic roll waves in the real Ginzburg-Landau equation the nonlinear dispersion relation vanishes identically due to reflection symmetry, cf.~\cite[Section~3.1]{DSSS}, rendering $\nu = 0$ in equations~\eqref{e:HamJac} and~\eqref{e:Burgers}.} We address this issue by removing the relevant nonlinear terms in the perturbed viscous Hamilton-Jacobi and Burgers' equations with the aid of the Cole-Hopf transform, resulting in an equation, which is a linear convective heat equation in the Cole-Hopf variable, but which is nonlinear in the residual variables. With the relevant nonlinear terms removed, iterative estimates on the associated Duhamel formulation are strong enough to control the Cole-Hopf variable over time and, thus, the phase $\gamma(t)$ and its derivatives. 

The second challenge is to address the loss of regularity experienced in the nonlinear iteration for the modulated perturbation, which satisfies the quasilinear equation~\eqref{e:pertbeq}. In contrast to previous works, the lack of localization prohibits the use of $L^2$-energy estimates to regain regularity. Instead, we proceed as in~\cite{RS21,HJPR21} by incorporating tame estimates on the unmodulated perturbation $\vt(t) = u(t) - \phi_0$ into the iteration scheme, which satisfies the semilinear equation~\eqref{e:umodpert} in which no derivatives are lost, yet where decay is too slow to close an independent iteration scheme.

\begin{remark} \label{rem:linfty}
{\upshape In our nonlinear stability analysis we use that the perturbation $\smash{\widetilde{v}(t)}$, which satisfies the parabolic semilinear equation~\eqref{e:umodpert}, maps continuously from its maximal interval of existence into $L^\infty(\R)$.\footnote{Indeed, we use in the proof of Theorem~\ref{t:mainresult} that the function $t \mapsto \sup_{0 \leq s \leq t} \|\vt(s)\|_\infty$ is continuous.} In particular, we require that $\smash{\widetilde{v}(t)}$ converges to its initial condition $v_0$ in $L^\infty$-norm as $t \downarrow 0$. Standard analytic semigroup theory for parabolic problems provides such convergence if and only if $v_0$ lies in the closure of the domain of the linearization, cf.~\cite[Theorem~7.1.2]{LUN}. Upon considering the linearization $\El_0$ as an operator on the maximal space $L^\infty(\R)$, the closure of its domain is given by the space $C_{\mathrm{ub}}(\R)$ of bounded, uniformly continuous functions by~\cite[Theorem~3.1.7]{LUN}. Hence, the regularity imposed on the initial condition $v_0$ in Theorem~\ref{t:mainresult} is the minimal one for (right-)continuity of the perturbation $\smash{\vt(t)}$ at $t = 0$ in $L^\infty(\R)$. 
}\end{remark}

\subsection{Outline}

In~\S\ref{sec:prelim} we collect some preliminary results on wave trains and their linear and nonlinear dispersion relations. Subsequently, we decompose the semigroup generated by the linearization of~\eqref{RD2} about the wave train and establish $L^\infty$-estimates on the respective components in~\S\ref{sec:decomp}. The iteration scheme in the variables $v(t), \gamma(t)$ and $\vt(t)$ for our nonlinear stability argument is presented in~\S\ref{sec:itscheme}. The proof of our main result, Theorem~\ref{t:mainresult}, can then be found in~\S\ref{sec:nonlinearstab}. The discussion of our main result, its embedding in the literature and related open problems are the contents of~\S\ref{sec:discussion}. Finally, Appendix~\ref{app:aux} is dedicated to some technical auxiliary result to establish pointwise Green's function estimates, whereas Appendix~\ref{app:B} contains the proof of the local existence result for the phase modulation $\gamma(t)$.

\paragraph*{Notation.} Let $S$ be a set, and let $A, B \colon S \to \R$. Throughout the paper, the expression ``$A(x) \lesssim B(x)$ for $x \in S$'', means that there exists a constant $C>0$, independent of $x$, such that $A(x) \leq CB(x)$ holds for all $x \in S$. 

\paragraph*{Acknowledgments.}  This project is funded by the Deutsche Forschungsgemeinschaft (DFG, German Research Foundation) -- Project-ID 491897824. The author would like to thank the reviewers for their insights and constructive comments, which have significantly improved the manuscript.

\paragraph*{Data availability statement.} Data sharing is not applicable to this article as no datasets were generated or analyzed during the current study.

\section{Preliminaries} \label{sec:prelim}

In this section we collect some basic results on wave trains and their dispersion relations, which are relevant for our nonlinear stability analysis. We refer to~\cite[Section~4]{DSSS} for a more extensive treatment. 

First, we note that a simple application of the implicit function theorem shows that wave-train solutions to~\eqref{RD} arise in families parameterized by the wavenumber, cf.~\cite[Section~4.2]{DSSS}.

\begin{proposition} \label{prop:family}
Assume~\ref{assH1} and~\ref{assD3}. Then, there exists a neighborhood $U \subset \R$ of $k_0$ and smooth functions $\phi \colon \R \times U \to \R$ and $\omega \colon U \to \R$ with $\phi(\cdot;k_0) = \phi_0$ and $\omega(k_0) = \omega_0$ such that
$$u_k(x,t) = \phi(k x - \omega(k) t;k),$$
is a wave-train solution to~\eqref{RD} of period $1$ for each wavenumber $k \in U$. By shifting the wave trains if necessary, we can arrange for~\eqref{e:gauge} to hold, where $\smash{\widetilde{\Phi}_0}$ is the eigenfunction of the adjoint Bloch operator $\El(0)^*$ satisfying~\eqref{e:adjoint}.
\end{proposition}

We recall that the function $\omega(k)$, established in Proposition~\ref{prop:family}, is the so-called nonlinear dispersion relation, describing the dependency of the frequency on the wavenumber. 

Since we assumed that $0$ is a simple eigenvalue of $\El(0)$, it follows by standard analytic perturbation theory that there exists an analytic curve $\lambda_c(\xi)$ with $\lambda_c(0) = 0$, such that $\lambda_c(\xi)$ is a simple eigenvalue of the Bloch operator $\El(\xi)$ for $\xi \in \R$ sufficiently close to $0$. The curve $\lambda_c(\xi)$ is the \emph{linear dispersion relation}. The eigenfunction $\Phi_\xi$ of $\El(\xi)$ associated with $\lambda_c(\xi)$ also depends analytically on $\xi$ and lies, by a standard bootstrapping argument, in $H^m_{\mathrm{per}}(0,1)$ for any $m \in \mathbb N_0$. Using Lyapunov-Schmidt reduction, the eigenvalue $\lambda_c(\xi)$, as well as the eigenfunction $\Phi_\xi$, can be expanded in $\xi$, cf.~\cite[Section~4.2]{DSSS} or~\cite[Section~2]{JONZW}.\footnote{We point out that there is a small difference in the coefficients in the expansions in~\cite{DSSS} compared to those used here and in~\cite{JONZW} due to the fact that the Bloch frequency variable $\xi$ is scaled by the wavenumber $k_0$ in~\cite{DSSS}. Moreover, in~\cite{JONZW} the diffusion matrix $D$ is taken to be the identity matrix.}  All in all, we establish the following result.

\begin{proposition} \label{prop:speccons}
Assume~\ref{assH1} and~\ref{assD2}-\ref{assD3}. Let $m \in \mathbb N_0$. Then, there exist a constant $\xi_0 \in (0,\pi)$ and an analytic curve $\lambda_c \colon (-\xi_0,\xi_0) \to \C$ satisfying 
\begin{itemize}
\setlength\itemsep{0em}
\item[(i)] The complex number $\lambda_c(\xi)$ is a simple eigenvalue of $\El(\xi)$ for any $\xi \in (-\xi_0,\xi_0)$. An associated eigenfunction $\Phi_\xi$ of $\El(\xi)$ lies in $H_{\mathrm{per}}^m(0,1)$, satisfies $\Phi_0 = \phi_0'$, is analytic in $\xi$ and fulfills
\begin{align*}
 \big\langle \widetilde{\Phi}_0,\Phi_\xi\big\rangle_{L^2(0,1)} = 1.
\end{align*}
\item[(ii)] The complex conjugate $\overline{\lambda_c(\xi)}$ is a simple eigenvalue of the adjoint $\El(\xi)^*$ for any $\xi \in (-\xi_0,\xi_0)$. An associated eigenfunction $\widetilde{\Phi}_\xi$ lies in $H_{\mathrm{per}}^m(0,1)$, is smooth in $\xi$ and satisfies
\begin{align*}
 \big\langle \widetilde{\Phi}_\xi,\Phi_\xi\big\rangle_{L^2(0,1)} = 1.
\end{align*}
\item[(iii)] The expansions
\begin{align} \label{e:eig_exp}
\left|\lambda_c(\xi) - \ri a\xi + d \xi^2\right| \lesssim |\xi|^3, \qquad \left\|\Phi_\xi - \phi_0' - \ri k_0 \xi \partial_k \phi(\cdot;k_0)\right\|_{H^m(0,1)} \lesssim |\xi|^2,
\end{align}
hold for $\xi \in (-\xi_0,\xi_0)$ with coefficients $a \in \R$ and $d > 0$ given by~\eqref{e:defad}.
\end{itemize}
\end{proposition}

\section{Semigroup decomposition and estimates} \label{sec:decomp}

The linearization $\El_0$ of~\eqref{RD2} is a densely defined, sectorial operator on $C_{\mathrm{ub}}(\R)$ with domain $D(\El_0) = C_{\mathrm{ub}}^2(\R)$, cf.~\cite[Corollary~3.1.9]{LUN}, and thus generates an analytic semigroup $\re^{\El_0 t}$, $t \geq 0$. In this section we establish bounds on the semigroup $\re^{\El_0 t}$ as an operator between $C_{\mathrm{ub}}$-spaces. We find that $\re^{\El_0 t}$ obeys the same bounds as the heat semigroup $\smash{\re^{\partial_x^2 t}}$, which, as outlined in~\S\ref{sec:outlineproof}, are not strong enough to close the nonlinear iteration. Therefore, we split off the most critical diffusive behavior by decomposing the semigroup, largely following~\cite{JONZW}, see also~\cite{JONZNL,JONZ}. That is, we first isolate the critical low-frequency modes, which correspond to the linear dispersion relation $\lambda_c(\xi), \xi \in (-\xi_0,\xi_0)$, cf.~Proposition~\ref{prop:speccons}. This yields a decomposition of the semigroup in a critical part and a residual part, which corresponds to the remaining exponentially damped modes and decays rapidly. We further decompose the critical part of the semigroup in a principal part, which obeys the same $L^\infty$-bounds as the heat semigroup $\smash{\re^{\partial_x^2 t}}$, and a residual part, exhibiting higher algebraic decay rates. Finally, in order to later expose the leading-order Hamilton-Jacobi dynamics of the phase variable, we relate the principal part to the convective heat semigroup $\smash{\re^{(d\partial_\xx^2 + a \partial_\xx)t}}$ using the expansion~\eqref{e:eig_exp} of the linear dispersion relation $\lambda_c(\xi)$.

In~\cite{JONZNL,JONZW,JONZ} the decomposition of the semigroup, there acting on $L^2$-localized functions, is carried out in Floquet-Bloch frequency domain using the representation
\begin{align*} \re^{\El_0 t}v(\xx) = \frac{1}{2\pi} \int_{-\pi}^\pi \re^{\ri\xx \xi} \re^{\El(\xi) t} \check{v}(\xi,\xx)\de \xi,
\end{align*}
where $\check{v}$ denotes the Floquet-Bloch transform of $v \in L^2(\R)$. Although it is possible to transfer the Floquet-Bloch transform to $L^\infty$-spaces by making use of tempered distributions, we avoid technicalities by realizing the decomposition on the level of the associated temporal Green's function. The Green's function is defined as the distribution $G(\xx,\xt,t) = \left[\re^{\El_0 t} \delta_{\xt}\right](\xx)$, where $\delta_\xt$ is the Dirac distribution centered at $\xt \in \R$. It is well-known that for elliptic differential operators, such as $\El_0$, the Green's function is an actual function, which is $C^2$ in its variables and exponentially localized in space, see for instance~\cite[Proposition~11.3]{ZUH}. The relevant decomposition of the Green's function and associated pointwise estimates have partly been established in~\cite{JUN}, see also~\cite{RS21}. The decomposition of the semigroup $\re^{\El_0 t}$ and corresponding $L^\infty$-estimates then follow readily by using the representation
\begin{align*} 
\re^{\El_0 t}v(\zeta)= \int_\R G(\xx,\xt,t) v(\xt) \de \xt, \qquad v \in C_{\mathrm{ub}}(\R),
\end{align*}
and employing $L^1$-$L^\infty$-convolution estimates.

\subsection{\texorpdfstring{$L^\infty$}{L\^infty}-bounds on the full semigroup} 

Before decomposing the semigroup $\re^{\El_0t}$ and establishing bounds on the respective components, we derive $L^\infty$-bounds on the full semigroup $\re^{\El_0 t}$. Such $L^\infty$-bounds readily follow from the pointwise Green's function estimates obtained in~\cite{JUN}, see also~\cite{RS21}. Although the bounds on the full semigroup are not strong enough to close the nonlinear iteration, they are employed in our analysis to control the unmodulated perturbation.

\begin{proposition} \label{prop:full}
Assume~\ref{assH1} and~\ref{assD1}-\ref{assD3}. Let $j,l \in \{0,1\}$ with $0 \leq j + l \leq 1$. Then, the semigroup generated by $\El_0$ enjoys the estimate
\begin{align*}
\left\|\partial_\xx^j \re^{\El_0 t} \partial_\xx^l v\right\|_{\infty} \lesssim \left(1 + t^{-\frac{l+j}{2}}\right) \|v\|_{\infty},
\end{align*}
for $v \in C_{\mathrm{ub}}^l(\R)$ and $t > 0$.
\end{proposition}
\begin{proof}
By~\cite[Theorem~1.3]{JUN}, see also~\cite[Theorem~3.4]{RS21} for the planar case, there exists a constant $M_0 > 1$ such that the Green's function enjoys the pointwise estimate
\begin{align*}
\left|\partial_\xx^j \partial_\xt^l G(\xx,\xt,t)\right| \lesssim \left(1+t^{-\frac{j+l}{2}}\right)t^{-\frac{1}{2}} \re^{-\frac{(\xx-\xt+at)^2}{M_0t}},
\end{align*}
for $\xx,\xt \in \R$ and $t > 0$. Hence, integration by parts yields
\begin{align*} 
\begin{split}
\left\|\partial_\xx^j \re^{\El_0 t} \partial_\xx^l v\right\|_\infty &= \left\|\int_\R \partial_\xx^j \partial_\xt^l G(\cdot,\xt,t) v(\xt) \de \xt\right\|_\infty \lesssim \left(1+t^{-\frac{j+l}{2}}\right)\|v\|_\infty \int_\R \frac{\re^{-\frac{z^2}{M_0t}}}{\sqrt{t}} \de z\\ 
&\lesssim \left(1+t^{-\frac{j+l}{2}}\right) \|v\|_\infty, 
\end{split}
\end{align*}
for $v \in C_{\mathrm{ub}}^l(\R)$ and $t > 0$. 
\end{proof}

\subsection{Isolating the critical low frequency modes} 

Following~\cite{JUN}, we decompose the temporal Green's function $G(\xx,\xt,t)$ by splitting off the critical low-frequency modes for large times. Thus, we introduce the smooth cutoff functions $\rho \colon \R \to \R$ and $\chi \colon [0,\infty) \to \R$ satisfying $\rho(\xi)=1$ for $|\xi|<\frac{\xi_0}{2}$, $\rho(\xi)=0$ for $|\xi|>\xi_0$, $\chi(t) = 0$ for $t \in [0,1]$ and $\chi(t) = 1$ for $t \in [2,\infty)$, where $\xi_0 \in (0,\pi)$ is as in Proposition~\ref{prop:speccons}. The Green's function decomposition established in~\cite{JUN} reads
\begin{align*}
G(\xx,\xt,t) = G_c(\xx,\xt,t) + G_e(\xx,\xt,t), 
\end{align*}
for $\xx,\xt \in \R$ and $t \geq 0$, where 
\begin{align*}
G_c(\xx,\xt,t) &= \frac{\chi(t)}{2\pi} \int_{-\pi}^{\pi} \rho(\xi) \re^{\ri\xi(\xx - \xt)} \re^{\lambda_c(\xi) t} \Phi_\xi(\xx) \widetilde{\Phi}_\xi(\xt)^* \de \xi,
\end{align*}
represents the component of the Green's function associated with the linear dispersion relation $\lambda_c(\xi)$, cf.~Proposition~\ref{prop:speccons}. The corresponding decomposition of the semigroup is given by
\begin{align*} \re^{\El_0 t} = S_e(t) + S_c(t),\end{align*}
where the propagators $S_e(t)$ and $S_c(t)$ are defined as
$$S_e(t)v(\zeta) = \int_\R G_e(\xx,\xt,t) v(\xt) \de \xt, \qquad S_c(t)v(\zeta) = \int_\R G_c(\xx,\xt,t) v(\xt) \de \xt,$$
for $v \in C_{\mathrm{ub}}(\R)$ and $t \geq 0$. The pointwise Green's function estimates obtained in~\cite{JUN}, see also~\cite{RS21}, in combination with Lemma~\ref{lem:semigroupEstimate1} readily imply that the residual $S_e(t)$ decays rapidly.

\begin{proposition} \label{prop:semexp}
Assume~\ref{assH1} and~\ref{assD1}-\ref{assD3}. Let $j, l \in \{0,1\}$ with $j + l \leq 1$. Then, it holds
\begin{align*}
\left\|\partial_\xx^j S_e(t) \partial_\xx^l v\right\|_\infty &\lesssim (1+t)^{-\frac{11}{10}} \left(1 + t^{-\frac{j+l}{2}}\right) \|v\|_\infty, 
\end{align*}
for $v \in C_{\mathrm{ub}}^l(\R)$ and $t > 0$.
\end{proposition}
\begin{proof}
In case $t \in (0,1]$ we have $\chi(t) = 0$, which implies $S_e(t) = \re^{\El_0 t}$. Hence, for $t \in (0,1]$ the result directly follows from Proposition~\ref{prop:full}. Thus, all that remains is to derive the desired estimate for large times. By the analysis in~\cite[Section~4]{JUN}, see also~\cite[Lemma~A.4]{RS21}, there exist constants $\mu_1 > 0$ and $M_0, M_1 > 1$ such that the Green's function enjoys the exponential estimate
\begin{align*}
\left|\partial_\xx^j \partial_\xt^l G(\xx,\xt,t)\right| \lesssim t^{-\frac{1+j+l}{2}} \re^{-\mu_1 t} \re^{-\frac{(\xx-\xt+at)^2}{M_1t}},
\end{align*}
for any $\xx,\xt \in \R$ and $t > 0$ satisfying
\begin{align} |\xx - \xt + at| \geq M_0t. \label{eq:large}\end{align}
In addition, Proposition~\ref{prop:speccons} and Lemma~\ref{lem:semigroupEstimate1} yield
\begin{align*} \left|\partial_\xx^j \partial_\xt^l G_c(\xx,\xt,t)\right| &= \left|\frac{\chi(t)}{2\pi} \int_\R \rho(\xi) \re^{\lambda_c(\xi) t}  \partial_\xx^j \left(\re^{\ri\xi \xx} \Phi_\xi(\xx)\right) \partial_\xt^l \left(\re^{-\ri \xi \xt} \widetilde{\Phi}_\xi(\xt)^*\right)  \de \xi\right|\\
 &\lesssim t^{-\frac{1}{2}}\left(1 + \frac{(\xx-\xt + a t)^4}{t^2}\right)^{-1} \lesssim t^{-\frac{1}{2}}\left(1 + M_0^{\frac{11}{4}} t^{\frac{3}{4}} |\xx-\xt + a t|^{\frac{5}{4}}\right)^{-1},
\end{align*}
for $\xx,\xt \in \R$ and $t \geq 1$ satisfying~\eqref{eq:large}. Finally, by the analysis in~\cite[Section~4]{JUN}, see also the proof of~\cite[Lemma~A.5]{RS21}, there exists a constant $M_2 > 0$ such that we have the exponential estimate
\begin{align*}
\left|\partial_\xx^j \partial_\xt^l G_e(\xx,\xt,t)\right| \lesssim t^{-\frac{1+j+l}{2}} \re^{-\mu_2 t} \re^{-\frac{(\xx-\xt+at)^2}{M_2t}},
\end{align*}
for any $\xx,\xt \in \R$ and $t > 0$ satisfying
\begin{align*} |\xx - \xt + at| \leq M_0 t.\end{align*}
Combining the three pointwise Green's function estimates established above, we use integration by parts to arrive at
\begin{align*} \left\|\partial_\xx^j S_e(t) \partial_\xx^l v\right\|_\infty &\leq \sup_{\xx \in \R} \left(\left|\int_{|\xx-\xt+at| \leq M_0t}  \partial_\xx^j \partial_\xt^l G_e(\xx,\xt,t) v(\xt) \de \xt\right| + \left|\int_{|\xx-\xt+at| \geq M_0t} \partial_\xx^j \partial_\xt^l G(\xx,\xt,t) v(\xt) \de \xt\right|\right.\\
&\qquad\qquad\qquad \left. + \,\left|\int_{|\xx-\xt+at| \geq M_0t} \partial_\xx^j \partial_\xt^l G_c(\xx,\xt,t) v(\xt) \de \xt\right|\right)\\
&\lesssim t^{-\frac{j+l}{2}} \|v\|_\infty\sum_{i = 1}^2 \re^{-\mu_i t} \int_\R \frac{\re^{-\frac{z^2}{M_it}}}{\sqrt{t}} \de z + t^{-\frac{1}{2}} \|v\|_\infty \int_\R \left(1 + M_0^{\frac{11}{4}} t^{\frac{3}{4}} |z|^{\frac{5}{4}}\right)^{-1} \de z \\
&\lesssim t^{-\frac{11}{10}} \left(1 + t^{-\frac{j+l}{2}}\right) \|v\|_\infty, 
\end{align*}
for $t \geq 1$ and $v \in C_{\mathrm{ub}}^l(\R)$.
\end{proof}

\begin{remark} 
{\upshape
The modes corresponding to the residual part $S_e(t)$ of the semigroup $\re^{\El_0t}$ are exponentially damped, i.e.~there exists $\delta_0 > 0$ such that $\sigma(\El_0) \setminus \{\lambda_c(\xi) : \xi \in (-\xi_0,\xi_0)\}$ is contained in the half plane $\{\lambda \in \C : \Re \, \lambda < -\delta_0\}$. Therefore, we expect that one can replace the algebraic $L^\infty$-bound in Proposition~\ref{prop:semexp} by an exponential one, see~\cite[Proposition~3.1]{JONZ} for the corresponding $L^2$-result. This would require a sharper pointwise estimate on $\partial_\xx^j \partial_\xt^l G_c(\xx,\xt,t)$ in the proof of Proposition~\ref{prop:semexp}, which could possibly be obtained by exploiting analyticity of the integrand in $\xi$ and deforming contours. However, since the algebraic bound on $S_e(t)$ in Proposition~\ref{prop:semexp} is sufficient for our purposes, we refrain from doing so in an effort to avoid unnecessary technicalities. 
}\end{remark}

\subsection{Decomposing the critical component}

Motivated by the expansion~\eqref{e:eig_exp} of the eigenfunction $\Phi_\zeta$ of the Bloch operator $\El(\xi)$, we further decompose the critical component of the Green's function as
\begin{align*}
G_c(\xx,\xt,t) = \left(\phi_0'(\xx) + k_0 \partial_k \phi(\xx;k_0) \partial_\xx \right)G_p^0(\xx,\xt,t) + G_r(\xx,\xt,t),
\end{align*}
for $\xx,\xt \in \R$ and $t \geq 0$, where
\begin{align}
\label{e:defGP}
G_p^0(\xx,\xt,t) &= \frac{\chi(t)}{2\pi} \int_{-\pi}^{\pi} \rho(\xi) \re^{\ri\xi(\xx - \xt)} \re^{\lambda_c(\xi) t} \widetilde{\Phi}_\xi(\xt)^* \de \xi,
\end{align}
represents the principal component, and where
\begin{align*}
G_r(\xx,\xt,t) &= \frac{\chi(t)}{2\pi} \int_{-\pi}^{\pi} \rho(\xi) \re^{\ri\xi(\xx - \xt)} \re^{\lambda_c(\xi) t} \left(\Phi_\xi(\xx) - \phi_0'(\xx) - \ri k_0 \xi \partial_k \phi(\xx;k_0)\right) \widetilde{\Phi}_\xi(\xt)^* \de \xi,
\end{align*}
is a remainder term. Defining corresponding propagators
\begin{align*}
S_p^0(t)v(\zeta) = \int_\R G_p^0(\xx,\xt,t) v(\xt) \de \xt, \qquad S_r(t)v(\zeta) = \int_\R G_r(\xx,\xt,t) v(\xt) \de \xt,
\end{align*}
yields the semigroup decomposition
\begin{align} \label{e:semidecomp}
\re^{\El_0 t} =  S_e(t) + S_c(t) = \widetilde{S}(t) + \left(\phi_0' + k_0 \partial_k \phi(\cdot;k_0) \partial_\xx \right)S_p^0(t),
\end{align}
where we denote
\begin{align} \label{e:semdecomp}
\widetilde{S}(t) := S_r(t) + S_e(t),
\end{align}
for $t \geq 0$. The reason for explicitly factoring out the term $\phi_0' + k_0 \partial_k \phi(\cdot;k_0) \partial_\xx$ in~\eqref{e:semidecomp} is that composition from the left of $S_p^0(t)$ with spatial and temporal derivatives yields additional temporal decay. Yet, the same does not hold for composition with derivatives from the right, as can be seen by computing the commutators
\begin{align} \label{e:commu} 
\begin{split}
\partial_\xx S_p^0(t) - S_p^0(t)\partial_\xx &= S_p^1(t), \qquad
\partial_\xx^2 S_p^0(t) - S_p^0(t)\partial_\xx^2
= 2\partial_\xx S_p^1(t) - S_p^2(t),
\end{split}
\end{align}
where we denote
\begin{align*}
S_p^i(t)v(\zeta) = \int_\R G_p^i(\xx,\xt,t) v(\xt) \de \xt, \qquad G_p^i(\xx,\xt,t) &= \frac{\chi(t)}{2\pi} \int_{-\pi}^{\pi} \rho(\xi) \re^{\ri\xi(\xx - \xt)} \re^{\lambda_c(\xi) t} \partial_\xt^i \widetilde{\Phi}_\xi(\xt)^* \de \xi,
\end{align*}
for $i = 0,1,2$. These facts, as well as the higher algebraic decay rates of the remainder $S_r(t)$, are confirmed by the following result.

\begin{proposition} \label{prop:semdif}
Assume~\ref{assH1} and~\ref{assD1}-\ref{assD3}. Let $j,l,m \in \mathbb N_0$ and $i \in \{0,1,2\}$. Then, we have the estimates
 \begin{align}
\label{e:semibprin}
\left\|\left(\partial_t - a\partial_\zeta\right)^j \partial_\zeta^l S_p^i(t) \partial_\zeta^m v\right\|_\infty &\lesssim (1+t)^{-\frac{2j+l}{2}} \|v\|_\infty,\\
\label{e:semibrem}
\left\|\partial_\zeta^l S_r(t)\partial_\zeta^m v\right\|_\infty &\lesssim (1+t)^{-1}\|v\|_\infty,
\end{align}
for $v \in C_{\mathrm{ub}}^m(\R)$ and $t \geq 0$.
\end{proposition}
\begin{proof}
First, we compute
\begin{align*}
\left(\partial_t - a\partial_\zeta\right)^j \partial_\zeta^l \partial_\xt^m G_p^i(\xx,\xt,t) = \frac{\chi(t)}{2\pi} \int_{-\pi}^{\pi} \rho(\xi) \left(\ri \xi\right)^l \left(\lambda_c(\xi) - \ri a\xi\right)^j \re^{\lambda_c(\xi) t} \partial_\xt^m \left(\re^{\ri\xi(\xx - \xt)} \partial_\xt^i \widetilde{\Phi}_\xi(\xt)^*\right) \de \xi,
\end{align*}
for $\xx,\xt \in \R$ and $t \geq 2$, where we recall $\chi(t) = 1$ for $t \in [2,\infty)$. Next, we note that by Proposition~\ref{prop:speccons} and the embedding $H^1_{\mathrm{per}}(\R) \hookrightarrow L^\infty(\R)$, it holds
\begin{align*} \left|\lambda_c(\xi) - \ri a \xi\right| \lesssim \xi^2, \qquad \left\|\partial_\xx^l \left(\Phi_\xi - \phi_0' - \ri k_0 \xi \partial_k \phi(\cdot;k_0)\right)\right\|_\infty \lesssim \xi^2,\end{align*}
for $\xi \in (-\xi_0,\xi_0)$. Hence, Proposition~\ref{prop:speccons} and Lemma~\ref{lem:semigroupEstimate1} imply that the following pointwise estimates
\begin{align*}
\left|\left(\partial_t - a\partial_\zeta\right)^j \partial_\zeta^l \partial_\xt^m G_p^i(\xx,\xt,t)\right| &\lesssim t^{-\frac{1+2j+l}{2}}\left(1 + \frac{(\xx-\xt + a t)^4}{t^2}\right)^{-1}, & & t \in [2,\infty),\\
\left|\left(\partial_t - a\partial_\zeta\right)^j \partial_\zeta^l \partial_\xt^m G_p^i(\xx,\xt,t)\right| &\lesssim t^{-\frac{1}{2}}\left(1 + \frac{(\xx-\xt + a t)^4}{t^2}\right)^{-1}, & & t \in [1,2],\\
\left|\partial_\zeta^l \partial_\xt^m G_r(\xx,\xt,t)\right| &\lesssim t^{-\frac{3}{2}}\left(1 + \frac{(\xx-\xt + a t)^4}{t^2}\right)^{-1}, & & t \in [1,\infty),
\end{align*}
hold for $\xx,\xt \in \R$. Using the first pointwise estimate above, we integrate by parts to arrive at
\begin{align*}
\left\|\left(\partial_t - a\partial_\zeta\right)^j \partial_\zeta^l S_p^i(t) \partial_\zeta^m v\right\|_\infty &= \left\|\int_\R \left(\partial_t - a\partial_\zeta\right)^j \partial_\xx^l \partial_\xt^m G_p^i(\cdot,\xt,t) v(\xt) \de \xt\right\|_\infty\\ 
&\lesssim t^{-\frac{1+2j+l}{2}} \|v\|_\infty \int_\R \left(1+\frac{z^4}{t^2}\right)^{-1} \de z \lesssim t^{-\frac{2j+l}{2}} \|v\|_\infty, 
\end{align*}
which yields~\eqref{e:semibprin} for $t \geq 2$. The estimates~\eqref{e:semibprin} for $t \in [1,2]$ and~\eqref{e:semibrem} for $t \geq 1$ follow analogously by integrating the second and third pointwise Green's function estimate, respectively. Finally, recalling that $\chi(t)$ vanishes on $[0,1]$, the inequalities~\eqref{e:semibprin} and~\eqref{e:semibrem} are trivially satisfied for $t \in [0,1]$.
\end{proof}

\subsection{Estimates on the full semigroup acting between \texorpdfstring{$C_{\mathrm{ub}}^j(\R)$}{C\_ub\^j(R)}-spaces}

To control higher derivatives of the unmodulated perturbation in our nonlinear analysis we need estimates on the semigroup $\re^{\El_0 t}$, as an operator from the space $C_{\mathrm{ub}}^{2j}(\R)$ into $C_{\mathrm{ub}}^{2j-l}(\R)$ for $j,l \in \mathbb N_0$ with $0 \leq l \leq 2j$. Such estimates readily follow from Propositions~\ref{prop:full},~\ref{prop:semexp} and~\ref{prop:semdif} by applying standard analytic semigroup theory. 

\begin{corollary} \label{cor:der}
Assume~\ref{assH1} and~\ref{assD1}-\ref{assD3}. Let $j \in \mathbb N$ and $m \in \mathbb N_0$. It holds
\begin{align*}
\left\|\re^{\El_0 t} v\right\|_{W^{2j,\infty}} &\lesssim \left(1+t^{-1}\right)\|v\|_{W^{2j-2,\infty}}, & & v \in C_{\mathrm{ub}}^{2j-2}(\R),\\
\left\|\re^{\El_0 t} \partial_\xx^m v\right\|_{W^{2j,\infty}} &\lesssim  \left(1+t\right)^{-\frac{11}{10}} \left(1+t^{-\frac{1}{2}}\right)\|v\|_{W^{2j+m-1,\infty}} + \|v\|_\infty, & & v \in C_{\mathrm{ub}}^{2j+m-1}(\R),\\
\left\|\re^{\El_0 t} v\right\|_{W^{2j,\infty}} &\lesssim \|v\|_{W^{2j,\infty}}, & & v \in C_{\mathrm{ub}}^{2j}(\R),
\end{align*}
for $t > 0$.
\end{corollary}
\begin{proof}
Standard analytic semigroup theory, cf.~\cite[Proposition~2.2.1]{LUN}, yields a constant $\mu \in \R$ such that
\begin{align} \left\|\El_0 \re^{\El_0 t} v\right\|_\infty \lesssim t^{-1}\re^{\mu t}\|v\|_\infty, \label{eq:sem1}\end{align}
for $t > 0$ and $v \in C_{\mathrm{ub}}(\R)$. Noting that the graph norm on the space $D(\El_0^j) = C_{\mathrm{ub}}^{2j}(\R)$ is equivalent to its $W^{2j,\infty}$-norm, we write 
\begin{align*} \left\|\re^{\El_0 t} v\right\|_{W^{2j,\infty}} \lesssim \left\|\El_0^j \re^{\El_0 t} v\right\|_{\infty} + \left\|\re^{\El_0 t} v\right\|_{\infty} = \left\|\El_0 \re^{\El_0 t} \El_0^{j-1} v\right\|_{\infty} + \left\|\re^{\El_0 t} v\right\|_{\infty},
\end{align*}
for $v \in C_{\mathrm{ub}}^{2j-2}(\R)$ and $t > 0$, where we use that $\El_0$ and $\re^{\El_0 t}$ commute. We use the semigroup property, estimate~\eqref{eq:sem1} and Proposition~\ref{prop:full} to
estimate
\begin{align*}
\left\|\El_0 \re^{\El_0 t} \El_0^{j-1} v\right\|_{\infty} &\lesssim t^{-1} \left\|\El_0^{j-1} v\right\|_{\infty},  & & t \in (0,1],\\
\left\|\El_0 \re^{\El_0 t} \El_0^{j-1} v\right\|_{\infty} &= \left\|\El_0 \re^{\El_0} \re^{\El_0(t-1)} \El_0^{j-1} v\right\|_{\infty}  \lesssim t^{-1} \left\|\re^{\El_0(t-1)} \El_0^{j-1} v\right\|_{\infty} \lesssim t^{-1} \left\|\El_0^{j-1} v\right\|_{\infty}, & & t > 1,
\end{align*}
for $v \in C_{\mathrm{ub}}^{2j-2}(\R)$. So, combining the last estimates and applying Proposition~\ref{prop:full}, we arrive at 
\begin{align*} 
\left\|\re^{\El_0 t} v\right\|_{W^{2j,\infty}} \lesssim
\left(1+t^{-1}\right)\left( \left\|\El_0^{j-1} v\right\|_{\infty} + \|v\|_\infty\right) \lesssim \left(1+t^{-1}\right)\|v\|_{W^{2j-2,\infty}},
\end{align*}
for $t > 0$ and $v \in C_{\mathrm{ub}}^{2j-2}(\R)$, which yields the first result.

For the second estimate we apply Propositions~\ref{prop:full},~\ref{prop:semexp} and~\ref{prop:semdif}, recall the decompositions~\eqref{e:semidecomp} and~\eqref{e:semdecomp} and use the equivalence of norms on $D(\El_0^j) = C_{\mathrm{ub}}^{2j}(\R)$ to bound
\begin{align*}
&\left\|\re^{\El_0 t} \partial_\xx^m v\right\|_{W^{2j,\infty}} \lesssim \left\|\re^{\El_0 t} \El_0^j \partial_\xx^m v\right\|_{\infty} + \left\|\re^{\El_0 t} \partial_\xx^m v\right\|_{\infty} \\ &\ \quad \leq \left\|\re^{\El_0 t} \partial_\xx k_0^2 D\partial_\xx \El_0^{j-1} \partial_\xx^mv\right\|_\infty + \left\|\re^{\El_0 t} \left(k_0^2D\partial_\xx^2 - \El_0\right)\El_0^{j-1} \partial_\xx^mv\right\|_{\infty} + \left\|\re^{\El_0 t} \partial_\xx^mv\right\|_{\infty}\\ &\ \quad \leq
\left\|S_e(t) \partial_\xx k_0^2 D\partial_\xx \El_0^{j-1} \partial_\xx^mv\right\|_\infty + \left\|S_e(t) \left(k_0^2D\partial_\xx^2 - \El_0\right)\El_0^{j-1} \partial_\xx^mv\right\|_{\infty} + \left\|S_e(t) \partial_\xx^m v\right\|_{\infty} \\
&\qquad \ \quad + \, \left\|S_c(t) \partial_\xx k_0^2 D\partial_\xx \El_0^{j-1} \partial_\xx^m v\right\|_\infty + \left\|S_c(t) \left(k_0^2D\partial_\xx^2 - \El_0\right)\El_0^{j-1} \partial_\xx^mv\right\|_{\infty} + \left\|S_c(t) \partial_\xx^m v\right\|_{\infty}\\
&\quad \ \lesssim \left(1+t\right)^{-\frac{11}{10}} \left(\left(1+t^{-\frac{1}{2}}\right) \left\|D\partial_\xx \El_0^{j-1} \partial_\xx^m v\right\|_\infty + \left\|\left(k_0^2D\partial_\xx^2 - \El_0\right)\El_0^{j-1} \partial_\xx^m v\right\|_{\infty} + \left\|\partial_\xx^m v\right\|_\infty\right) + \left\|v\right\|_{\infty} \\
&\ \quad \lesssim \left(1+t\right)^{-\frac{11}{10}} \left(1+t^{-\frac{1}{2}}\right)\|v\|_{W^{2j+m-1,\infty}} + \|v\|_\infty,
\end{align*}
for $t > 0$ and $v \in \smash{C_{\mathrm{ub}}^{2j+m}(\R)}$. Since the left and right hand side of the last chain of inequalities only requires $v \in \smash{C_{\mathrm{ub}}^{2j+m-1}(\R)}$, the second estimate follows immediately by a density argument.

Finally, for the last estimate we use the equivalence of norms on $\smash{D(\El_0^j) = C_{\mathrm{ub}}^{2j}(\R)}$ and Proposition~\ref{prop:full} to bound 
\begin{align*}
&\left\|\re^{\El_0 t} v\right\|_{W^{2j,\infty}} \lesssim \left\|\re^{\El_0 t} \El_0^j v\right\|_{\infty} + \left\|\re^{\El_0 t} v\right\|_{\infty} \lesssim \left\|\El_0^j v\right\|_\infty + \|v\|_\infty \lesssim \|v\|_{W^{2j,\infty}}
\end{align*}
for $t > 0$ and $v \in C_{\mathrm{ub}}^{2j}(\R)$.
\end{proof}

\subsection{Relation to the convective heat semigroup} \label{sec:decompprin}

We wish to relate the principal component $G_p^i(\xx,\xt,t)$ to the temporal Green's function corresponding to the convective heat equation $\partial_t u = du_{\xx\xx} + au_\xx$. Thus, we approximate $\lambda_c(\xi)$ by $\ri a\xi-d\xi^2$ and factor out the adjoint eigenfunction $
\smash{\widetilde{\Phi}_\xi(\xt)}$ in~\eqref{e:defGP} by approximating it by $\smash{\widetilde{\Phi}_0(\xt)}$ so that it no longer depends on variable $\xi$ and can be pulled out of the integral.  All in all, we establish the decomposition
\begin{align*}
G_p^i(\xx,\xt,t) = H(\xx-\xt,t)\partial_\xt^i \widetilde{\Phi}_0(\xt)^* + \widetilde{G}_r^i(\xx,\xt,t),
\end{align*}
for $\xx,\xt \in \R$, $t > 0$ and $i = 0,1,2$, where $H(\xx,t)$ is the convective heat kernel
\begin{align*}
H(\xx,t) &= \frac{1}{2\pi} \int_\R \re^{\ri\xi\xx + \left(\ri a\xi - d \xi^2\right)t} \de \xi = \frac{\re^{-\frac{|\xx+a t|^2}{4d t}}}{\sqrt{4 \pi d t}}, 
\end{align*}
and the remainder $\widetilde{G}_r^i(\xx,\xt,t)$ is given by
\begin{align*}
\widetilde{G}_r^i(\xx,\xt,t) &= \frac{1}{2\pi} \int_{\R} \re^{\ri\xi(\xx - \xt)+ \left(\ri a\xi - d \xi^2\right)t} \left(\chi(t)\rho(\xi)\re^{\lambda_r(\xi) t}  \partial_\xt^i \widetilde{\Phi}_\xi(\xt)^* - \partial_\xt^i \widetilde{\Phi}_0(\xt)^*\right) \de \xi,
\end{align*}
with $\lambda_r(\xi) := \lambda_c(\xi) - \ri a\xi + d \xi^2$. We introduce the corresponding propagators
\begin{align} \label{e:Shexpr}\widetilde{S}_r^i(t)v(\zeta) = \int_\R \widetilde{G}_r^i(\xx,\xt,t) v(\xt) \de \xt, \qquad S_h^i(t)v(\zeta) = \int_\R H(\xx-\xt,t) \partial_\xt^i \widetilde{\Phi}_0(\xt)^* v(\xt) \de \xt,\end{align}
and obtain the decomposition
\begin{align} \label{e:prindecomp}
S_p^i(t) = S_h^i(t) + \widetilde{S}_r^i(t), \qquad t \geq 0,
\end{align}
of the principal component. We prove that $\widetilde{S}_r^i(t)$ indeed exhibits higher algebraic decay rates, whereas $S_h^i(t)$ enjoys the same bounds as the convective heat semigroup 
\begin{align}
\re^{\left(d\partial_\xx^2 + a\partial_\xx\right)t} v(\xx) = \int_\R H(\xx-\xt,t)v(\xt)\de \xt, \qquad v \in C_{\mathrm{ub}}(\R), \, t > 0. \label{e:convheatsem}
\end{align}

\begin{proposition} \label{prop:semiprinrec}
Assume~\ref{assH1} and~\ref{assD1}-\ref{assD3}. Let $l \in \mathbb N_0$ and $i,m \in \{0,1,2\}$. Then, the estimates
\begin{align} \label{e:semiheat}
\left\|\partial_\zeta^l S_h^i(t)v\right\|_{\infty} \lesssim t^{-\frac{l}{2}}\|v\|_{\infty}, \quad & \quad  \left\|\partial_\zeta^l \re^{\left(d\partial_\xx^2 + a\partial_\xx\right) t}v\right\|_{\infty} \lesssim t^{-\frac{l}{2}}\|v\|_{\infty},\\
\left\|\partial_\zeta^m \widetilde{S}_r^i(t)v\right\|_\infty &\lesssim (1+t)^{-\frac{1}{2}}t^{-\frac{m}{2}}\|v\|_{\infty}, \label{e:semibrem2}
\end{align}
are satisfied for $v \in C_{\mathrm{ub}}(\R)$ and $t > 0$. Moreover, we have
\begin{align} \label{e:semiheat2}
\left\|\partial_\zeta S_h^i(t)v\right\|_{\infty} \lesssim (1+t)^{-\frac{1}{2}}\|v\|_{W^{1,\infty}}, \quad & \quad \left\|\partial_\zeta \re^{\left(d\partial_\xx^2 + a\partial_\xx\right) t}v\right\|_{\infty} \lesssim (1+t)^{-\frac{1}{2}}\|v\|_{W^{1,\infty}},\\
\left\|\partial_\zeta \widetilde{S}_r^i(t)v\right\|_\infty &\lesssim (1+t)^{-1}\|v\|_{W^{1,\infty}}, \label{e:semibrem3}
\end{align}
for $v \in C_{\mathrm{ub}}^1(\R)$ and $t \geq 0$.
\end{proposition}
\begin{proof}
First, we observe that~\eqref{e:semiheat} follows from a standard application of Young's convolution inequality. Indeed, one bounds
\begin{align*}
\left\|\partial_\xx^l S_h^i(t) v\right\|_\infty, \left\|\partial_\zeta^l \re^{\left(d\partial_\xx^2 + a\partial_\xx\right) t}v\right\|_\infty \lesssim \left\|\partial_\xx^l H(\cdot,t)\right\|_1 \left\|v\right\|_\infty \lesssim t^{-\frac{l}{2}} \|v\|_\infty,
\end{align*}
for $v \in C_{\mathrm{ub}}(\R)$ and $t > 0$. Integrating by parts, we obtain similarly
\begin{align*}
\left\|\partial_\xx S_h^i(t) v\right\|_\infty, \left\|\partial_\xx \re^{\left(d\partial_\xx^2 + a\partial_\xx\right) t}v\right\|_\infty \lesssim \left\|H(\cdot,t)\right\|_1 \left\|v\right\|_{W^{1,\infty}} \lesssim \|v\|_{W^{1,\infty}},
\end{align*}
for $v \in C_{\mathrm{ub}}^1(\R)$ and $t \geq 0$, which together with~\eqref{e:semiheat} implies~\eqref{e:semiheat2}.

Next, we prove~\eqref{e:semibrem2} and~\eqref{e:semibrem3}. We recall that $\chi(t)$ vanishes on $[0,1]$. Hence, we have $\smash{\widetilde{S}_r^i(t)} = -S_h^i(t)$ for $t \in (0,1]$. Therefore, estimates~\eqref{e:semiheat} and~\eqref{e:semiheat2} imply~\eqref{e:semibrem2} and~\eqref{e:semibrem3} for short times $t \in (0,1]$. So, all that remains is to establish~\eqref{e:semibrem2} and~\eqref{e:semibrem3} for large times $t \geq 1$. First, we note that for all $z \in \mathbb C$ we have 
\begin{align*} \left|\re^{z} - 1\right| \leq \re^{|z|} - 1 \leq |z|\re^{|z|}.
\end{align*}
Moreover, Proposition~\ref{prop:speccons} yields a constant $C > 0$ such that $|\lambda_r(\xi)| \leq C |\xi|^3$ for $\xi \in (-\xi_0,\xi_0)$. Hence, applying Proposition~\ref{prop:speccons}, using the embedding $H^1_{\mathrm{per}}(\R) \hookrightarrow L^\infty(\R)$ and recalling $\chi(t) = 1$ for $t \geq 2$, we arrive at
\begin{align} \begin{split} \left\|\chi(t)\re^{\lambda_r(\xi) t} \partial_\xx^i \widetilde{\Phi}_\xi^* - \partial_\xx^i \widetilde{\Phi}_0^*\right\|_\infty &\leq \left|\re^{\lambda_r(\xi) t} - 1\right| \left\|\partial_\xx^i \widetilde{\Phi}_\xi^*\right\|_\infty + \left\|\partial_\xx^i \left(\widetilde{\Phi}_\xi^* - \widetilde{\Phi}_0^*\right)\right\|_\infty\\ & \lesssim \left|\lambda_r(\xi)\right| t \re^{|\lambda_r(\xi)| t} + |\xi| \lesssim \left(|\xi|^3 t + |\xi|\right)\re^{\frac{d}{2} \xi^2 t},
\end{split} \label{e:approxadjoint}\end{align}
for $\xi \in (-\xi_0,\xi_0)$ and $t \geq 2$. On the other hand, we estimate
\begin{align} \left\|\chi(t)\re^{\lambda_r(\xi) t} \partial_\xx^i  \widetilde{\Phi}_\xi^* - \partial_\xx^i \widetilde{\Phi}_0^*\right\|_\infty \lesssim \left|\re^{\lambda_r(\xi) t}\right| \left\|\partial_\xx^i \widetilde{\Phi}_\xi^*\right\|_\infty + \left\|\partial_\xx^i \widetilde{\Phi}_0^*\right\|_\infty 
\lesssim 1 \lesssim (1+t)^{-\frac{1}{2}},  \label{e:approxadjoint2}\end{align}
for $\xi \in (-\xi_0,\xi_0)$ and $t \in [1,2]$ (taking $\xi_0$ smaller if necessary). Furthermore, there exist constants $\delta_0,\mu_0 > 0$ and a function $\lambda_h \in C^4\big(\R,\C\big)$ satisfying $\lambda_h(\xi) = \ri a\xi - d \xi^2$ for $\xi \in \R \setminus (-\frac{\xi_0}{2},\frac{\xi_0}{2})$, $\lambda_h'(0) \in \ri \R$ and $\Re \, \lambda_h(\xi) \leq -\delta_0 - \mu_0|\xi|^2$ for all $\xi \in \R$. Recalling that $\rho(\xi)-1$ vanishes on $(-\frac{\xi_0}{2},\frac{\xi_0}{2})$, we rewrite
\begin{align*}
\partial_\xx^m &\widetilde{G}_r^i(\xx,\xt,t) = \frac{1}{2\pi} \int_{\R}\re^{\ri\xi(\xx - \xt)+ \left(\ri a\xi - d \xi^2\right)t}  (\ri \xi)^m \rho(\xi)\left(\chi(t)\re^{\lambda_r(\xi) t} \partial_\xt^i  \widetilde{\Phi}_\xi(\xt)^* - \partial_\xt^i \widetilde{\Phi}_0(\xt)^*\right) \de \xi\\
&+\, \frac{\re^{-\delta_0 t}}{2\pi} \int_{\R}  \re^{\ri\xi(\xx - \xt)+ \left(\lambda_h(\xi) + \delta_0\right) t} (\ri \xi)^m\rho(\xi) \partial_\xt^i \widetilde{\Phi}_0(\xt)^* \de \xi - \frac{1}{2\pi} \int_{\R} \re^{\ri\xi(\xx - \xt)+ \lambda_h(\xi) t}(\ri \xi)^m \de \xi \partial_\xt^i \widetilde{\Phi}_0(\xt)^*,
\end{align*}
for $\xx,\xt \in \R$ and $t > 0$. We use Lemma~\ref{lem:semigroupEstimate1} and estimates~\eqref{e:approxadjoint} and~\eqref{e:approxadjoint2} to establish pointwise bounds on the first two terms on the right-hand side of the last equation. On the other hand, we apply~\cite[Lemma~A.2]{HDRS22} to estimate
\begin{align*} \left\|\frac{1}{2\pi} \int_{\R}\int_\R \re^{\ri\xi(\cdot - \xt)+ \lambda_h(\xi) t} (\ri \xi)^m \de \xi w(\xt)\de \xt\right\|_\infty \lesssim \re^{-\delta_0 t} \left(1 + t^{-\frac{m}{2}}\right) \|w\|_\infty,
\end{align*}
for $t > 0$ and $w \in C_{\mathrm{ub}}(\R)$. All in all, we obtain
\begin{align*}
\left\|\partial_\zeta^m \widetilde{S}_r^i(t)v\right\|_\infty &\lesssim \|v\|_\infty\left(\left(t^{-1-\frac{m}{2}} + \re^{-\delta_0 t} t^{-\frac{m+1}{2}}\right) \int_\R \left(1 + \frac{z^4}{t^2}\right)^{-1}\de z + \re^{-\delta_0 t} \left(1 + t^{-\frac{m}{2}}\right)\right)\\ 
&\lesssim (1+t)^{-\frac{1}{2}}t^{-\frac{m}{2}}\|v\|_{\infty},
\end{align*}
for $v \in C_{\mathrm{ub}}(\R)$ and $t \geq 1$.
\end{proof}

As explained in~\S\ref{sec:outlineproof}, we encounter a critical  term in the upcoming nonlinear stability analysis that cannot be controlled using iterative $L^\infty$-estimates. This term corresponds to a Hamilton-Jacobi nonlinearity of the form $f_p \gamma_\xx^2$, where $f_p \colon \R \to \R^n$ is $1$-periodic and $\gamma \colon \R \to \R$ is the phase modulation to be defined in~\S\ref{sec:modpert}. The following result allows us to isolate the critical term after which we can apply the Cole-Hopf transform to eliminate it, cf.~\S\ref{sec:derHamJac}. The result relies on an expansion of the $1$-periodic adjoint eigenfunction $\smash{\widetilde{\Phi}_0}$ as a Fourier series and yields a factorization of the leading-order action of the propagator $\smash{S_h^0(t)}$ on the product $f_p \smash{\gamma_\xx^2}$ into the coefficient $\smash{\langle \widetilde{\Phi}_0, f_p\rangle_{L^2(0,1)}}$ and the function obtained by the action of the convective heat semigroup $\smash{\re^{(d\partial_\xx^2 + a\partial_\xx) t}}$ on $\smash{\gamma_\xx^2}$.

\begin{proposition} \label{prop:semiapp}
Assume~\ref{assH1} and~\ref{assD1}-\ref{assD3}. Then, there exists a bounded linear operator $A_h \colon L^2_{\mathrm{per}}\big((0,1),\R^n\big) \to C(\R,\R)$ such that it holds
\begin{align*}
S_h^0(t)\left(gv\right) = \re^{\left(d\partial_\xx^2 + a\partial_\xx\right) t}\left(\langle \widetilde{\Phi}_0, g\rangle_{L^2(0,1)} v - A_h(g) \partial_\xx v\right) +  \partial_\xx \re^{\left(d\partial_\xx^2 + a\partial_\xx\right) t}\left(A_h(g) v\right),
\end{align*}
for $g \in L^2_{\mathrm{per}}((0,1),\R^n)$, $v \in C_{\mathrm{ub}}^1(\R,\R)$ and $t > 0$.
\end{proposition}
\begin{proof}
We proceed as in~\cite[Lemma~3.3]{JONZW} and expand the $1$-periodic function $\widetilde{\Phi}_0^* g \in L^2_{\mathrm{per}}((0,1),\R)$ as a Fourier series to obtain
\begin{align*}
&S_h^0(t)\left(gv\right)(\zeta) = \int_\R H(\xx-\xt,t)\widetilde{\Phi}_0(\xt)^* g(\xt) v(\xt) \de \xt\\ &\quad = \langle \widetilde{\Phi}_0, g\rangle_{L^2(0,1)} \int_\R H(\xx-\xt,t) v(\xt) \de \xt + \sum_{j \in \Z \setminus \{0\}} \int_\R \langle \widetilde{\Phi}_0^* g, \re^{2\pi \ri j \cdot }\rangle_{L^2(0,1)} \re^{2\pi \ri j \xt} H(\xx-\xt,t) v(\xt) \de \xt,\end{align*}
for $g \in L^2_{\mathrm{per}}((0,1),\R^n)$, $v \in C_{\mathrm{ub}}^1(\R,\R)$, $\xx \in \R$ and $t > 0$. We integrate by parts to rewrite the last term as
\begin{align*}
- \int_\R A_h(g)(\xt) \partial_\xt \left(H(\xx-\xt,t) v(\xt)\right) \de \xt
&= \partial_\xx \int_\R  H(\xx-\xt,t) A_h(g)(\xt) v(\xt) \de \xt\\ 
&\qquad - \, \int_\R  H(\xx-\xt,t) A_h(g)(\xt) \partial_\xt v(\xt) \de \xt,
\end{align*}
where $A_h(g)$ is the Fourier series
\begin{align*} A_h(g)(\xx) = \sum_{j \in \Z \setminus \{0\}} \frac{\langle \widetilde{\Phi}_0^* g, \re^{2\pi \ri j \cdot }\rangle_{L^2(0,1)}}{2\pi\ri j} \re^{2\pi \ri j \xx}, \qquad \xx \in \R.\end{align*}
By H\"older and Bessel's inequalities we have
\begin{align*}\sum_{j \in \Z \setminus \{0\}} \left|\frac{\langle \widetilde{\Phi}_0^* g, \re^{2\pi \ri j \cdot }\rangle_{L^2(0,1)}}{2\pi\ri j}\right| \leq \left\|\widetilde{\Phi}_0^* g\right\|_{L^2(0,1)} \left(\sum_{j \in \Z \setminus \{0\}} \frac{1}{4\pi j^2}\right)^{\frac{1}{2}} \lesssim \|g\|_{L^2(0,1)},\end{align*}
for $g \in L^2_{\mathrm{per}}((0,1),\R^n)$. Hence, the Fourier series $A_h(g)$ convergences absolutely and, thus, is continuous. In particular, $A_h \colon L^2_{\mathrm{per}}\big((0,1),\R^n\big) \to C(\R,\R)$ is a bounded linear map. Recalling~\eqref{e:convheatsem}, the desired result readily follows.
\end{proof}

\section{Nonlinear iteration scheme} \label{sec:itscheme}

We introduce the nonlinear iteration scheme that will be employed in~\S\ref{sec:nonlinearstab} to prove our nonlinear stability result, Theorem~\ref{t:mainresult}. To this end, let $u_0(x,t) = \phi_0(k_0x-\omega_0 t)$ be a diffusively spectrally stable wave-train solution to~\eqref{RD}, so that~\ref{assH1} and~\ref{assD1}-\ref{assD3} are satisfied. Moreover, let $v_0 \in C_{\mathrm{ub}}(\R)$. We consider the perturbed solution $u(t)$ to~\eqref{RD2} with initial condition $u(0) = \phi_0 + v_0$.

First, we study the equation for the unmodulated perturbation $\vt(t) = u(t) - \phi_0$ and establish $L^\infty$-bounds on the nonlinearity. As outlined in~\S\ref{sec:outlineproof}, the semilinear equation for $\vt(t)$ is only utilized to control regularity in the nonlinear stability argument. To gain sufficient temporal decay, we work with the modulated perturbation $v(t)$ defined in~\eqref{e:defv}. We derive a quasilinear equation for $v(t)$ and establish $L^\infty$-bounds on the nonlinearity. The phase modulation $\gamma(t)$ in~\eqref{e:defv} is then defined a posteriori, so that it compensates for the most critical contributions in the Duhamel formulation for $v(t)$, which can be identified using the semigroup decomposition obtained in~\S\ref{sec:decomp}. Consequently, $\gamma(t)$ is defined implicitly through an integral equation. By isolating the most critical nonlinear term in this integral equation and by relating the principal part of the semigroup to the convective heat semigroup, cf.~\S\ref{sec:decompprin}, one uncovers the perturbed viscous Hamilton-Jacobi equation that is satisfied by $\gamma(t)$. Up to a correction term, exhibiting higher-order temporal decay, the nonlinear terms in this equation are homogeneous in the variables $v(t)$, $\gamma_\xx(t)$ and their derivatives and possess $1$-periodic coefficients. Finally, we remove the most critical nonlinear term in the perturbed viscous Hamilton-Jacobi equation for $\gamma(t)$ by applying the Cole-Hopf transform and derive a Duhamel formulation for the Cole-Hopf variable. 
 
\subsection{The unmodulated perturbation}

The unmodulated perturbation $\vt(t)$ satisfies the semilinear equation~\eqref{e:umodpert}, whose Duhamel formulation reads
\begin{align}
\vt(t)=\re^{\El_0 t} v_0 + \int_0^t \re^{\El_0(t-s)}\widetilde{\mathcal{N}}(\vt(s))\de s. \label{e:intvt}
\end{align}
Using Taylor's theorem and the smoothness of the nonlinearity $f$ in~\eqref{RD} we establish the relevant nonlinear estimates.

\begin{lemma} \label{lem:nlboundsunmod}
Assume~\ref{assH1}. Fix a constant $C > 0$. Then, it holds
\begin{align*}
\left\|\NT(\vt)\right\|_\infty &\lesssim \|\vt\|_\infty^2,
\end{align*}
for $\vt \in C_{\mathrm{ub}}(\R)$ satisfying $\|\vt\|_\infty \leq C$. In addition, we have
\begin{align*}
\left\|\NT(\vt)\right\|_{W^{3,\infty}} &\lesssim  \|\vt\|_{W^{1,\infty}}^3 + \|\vt\|_{W^{2,\infty}}\|\vt\|_{W^{1,\infty}} + \|\vt\|_\infty\|\vt\|_{W^{3,\infty}},
\end{align*}
for $\vt \in C_{\mathrm{ub}}^3(\R)$ satisfying $\|\vt\|_\infty \leq C$.
\end{lemma}

Note that $\El_0$ is a sectorial operator on $C_{\mathrm{ub}}^j(\R)$ with dense domain $D(\El_0) = C_{\mathrm{ub}}^{j+2}(\R)$ for any $j \in \mathbb N_0$, cf.~\cite[Corollary~3.1.9]{LUN}. Moreover, by smoothness of $f$, the nonlinearity in~\eqref{e:umodpert} is locally Lipschitz continuous on $C_{\mathrm{ub}}^j(\R)$ for any $j \in \mathbb N_0$. Consequently, local existence and uniqueness of the unmodulated perturbation is an immediate consequence of standard analytic semigroup theory, see~\cite[Theorem~7.1.5 and~Propositions~7.1.8 and~7.1.10]{LUN}.

\begin{proposition} \label{p:local_unmod} Assume~\ref{assH1}. Let $v_0 \in C_{\mathrm{ub}}(\R)$. Then, there exists a maximal time $T_{\max} \in (0,\infty]$ such that~\eqref{e:umodpert} admits a unique classical solution
\begin{align*}\vt \in C\big([0,T_{\max}),C_{\mathrm{ub}}(\R)\big) \cap C\big((0,T_{\max}),C_{\mathrm{ub}}^2(\R)\big) \cap C^1\big((0,T_{\max}),C_{\mathrm{ub}}(\R)\big), \end{align*}
with initial condition $\vt(0) = v_0$. Moreover, the map $[0,T_{\max}) \to C_{\mathrm{ub}}(\R), t \mapsto \sqrt{t} \, \vt_\xx(t)$ is continuous and, if $T_{\max} < \infty$, then we have
\begin{align*} \limsup_{t \uparrow T_{\max}} \left\|\vt(t)\right\|_\infty = \infty. \end{align*}
Finally, if $v_0 \in C_{\mathrm{ub}}^j(\R)$ for some $j \in \mathbb{N}_0$, then $\vt \in  C\big([0,T_{\max}),C_{\mathrm{ub}}^j(\R)\big) \cap C\big((0,T_{\max}),C_{\mathrm{ub}}^{j+2}(\R)\big)$ and there exist constants $K, r > 0$ and a time $t_0 \in (0,T_{\max})$, which are independent of $v_0$, such that, if $\|v_0\|_{W^{j,\infty}} < r$, then it holds
\begin{align*}
\|\vt(t)\|_{W^{j,\infty}} + \sqrt{t} \, \|\vt(t)\|_{W^{j + 1,\infty}} \leq K\|v_0\|_{W^{j,\infty}},
\end{align*}
for all $t \in [0,t_0]$.
\end{proposition}

\subsection{The modulated perturbation} \label{sec:modpert}

We define the modulated perturbation $v(t)$ by~\eqref{e:defv}, where the spatio-temporal phase $\gamma(t)$ satisfies $\gamma(0)=0$ and is to be defined a posteriori. Substituting $u(\xx - \gamma(\xx,t),t) = \phi_0(\xx) + v(\xx,t)$ into~\eqref{RD2} yields the quasilinear equation
\begin{align}
\left(\partial_t - \El_0\right)\left[v + \phi_0' \gamma\right] = \Non(v,\gamma,\partial_t \gamma) + \left(\partial_t - \El_0\right)\left[\gamma_\xx v\right], \label{e:modpertbeq}
\end{align}
for $v(t)$, where the nonlinearity $\Non$ is given by
\begin{align*}
\Non(v,\gamma,\gamma_t) &= \mathcal Q(v,\gamma) + \partial_\xx \mathcal R(v,\gamma,\gamma_t) + \partial_{\xx\xx} \mathcal S(v,\gamma),
\end{align*}
with
\begin{align*}
\mathcal Q(v,\gamma) &= \left(f(\phi_0+v) - f(\phi_0) - f'(\phi_0) v\right)\left(1-\gamma_\xx\right),\\
\mathcal R(v,\gamma,\gamma_t) &= -\gamma_t v + \omega_0 \gamma_\xx v + \frac{k_0^2}{1-\gamma_{\xx}} D \left(\gamma_\xx^2 \phi_0' - \frac{\gamma_{\xx\xx} v}{1-\gamma_\xx}\right),\\
\mathcal S(v,\gamma) &= k_0^2D \left(2\gamma_\xx v + \frac{\gamma_\xx^2 v}{1-\gamma_\xx}\right),
\end{align*}
cf.~\cite[Lemma~4.2]{JONZ}. With the aid of Taylor's theorem it is relatively straightforward to verify the relevant nonlinear bounds.

\begin{lemma} \label{lem:nlboundsmod}
Assume~\ref{assH1}. Fix a constant $C > 0$. Then, we have 
\begin{align*}
\|\mathcal Q(v,\gamma)\|_\infty &\lesssim \|v\|_\infty^2,\\
\|\mathcal R(v,\gamma,\gamma_t)\|_\infty &\lesssim \|v\|_\infty \|(\gamma_\xx,\gamma_t)\|_{W^{1,\infty} \times L^\infty} + \|\gamma_\xx\|_\infty^2,\\
\|\mathcal S(v,\gamma)\|_\infty &\lesssim \|v\|_\infty \|\gamma_\xx\|_\infty,
\end{align*}
for $v \in C_{\mathrm{ub}}(\R)$ and $(\gamma,\gamma_t) \in C_{\mathrm{ub}}^2(\R) \times C_{\mathrm{ub}}(\R)$ satisfying $\|v\|_\infty \leq C$ and $\|\gamma_\xx\|_\infty \leq \frac{1}{2}$. Moreover, it holds
\begin{align*}
\left\|\partial_\xx \mathcal S(v,\gamma)\right\|_\infty &\lesssim \|v\|_{W^{1,\infty}} \|\gamma_\xx\|_{W^{1,\infty}},
\end{align*}
for $v \in C_{\mathrm{ub}}^1(\R)$ and $\gamma \in C_{\mathrm{ub}}^2(\R)$ satisfying $\|\gamma_\xx\|_\infty \leq \frac{1}{2}$. Finally, we have
\begin{align*}
\|\mathcal{Q}(v,\gamma)\|_{W^{1,\infty}} &\lesssim \|v\|_\infty\|v\|_{W^{1,\infty}},\\
\|\mathcal{R}(v,\gamma,\gamma_t)\|_{W^{2,\infty}} &\lesssim \|(\gamma_\xx,\gamma_t)\|_{W^{3,\infty} \times W^{2,\infty}} \left(\|v\|_{W^{2,\infty}} + \|\gamma_\xx\|_{W^{1,\infty}}\right),\\
\|\mathcal{S}(v,\gamma)\|_{W^{3,\infty}} &\lesssim \|\gamma_\xx\|_{W^{3,\infty}} \|v\|_{W^{3,\infty}},
\end{align*}
for $v \in C_{\mathrm{ub}}^3(\R)$ and $(\gamma,\gamma_t) \in C_{\mathrm{ub}}^4(\R) \times C_{\mathrm{ub}}^2(\R)$ satisfying $\|v\|_\infty \leq C$ and $\|\gamma_\xx\|_{W^{1,\infty}} \leq \frac{1}{2}$.
\end{lemma}

Integrating~\eqref{e:modpertbeq} and recalling $\gamma(0) = 0$ yields the Duhamel formulation
\begin{align}
v(t)+\phi_0'\gamma(t) = \re^{\El_0 t} v_0 + \int_0^t \re^{\El_0(t-s)}\mathcal{N}(v(s),\gamma(s),\partial_t \gamma(s))\de s + \gamma_\xx(t)v(t). \label{e:intv}
\end{align}
As in~\cite{JONZW,JONZ} we make a judicious choice for $\gamma(t)$ so that the linear term $\phi_0'\gamma(t)$ compensates for the most critical nonlinear contributions in~\eqref{e:intv}. Thus, motivated by the semigroup decomposition~\eqref{e:semidecomp}, we introduce the variables
\begin{align} \label{e:defz}
\begin{split}
z(t) &:= v(t) - k_0 \partial_k \phi(\cdot;k_0) \gamma_\xx(t),\\
\widetilde{\gamma}(t) &:= \partial_t \gamma(t) - a\gamma_\xx(t),
\end{split}
\end{align}
cf.~Proposition~\ref{prop:family}, and make the implicit choice
\begin{align}
\gamma(t) = S_p^0(t)v_0 + \int_0^t S_p^0(t-s) \mathcal{N}(v(s),\gamma(s),\partial_t \gamma(s))\de s. \label{e:intgamma}
\end{align}
Substituting $v(t) + \phi_0'\gamma(t) = z(t) + \left(\phi_0' + k_0 \partial_k \phi(\cdot;k_0) \partial_\xx \right) \gamma(t)$ and~\eqref{e:intgamma} into~\eqref{e:intv} leads to the Duhamel formulation
\begin{align}
\begin{split}
z(t) &= \widetilde{S}(t)v_0+\int_0^t\widetilde{S}(t-s)\mathcal{N}(v(s),\gamma(s),\partial_t \gamma(s)) \de s + \gamma_\xx(t)v(t),
\end{split}\label{e:intz}
\end{align}
for the residual $z(t)$. 

Noting that $v(t) = z(t) + k_0 \partial_k \phi(\cdot;k_0) \gamma_\xx(t)$ and $\partial_t \gamma(t) = \widetilde{\gamma}(t) + a \gamma_\xx(t)$, the equations~\eqref{e:intgamma} and~\eqref{e:intz} form a closed system in $z(t)$, $\gamma_\xx(t)$, $\widetilde{\gamma}(t)$ and their derivatives. Since $\smash{\widetilde{S}(t)}$ and $(\partial_t - a\partial_\xx)S_p^0(t)$ decay algebraically at rate $(1+t)^{-1}$ as operators on $C_{\mathrm{ub}}(\R)$ by Propositions~\ref{prop:semexp} and~\ref{prop:semdif}, $z(t)$ and $\widetilde{\gamma}(t)$ are expected to exhibit higher-order decay (at least on the linear level). On the other hand, since $\partial_\xx S_p^0(t)$ decays at rate $\smash{(1+t)^{-1/2}}$ and can be directly related to the derivative of the convective heat semigroup $\partial_\xx \smash{\re^{(d\partial_\xx^2 + a\partial_\xx) t}}$, we cannot expect that $\gamma_{\xx}(t)$ decays faster than $\smash{(1+t)^{-1/2}}$, cf.~Propositions~\ref{prop:semiprinrec} and~\ref{prop:semiapp}. Hence, implicitly defining the phase modulation $\gamma(t)$ by~\eqref{e:intgamma}, it indeed captures the most critical terms in~\eqref{e:intv}. Moreover, it holds $\gamma(0) = 0$, because we have $S_p^0(0) = 0$.

Local existence and uniqueness of the phase modulation follows by applying a standard contraction mapping argument to~\eqref{e:intgamma}, where we use that the modulated perturbation $v(t)$ can be expressed as
\begin{align} \label{e:defv2}
v(\xx,t) = \vt(\xx-\gamma(\xx,t),t) + \phi_0(\xx-\gamma(\xx,t)) - \phi_0(\xx).
\end{align}
Thus, having established local existence and uniqueness of the unmodulated perturbation in Proposition~\ref{p:local_unmod}, the integral equation~\eqref{e:intgamma} yields a closed fixed point problem in $\gamma(t)$ and its temporal derivatives. We arrive at the following result, whose proof has been delegated to Appendix~\ref{app:B}. 

\begin{proposition} \label{p:gamma}
Assume~\ref{assH1}. Let $v_0 \in C_{\mathrm{ub}}(\R)$ and $j,l, m \in \mathbb{N}_0$. For $\vt$ and $T_{\max}$ as in Proposition~\ref{p:local_unmod}, there exists a maximal time $\tau_{\max} \in (0,T_{\max}]$ such that~\eqref{e:intgamma}, with $v$ given by~\eqref{e:defv2}, has a solution
\begin{align*}\gamma \in C\big([0,\tau_{\max}),C_{\mathrm{ub}}^{2 + m}(\R)\big) \cap C^{1+j}\big([0,\tau_{\max}),C_{\mathrm{ub}}^{l}(\R)\big), \end{align*}
satisfying $\gamma(t)=0$ for all $t \in [0,\tau_{\max})$ with $t \leq 1$. In addition, it holds $\|(\gamma(t),\partial_t \gamma(t)\|_{W^{2,\infty} \times L^\infty} < \frac12$ for all $t \in [0,\tau_{\max})$. Finally, if $\tau_{\max} < T_{\max}$, then we have
\begin{align*} \limsup_{t \uparrow \tau_{\max}} \left\|\left(\gamma(t),\partial_t \gamma(t)\right)\right\|_{W^{2,\infty} \times L^\infty} = \frac12.\end{align*}
\end{proposition}

Recalling the definitions of the modulated perturbation $v(t)$ and the residual $z(t)$, their local existence and regularity is an immediate consequence of Propositions~\ref{p:local_unmod} and~\ref{p:gamma}.

\begin{corollary} \label{C:local_v}
Assume~\ref{assH1} and~\ref{assD3}. Let $v_0 \in C_{\mathrm{ub}}(\R)$. For $\vt$ as in Proposition~\ref{p:local_unmod} and $\gamma$ and $\tau_{\max}$ as in Proposition~\ref{p:gamma}, the modulated perturbation $v$, defined by~\eqref{e:defv}, and the residual $z$, defined by~\eqref{e:defz}, satisfy
$$v,z \in C\big([0,\tau_{\max}),C_{\mathrm{ub}}(\R)\big) \cap C\big((0,\tau_{\max}),C_{\mathrm{ub}}^2(\R)\big) \cap C^1\big((0,\tau_{\max}),C_{\mathrm{ub}}(\R)\big).$$
Moreover, their Duhamel formulations~\eqref{e:intv} and~\eqref{e:intz} hold for $t \in [0,\tau_{\max})$.
\end{corollary}

\subsection{Derivation of perturbed viscous Hamilton-Jacobi equation} \label{sec:derHamJac}

We derive a perturbed viscous Hamilton-Jacobi equation for the phase modulation $\gamma(t)$. First, we collect all $\gamma_\zeta^2$-contributions in the nonlinearity $\Non(v,\gamma,\gamma_t)$ in~\eqref{e:intgamma}, which are the nonlinear terms from which we expect the slowest decay. Then, we use the decomposition~\eqref{e:prindecomp} of the propagator $S_p^0(t)$ and Proposition~\ref{prop:semiapp} to decompose~\eqref{e:intgamma} in a part of the form
\begin{align} \int_0^t \re^{\left(d\partial_\xx^2 + a\partial_\xx\right) (t-s)} F(z(s),v(s),\gamma(s),\widetilde{\gamma}(s)) \de s, \label{e:form}\end{align}
and a residual $r(t)$ from which we expect higher-order temporal decay. Correcting the phase $\gamma(t)$ with the residual term $r(t)$ and applying the convective heat operator $\partial_t - d\partial_\xx^2 - a\partial_\xx$ to~\eqref{e:intgamma} then yields the desired perturbed viscous Hamilton-Jacobi equation with nonlinearity $F(z,v,\gamma,\widetilde{\gamma})$. 

We start by isolating the $\gamma_\xx^2$-contributions in the nonlinearity in~\eqref{e:intgamma}. Recalling $v(t) = z(t) + k_0 \partial_k \phi(\cdot;k_0) \gamma_\xx(t)$, $\partial_t \gamma(t) = \widetilde{\gamma}(t) + a \gamma_\xx(t)$ and $a = \omega_0 - k_0\omega'(k_0)$, cf.~\eqref{e:defad} and~\eqref{e:defz}, we rewrite the nonlinearity in~\eqref{e:intgamma} as
\begin{align} \label{e:nonldecomp}
\Non(v(s),\gamma(s),\partial_t \gamma(s)) = k_0^2 f_p \gamma_{\xx}(s)^2 + \mathcal N_p(z(s),v(s),\gamma(s),\widetilde{\gamma}(s)),
\end{align}
where the 1-periodic function
\begin{align*}
f_p &= \frac{1}{2} f''(\phi_0)\left(\partial_k \phi(\cdot;k_0),\partial_k \phi(\cdot;k_0)\right) + \omega'(k_0) \partial_{\xx k} \phi(\cdot;k_0) + D\left(\phi_0'' + 2k_0 \partial_{\xx\xx k} \phi(\cdot;k_0)\right),
\end{align*}
captures all $\gamma_\xx^2$-contributions, cf.~\cite[Lemma~5.1]{JONZW}, and the residual is given by
\begin{align*}
\mathcal N_p(z,v,\gamma,\widetilde{\gamma}) &= \mathcal Q_p(z,v,\gamma) + \partial_\xx \mathcal R_p(z,v,\gamma,\widetilde{\gamma}) + \partial_{\xx\xx} \mathcal S_p(z,v,\gamma),
\end{align*}
with
\begin{align*}
\mathcal Q_p(z,v,\gamma) &= \left(f(\phi_0+v) - f(\phi_0) - f'(\phi_0) v\right)\gamma_\xx + f(\phi_0+v) - f(\phi_0) - f'(\phi_0) v - \frac{1}{2} f''(\phi_0)(v,v)\\
&\qquad + \frac{1}{2}f''(\phi_0)(z,z) + k_0 \gamma_{\xx} f''(\phi_0)(z,\partial_k \phi(\cdot;k_0)) + 2k_0^2\omega'(k_0) \gamma_\xx \gamma_{\xx\xx} \partial_k \phi(\cdot;k_0)\\
&\qquad + \, 2k_0^2 D \left(\gamma_\xx \gamma_{\xx\xx} \left(\phi_0' + 4k_0\partial_{\xx k} \phi(\cdot;k_0)\right) + 2k_0\left(\gamma_{\xx\xx}^2 + \gamma_\xx \gamma_{\xx\xx\xx}\right) \partial_k \phi(\cdot;k_0)\right),\\
\mathcal R_p(z,v,\gamma,\widetilde{\gamma}) &= -v \widetilde{\gamma} + k_0\omega_0'(k_0) \gamma_{\xx} z + \frac{k_0^2}{1-\gamma_\xx}D\left(\gamma_\xx^3 \phi_0' - \frac{\gamma_{\xx\xx} v}{1-\gamma_\xx}\right),\\
\mathcal S_p(z,v,\gamma) &= k_0^2D \left(2\gamma_\xx z + \frac{\gamma_\xx^2 v}{1-\gamma_\xx}\right).
\end{align*}
With the aid of Taylor's theorem we readily establish the following nonlinear estimates.

\begin{lemma} \label{lem:nlboundsmod3}
Assume~\ref{assH1} and~\ref{assD3}. Fix a constant $C > 0$. Then, we have
\begin{align*}
\begin{split}
\|\mathcal Q_p(z,v,\gamma)\|_\infty &\lesssim \left(\|v\|_\infty + \|\gamma_\xx\|_\infty\right)\|v\|_\infty^2 + \left(\|z\|_\infty + \|\gamma_\xx\|_\infty\right)\|z\|_\infty + \|\gamma_\xx\|_{W^{1,\infty}}\|\gamma_{\xx\xx}\|_{W^{1,\infty}},\\
\|\mathcal R_p(z,v,\gamma,\widetilde{\gamma})\|_\infty &\lesssim \|v\|_\infty \left(\|\widetilde{\gamma}\|_\infty + \|\gamma_{\xx\xx}\|_\infty\right) + \|z\|_\infty \|\gamma_\xx\|_\infty + \|\gamma_\xx\|^3_\infty,\\
\|\mathcal S_p(z,v,\gamma)\|_\infty &\lesssim \|v\|_\infty\|\gamma_\xx\|^2_\infty + \|z\|_\infty \|\gamma_\xx\|_\infty,
\end{split}
\end{align*}
for $z \in C_{\mathrm{ub}}(\R)$, $v \in C_{\mathrm{ub}}(\R)$ and $(\gamma,\widetilde{\gamma}) \in C_{\mathrm{ub}}^3(\R) \times C_{\mathrm{ub}}(\R)$ satisfying $\|v\|_\infty \leq C$ and $\|\gamma_\xx\|_\infty \leq \frac{1}{2}$.
\end{lemma}

Using the commutation relations~\eqref{e:commu}, the decomposition~\eqref{e:prindecomp} of the propagator $S_p^i(t)$, the decomposition of $S_h^0(t-s)(f_p \gamma_\zeta(s)^2)$ established in Proposition~\ref{prop:semiapp} and the decomposition~\eqref{e:nonldecomp} of the nonlinearity, we rewrite~\eqref{e:intgamma} as
\begin{align}
\begin{split}
\gamma(t) &= r(t) + S_h^0(t)v_0 + \int_0^t \re^{\left(d\partial_\xx^2 + a\partial_\xx\right) (t-s)} \left(\nu \gamma_\xx(s)^2 - k_0^2 A_h(f_p) \partial_\xx \left(\gamma_\xx(s)^2\right)\right) \de s\\
&\qquad + \, \int_0^t S_h^0(t-s) \mathcal Q_p(z(s),v(s),\gamma(s))\de s
- \int_0^t S_h^1(t-s) \mathcal R_p(z(s),v(s),\gamma(s),\widetilde{\gamma}(s)) \de s\\
&\qquad + \, \int_0^t S_h^2(t-s) \mathcal S_p(z(s),v(s),\gamma(s)) \de s, 
\end{split}
\label{e:intgamma2}
\end{align}
where by the computations~\cite[Section~4.2]{DSSS} we have
$$\nu = k_0^2\langle \widetilde{\Phi}_0, f_p\rangle_{L^2(0,1)} = -\frac{1}{2}k_0^2 \omega''(k_0),$$
and where $r(t)$ is the residual
\begin{align}
\begin{split}
r(t) &:= \widetilde{S}_r^0(t)v_0 + k_0^2 \int_0^t \widetilde{S}_r^0(t-s)\left(f_p \gamma_\xx(s)^2\right) \de s + k_0^2 \partial_\xx \int_0^t \re^{\left(d\partial_\xx^2 + a\partial_\xx\right) (t-s)}\left(A_h(f_p) \gamma_\xx(s)^2\right)\\
&\quad + \, \partial_\xx \int_0^t S_p^0(t-s) \mathcal R_p(z(s),v(s),\gamma(s),\widetilde{\gamma}(s)) \de s + \partial_\xx^2 \int_0^t S_p^0(t-s) \mathcal S_p(z(s),v(s),\gamma(s)) \de s\\ 
&\quad - \, 2\partial_\xx \int_0^t S_p^1(t-s) \mathcal S_p(z(s),v(s),\gamma(s)) \de s + \int_0^t \widetilde{S}_r^0(t-s) \mathcal Q_p(z(s),v(s),\gamma(s)) \de s\\
&\quad - \, \int_0^t \widetilde{S}_r^1(t-s) \mathcal R_p(z(s),v(s),\gamma(s),\widetilde{\gamma}(s)) \de s
+ \int_0^t \widetilde{S}_r^2(t-s) \mathcal S_p(z(s),v(s),\gamma(s)) \de s,
\end{split}\label{e:intr}
\end{align}
capturing all contributions in~\eqref{e:intgamma}, which we expect to exhibit higher-order decay, cf.~Proposition~\ref{prop:semiprinrec}.

Recalling the definition~\eqref{e:Shexpr} of $S_h^i(t)$, equation~\eqref{e:intgamma2} implies that $\gamma(t) - r(t)$ is indeed of the form~\eqref{e:form}, where we have 
$F(z,v,\gamma,\widetilde{\gamma}) = \nu \gamma_\xx^2 + G(z,v,\gamma,\widetilde{\gamma})$ with
$$G(z,v,\gamma,\widetilde{\gamma}) = - k_0^2 A_h(f_p) \partial_\xx (\gamma_\xx^2) + \widetilde{\Phi}_0^*\mathcal Q_p(z,v,\gamma) - \left(\partial_\xx \widetilde{\Phi}_0^*\right)\mathcal R_p(z,v,\gamma,\widetilde{\gamma}) + \left(\partial_\xx^2 \widetilde{\Phi}_0^*\right)\mathcal S_p(z,v,\gamma).$$
Note that the map $t \mapsto F(z(t),v(t),\gamma(t),\widetilde{\gamma}(t))$ lies in $C\big([0,\tau_{\max}),C_{\mathrm{ub}}(\R)\big) \cap C^1\big((0,\tau_{\max}),C_{\mathrm{ub}}(\R)\big)$ by Proposition~\ref{p:gamma} and Corollary~\ref{C:local_v}. So, standard analytic semigroup theory, cf.~\cite[Theorem~4.3.4]{LUN}, readily yields the regularity of $\gamma(t) - r(t)$ and, thus, of $r(t)$ by Proposition~\ref{p:gamma}.

\begin{corollary} \label{C:local_r}
Assume~\ref{assH1} and~\ref{assD3}. Let $v_0 \in C_{\mathrm{ub}}(\R)$. For $\gamma$ and $\tau_{\max}$ as in Proposition~\ref{p:gamma} and for $v$ and $z$ as in Corollary~\ref{C:local_v}, the residual $r$, defined by~\eqref{e:intr}, satisfies
$$r \in C\big([0,\tau_{\max}),C_{\mathrm{ub}}(\R)\big) \cap C\big((0,\tau_{\max}),C_{\mathrm{ub}}^2(\R)\big) \cap C^1\big((0,\tau_{\max}),C_{\mathrm{ub}}(\R)\big).$$
In addition, the map $[0,\tau_{\max}) \to C_{\mathrm{ub}}(\R), t \mapsto \sqrt{t} \, r_\xx(t)$ is continuous.
\end{corollary}

So, we can apply the convective heat operator to $\partial_t - d\partial_\xx^2 - a\partial_\xx$ to~\eqref{e:intgamma2}, which leads us to the desired perturbed viscous Hamilton-Jacobi equation 
\begin{align} \label{e:hamjac}
\left(\partial_t - d\partial_\xx^2 - a\partial_\xx\right)\left(\gamma - r\right) = \nu \gamma_\xx^2 + G(z,v,\gamma,\widetilde{\gamma}),
\end{align}
Indeed, up to the correction term $r(t)$ and the perturbation 
$G(z,v,\gamma,\widetilde{\gamma})$, which are expected to exhibit higher-order temporal decay, equation~\eqref{e:hamjac} coincides with the Hamilton-Jacobi equation~\eqref{e:HamJac}. 

\subsection{Application of the Cole-Hopf transform}

We apply the Cole-Hopf transform to remove the critical nonlinear term $\nu \gamma_\xx^2$ in~\eqref{e:hamjac}. Thus, we introduce the new variable
\begin{align} \label{e:defy} y(t) = \re^{\frac{\nu}{d} \left(\gamma(t) - r(t)\right)} - 1,\end{align}
which satisfies
\begin{align}\label{e:regy} y \in C\big([0,\tau_{\max}),C_{\mathrm{ub}}(\R)\big) \cap C\big((0,\tau_{\max}),C_{\mathrm{ub}}^2(\R)\big) \cap C^1\big((0,\tau_{\max}),C_{\mathrm{ub}}(\R)\big),\end{align}
by Proposition~\ref{p:gamma} and Corollary~\ref{C:local_r}. We arrive at the convective heat equation
\begin{align} \left(\partial_t - d\partial_\xx^2 - a\partial_\xx\right)y = 2\nu r_\xx y_\xx + \frac{\nu}{d}\left(\nu r_\xx^2 + G(z,v,\gamma,\widetilde{\gamma})\right)\left(y+1\right), \label{e:colehopf} \end{align}
which is linear in $y$.

By Proposition~\ref{p:gamma} the phase modulation $\gamma$ vanishes on $[0,1]$. Hence, in the upcoming nonlinear argument the Cole-Hopf variable $y(t)$ can be controlled by the residual $r(t)$ through
\begin{align} \label{e:ysmall}
y(t) = \re^{-\frac{\nu}{d} r(t)} - 1.
\end{align}
for $t \in [0,\tau_{\max})$ with $t \leq 1$. On the other hand, we will control $y(t)$ through the Duhamel formulation
\begin{align}
\begin{split}
y(t) &= \re^{\left(d\partial_\xx^2 + a\partial_\xx\right) (t-1)} y(1) + \int_1^t \re^{\left(d\partial_\xx^2 + a\partial_\xx\right) (t-s)} \Non_c(r(s),y(s),z(s),v(s),\gamma(s),\widetilde{\gamma}(s)) \de s,
\end{split}
\label{e:inty}
\end{align}
for $t \in [0,\tau_{\max})$ with $t \geq 1$, where the nonlinearity is given by
\begin{align*}
\Non_c(r,y,z,v,\gamma,\widetilde{\gamma}) = 2\nu r_\xx y_\xx + \frac{\nu}{d}\left(\nu r_\xx^2 + G(z,v,\gamma,\widetilde{\gamma})\right)\left(y+1\right).
\end{align*}
The reason to use the formula~\eqref{e:ysmall} for short-time control of $y(t)$ (rather than its Duhamel formulation) is that the nonlinear term $r_\xx(s)^2$ in $\mathcal{N}_c$ obeys a nonintegrable short-time bound, which blows up as $1/s$ as $s \downarrow 0$, cf.~Corollary~\ref{C:local_r}. 

Using Proposition~\ref{prop:semiapp} and Lemma~\ref{lem:nlboundsmod3}, one readily obtains the relevant nonlinear estimate.

\begin{lemma}\label{lem:nlboundsmod4}
Fix a constant $C > 0$. Then, we have
\begin{align*}
\begin{split}
\|\Non_c(r,y,z,v,\gamma,\widetilde{\gamma})\|_\infty &\lesssim \|r_\xx\|_\infty \|y_\xx\|_\infty + \|r_\xx\|_\infty^2 + \left(\|v\|_\infty + \|\gamma_\xx\|_\infty\right)\|v\|_\infty^2 + \left(\|z\|_\infty + \|\gamma_\xx\|_\infty\right)\|z\|_\infty\\ 
&\qquad +\, \|\gamma_\xx\|_{W^{1,\infty}}\|\gamma_{\xx\xx}\|_{W^{1,\infty}} + \|v\|_\infty \left(\|\widetilde{\gamma}\|_\infty + \|\gamma_{\xx\xx}\|_\infty + \|\gamma_{\xx}\|^2_\infty\right) + \|\gamma_\xx\|^3_\infty,
\end{split}
\end{align*}
for $z \in C_{\mathrm{ub}}(\R)$, $v \in C_{\mathrm{ub}}(\R)$, $r,y \in C_{\mathrm{ub}}^1(\R)$ and $(\gamma,\widetilde{\gamma}) \in C_{\mathrm{ub}}^3(\R) \times C_{\mathrm{ub}}(\R)$ satisfying $\|v\|_\infty, \|y\|_\infty \leq C$ and $\|\gamma_\xx\|_\infty \leq \frac{1}{2}$.
\end{lemma}

\section{Nonlinear stability analysis} \label{sec:nonlinearstab}

We prove our main result, Theorem~\ref{t:mainresult}, by applying the linear estimates, established in Propositions~\ref{prop:full},~\ref{prop:semexp},~\ref{prop:semdif} and~\ref{prop:semiprinrec} and Corollary~\ref{cor:der}, and the nonlinear estimates, established in Lemmas~\ref{lem:nlboundsunmod},~\ref{lem:nlboundsmod},~\ref{lem:nlboundsmod3} and~\ref{lem:nlboundsmod4}, to the nonlinear iteration scheme, which consists of the integral equations~\eqref{e:intvt},~\eqref{e:intgamma},~\eqref{e:intz},~\eqref{e:intr} and~\eqref{e:inty}.

\begin{proof}[Proof of Theorem~\ref{t:mainresult}] We close a nonlinear iteration scheme, controlling the unmodulated perturbation $\vt \in C\big([0,T_{\max}),C_{\mathrm{ub}}(\R)\big) \cap C\big((0,T_{\max}),C_{\mathrm{ub}}^2(\R)\big)$, the residuals $r,z \in C\big([0,\tau_{\max}),C_{\mathrm{ub}}(\R)\big) \cap C\big((0,\tau_{\max}),C_{\mathrm{ub}}^2(\R)\big)$, the phase modulation $\gamma \in C\big([0,\tau_{\max}),C_{\mathrm{ub}}^4(\R)\big) \cap C^1\big([0,\tau_{\max}),C_{\mathrm{ub}}^2(\R)\big)$ and the Cole-Hopf variable $y \in C\big([0,\tau_{\max}),C_{\mathrm{ub}}(\R)\big) \cap C\big((0,\tau_{\max}),C_{\mathrm{ub}}^2(\R)\big)$. 
\bigskip\\
\noindent\textbf{Short-time argument.} By Proposition~\ref{p:local_unmod} there exist constants $K_1, r_1 > 0$ and a time $t_1 \in (0,\tfrac{1}{4}\min\{1,T_{\max}\})$, which are independent of $v_0$, such that, if $E_0 < r_1$, then it holds $\|\vt(t)\|_\infty + \sqrt{t}\|\vt(t)\|_{W^{1,\infty}} \leq K_1 E_0$ for all $t \in [0,4t_1]$. Next, we iteratively apply Proposition~\ref{p:local_unmod} with initial condition $\vt(t_i) \in C_{\mathrm{ub}}^i(\R)$ for $i = 1,2,3,$ to yield constants $K_{i+1},r_{i+1}>0$ and a time $t_{i+1} \in (t_i,2t_1)$, which are independent of $\vt(t_i)$, such that we have $\vt \in C\big((t_i,T_{\max}),C_{\mathrm{ub}}^{i+1}(\R)\big)$ and, if $\|\vt(t_i)\|_{W^{i,\infty}} < r_{i+1}$, then it holds
\begin{align*}{\|\vt(t)\|_{W^{i,\infty}} + \sqrt{t}\,\|\vt(t)\|_{W^{i+1,\infty}} \leq K_{i+1} \|\vt(t_i)\|_{W^{i,\infty}}\leq K_1 \ldots K_{i+1} E_0/\sqrt{t_1 \ldots t_i}},\end{align*} 
for $t \in [t_i,t_{i+1}]$. All in all, we infer that there exist constants $K_*,r_* > 0$ and a time $t_* \in (0,\tfrac12 \min\{1,T_{\max}\})$, which are independent of $v_0$, such that, if $E_0 \leq r_*$, then we have
\begin{align} \label{e:shorttime}
\vt \in C\big([t_*,T_{\max}),C_{\mathrm{ub}}^4(\R)\big), \qquad \|\vt(t_*)\|_{W^{4,\infty}} \leq K_* E_0, \qquad \|\vt(t)\|_\infty + \sqrt{t}\,\|\vt(t)\|_{W^{1,\infty}} \leq K_* E_0,
\end{align}
for $t \in [0,2t_*]$.
\bigskip\\
\noindent\textbf{Template function.} Let $\varrho \colon [0,\infty) \to [0,1]$ be a smooth cutoff function which vanishes on $[0,t_*]$ and satisfies $\varrho(t) = 1$ for all $t \in [2t_*,\infty) \supset [1,\infty)$. By Propositions~\ref{p:local_unmod} and~\ref{p:gamma}, Corollaries~\ref{C:local_v} and~\ref{C:local_r} and identities~\eqref{e:regy} and~\eqref{e:shorttime}, the template function $\eta \colon [0,\tau_{\max}) \to \R$ given by\footnote{For the motivation of the choice of temporal weights in the template function $\eta(t)$, we refer to Remark~\ref{rem:choice_eta} below.}
\begin{align*}
\eta(t) &= \sup_{0\leq s\leq t} \left[\varrho(s)\left(\sum_{i = 0}^2 \frac{\|\partial_\xx^{2+i} \vt(s)\|_\infty}{(1+s)^{\frac{i}{2}} \log(2+s)^{1-\frac{i}{2}}} + \frac{\|z_{\xx\xx}(s)\|_\infty + \sqrt{1+s} \|z_{\xx}(s)\|_\infty}{\log(2+s)}\right) + \|\gamma(s)\|_\infty \right.\\
&\left.\phantom{\sum_{i = 0}^2} \qquad \ \ + \|y(s)\|_\infty + \sqrt{s}\, \|y_\xx(s)\|_\infty +  \sqrt{1+s}\left(\|\gamma_\xx(s)\|_\infty  + \frac{\|r(s)\|_\infty + \sqrt{s}\,\|r_{\xx}(s)\|_\infty}{\log(2+s)}\right) \right. \\
&\left.\phantom{\sum_{i = 0}^2} \qquad\ \ + \|\vt(s)\|_\infty + \frac{\sqrt{s}\,\|\vt_\xx(s)\|_\infty}{\sqrt{1+s}} + \frac{1+s}{\log(2+s)}\left(\|z(s)\|_\infty + \left\|(\gamma_{\xx\xx}(s),\widetilde{\gamma}(s))\right\|_{W^{2,\infty} \times W^{2,\infty}}\right)\right],
\end{align*}
is well-defined and continuous, where we recall $\widetilde{\gamma}(t) = \partial_t \gamma(t) - a\gamma_\xx(t)$. Moreover, if $\tau_{\max} < \infty$, then it holds
\begin{align}
 \lim_{t \uparrow \tau_{\max}} \eta(t) \geq \frac12. \label{e:blowupeta}
\end{align}
\smallskip\\
\noindent\textbf{Approach.} Our aim is to prove that there exists a constant $C > 1$ such that for all $t \in [0,\tau_{\max})$ with $\eta(t) \leq \frac{1}{2}$ we have the key inequality
\begin{align}
\eta(t) \leq C\left(E_0 + \eta(t)^2\right). \label{e:etaest}
\end{align}
Then, taking
$$\epsilon = \min\left\{\frac{1}{4C^2},r_*\right\} \qquad M_0 = 2C,$$
it follows by the continuity, monotonicity and non-negativity of $\eta$ that, provided $E_0 \in (0,\epsilon)$, we have $\eta(t) \leq M_0E_0 = 2CE_0 < \frac{1}{2}$ for all $t \in [0,\tau_{\max})$. Indeed, given $t \in [0,\tau_{\max})$ with $\eta(s) \leq 2CE_0$ for each $s \in [0,t]$, we arrive at
$$\eta(t) \leq C\left(E_0 + 4C^2E_0^2\right) < 2CE_0,$$
by estimate~\eqref{e:etaest}. Thus, if the key inequality~\eqref{e:etaest} is satisfied, then we have $\eta(t) \leq 2CE_0 < \frac12$, for all $t \in [0,\tau_{\max})$, which implies by~\eqref{e:blowupeta} that it must hold $\tau_{\max} = \infty$. Consequently, $\eta(t) \leq M_0E_0$ is satisfied for all $t \geq 0$, which yields~\eqref{e:mtest10} and~\eqref{e:mtest2} for any $M \geq M_0$. In addition, recalling the definitions~\eqref{e:defv} and~\eqref{e:defz} of $v(t)$ and $z(t)$, respectively, we obtain
\begin{align*} 
\left\|u\left(\cdot-\gamma(\cdot,t),t\right)-\phi_0\right\|_\infty &\leq \|z(t)\|_\infty + \left\|k_0 \partial_k \phi(\cdot;k_0)\right\|_\infty \|\gamma_\zeta(t)\|_\infty \leq \frac{ME_0}{\sqrt{1+t}},
\end{align*}
for $t \geq 0$ and $M \geq M_0(1+\|k_0\partial_k \phi(\cdot;k_0)\|_\infty)$, which establishes~\eqref{e:mtest11}. Similarly, using Taylor's theorem, we estimate
\begin{align*}
\left\|u\left(\cdot-\gamma(\cdot,t),t\right)-\phi\left(\cdot;k_0(1+\gamma_\xx(\cdot,t)\right)\right\|_\infty 
&\leq \|z(t)\|_\infty + \frac{k_0^2}{2}\sup_{k \in U} \left\|\partial_{kk} \phi(\cdot;k)\right\|_\infty \|\gamma_\xx(t)\|^2_\infty\\ 
&\leq \frac{ME_0 \log(2+t)}{1+t}
\end{align*}
for $t \geq 0$ and $M \geq M_0(1 + k_0^2\sup_{k \in U} \|\partial_{kk} \phi(\cdot;k)\|_\infty)$, which yields~\eqref{e:mtest12}. 

All that remains is to prove the key inequality~\eqref{e:etaest} and then verify estimate~\eqref{e:mtest3}. In the following we bound the terms arising in the template function $\eta(t)$ one by one.
\bigskip\\
\noindent\textbf{Bounds on $z(t)$, $\gamma_{\xx\xx}(t)$ and $\widetilde{\gamma}(t)$.} 
Let $t \in [0,\tau_{\max})$ with $\eta(t) \leq \frac{1}{2}$. First, we observe that $v(s) = z(s) + k_0 \partial_k \phi(\cdot;k_0) \gamma_\xx(s)$ and $\partial_t \gamma(s) = \widetilde{\gamma}(s) + a\gamma_\xx(s)$ can be bounded as
\begin{align}
\begin{split}
\|\partial_\xx^m v(s)\|_\infty &\lesssim \|\partial_\xx^m z(s)\|_\infty + \|\gamma_\xx(s)\|_{W^{2,\infty}}, \qquad \|\partial_t \gamma(s)\|_\infty \lesssim \|\widetilde{\gamma}(s)\|_\infty + \|\gamma_\xx(s)\|_\infty.
\end{split}
\label{e:vbound} \end{align}
for $m = 0,1,2$ and $s \in (0,t]$. So, employing Lemma~\ref{lem:nlboundsmod} we obtain
\begin{align} \label{e:nlest100}
\|\mathcal Q(v(s),\gamma(s))\|_\infty, \|\mathcal R(v(s),\gamma(s),\partial_t \gamma(s))\|_\infty, \|\mathcal S(v(s),\gamma(s))\|_\infty &\lesssim (1+s)^{-1}\eta(s)^2,
\end{align}
and
\begin{align*}
\|\mathcal S(v(s),\gamma(s))\|_{W^{1,\infty}} &\lesssim \frac{\log(2+s)}{1+s}\eta(s)^2,
\end{align*}
for $s \in [0,t]$, where we use $\eta(t) \leq \frac{1}{2}$ and the fact that, if $s \leq 2t_* \leq 1$, then we have $\gamma(s) \equiv 0$ by Proposition~\ref{p:gamma}. Hence, applying Proposition~\ref{prop:semexp} we arrive at
\begin{align} \label{e:nlest1}
\begin{split}
&\left\|\int_0^t S_e(t-s) \mathcal{N}\left(v(s),\gamma(s),\partial_t \gamma(s)\right) \de s\right\|_{\infty} \lesssim \int_0^t \left\|S_e(t-s)\mathcal Q(v(s),\gamma(s))\right\|_\infty \de s \\ &\qquad\qquad\qquad +\, \int_0^t \left\|S_e(t-s) \partial_\xx \left(\mathcal{R}\left(v(s),\gamma(s),\partial_t \gamma(s)\right) + \partial_\xx \mathcal S(v(s),\gamma(s))\right) \right\|_{\infty} \de s\\
&\qquad \lesssim \int_0^t \frac{\eta(s)^2}{(1+t-s)^{\frac{11}{10}} (1+s)} \de s + \int_0^t \frac{\log(2+s)\eta(s)^2}{(1+t-s)^{\frac{11}{10}} (1+s)}\left(1+\frac{1}{\sqrt{t-s}}\right) \de s\\ 
&\qquad \lesssim \frac{\eta(t)^2\log(2+t)}{1+t}.
\end{split}
\end{align}
for all $t \in [0,\tau_{\max})$ with $\eta(t) \leq \frac{1}{2}$. On the other hand, Proposition~\ref{prop:semdif} and estimate~\eqref{e:nlest100} yield
\begin{align}
\left\|\int_0^t S_r(t-s) \mathcal{N}\left(v(s),\gamma(s),\partial_t \gamma(s)\right) \de s\right\|_\infty & \lesssim \int_0^t \frac{\eta(s)^2}{(1+t-s)(1+s)} \de s \lesssim \frac{\eta(t)^2\log(2+t)}{1+t},\label{e:nlest2}
\end{align}
and 
\begin{align}
\begin{split}
\left\|(\partial_t - a\partial_\xx)^j \partial_\xx^l \int_0^t S_p^0(t-s) \mathcal{N}\left(v(s),\gamma(s),\partial_t \gamma(s)\right) \de s\right\|_\infty & \lesssim \int_0^t \frac{\eta(s)^2}{(1+t-s)(1+s)} \de s\\
&\lesssim \frac{\eta(t)^2\log(2+t)}{1+t},\end{split} \label{e:nlest22}
\end{align}
for all $t \in [0,\tau_{\max})$ with $\eta(t) \leq \frac{1}{2}$ and $j,l \in \mathbb{N}_0$ with $2 \leq l + 2j \leq 4$, where we used $S_p^0(0) = 0$ when taking the temporal derivative. Thus, using Propositions~\ref{prop:semexp} and~\ref{prop:semdif}, the decomposition~\eqref{e:semdecomp} of $\smash{\widetilde{S}(t)}$ and estimates~\eqref{e:vbound},~\eqref{e:nlest1},~\eqref{e:nlest22} and~\eqref{e:nlest2}, we bound the right-hand sides of~\eqref{e:intgamma} and~\eqref{e:intz} and obtain
\begin{align} \label{e:nlest3}
\begin{split}
\|(\gamma_{\xx\xx}(t),\widetilde{\gamma}(t))\|_{W^{2,\infty} \times W^{2,\infty}}, \|z(t)\|_\infty \lesssim \left(E_0 + \eta(t)^2\right)\frac{\log(2+t)}{1+t},
\end{split}
\end{align}
for all $t \in [0,\tau_{\max})$ with $\eta(t) \leq \frac{1}{2}$.
\bigskip\\
\noindent\textbf{Bounds on $r(t)$ and $r_\xx(t)$.} Let $t \in [0,\tau_{\max})$ with $\eta(t) \leq \frac{1}{2}$. We employ Lemma~\ref{lem:nlboundsmod3} and estimate~\eqref{e:vbound} to bound
\begin{align}
\label{e:nlest71}
\begin{split}
&\|\mathcal Q_p(z(s),v(s),\gamma(s))\|_\infty, \|\mathcal R_p(z(s),v(s),\gamma(s),\partial_t \gamma(s))\|_\infty, \|\mathcal S_p(z(s),v(s),\gamma(s))\|_\infty\\ 
&\qquad \lesssim \frac{\log(2+s)}{(1+s)^{\frac{3}{2}}}\eta(s)^2,
\end{split}
\end{align}
for $s \in [0,t]$. Hence, applying Proposition~\ref{prop:semdif} we bound
\begin{align} \label{e:nlest7}
\begin{split}
\left\|\partial_\xx^{m+1} \int_0^t S_p^0(t-s) \mathcal R_p(z(s),v(s),\gamma(s),\widetilde{\gamma}(s))\de s \right\|_\infty &\lesssim \int_0^t \frac{\eta(s)^2\log(2+s)}{(1+t-s)^{\frac{1+m}{2}} (1+s)^{\frac{3}{2}}} \de s\\ &\lesssim \frac{\eta(t)^2}{(1+t)^{\frac{1+m}{2}}},
\end{split}
\end{align}
and, analogously,
\begin{align} \label{e:nlest78}
\begin{split}
\left\|\partial_\xx^{m+2-i} \int_0^t S_p^i(t-s) \mathcal S_p(z(s),v(s),\gamma(s))\de s \right\|_\infty &\lesssim \frac{\eta(t)^2}{(1+t)^{\frac{1+m}{2}}},
\end{split}
\end{align}
for all $t \in [0,\tau_{\max})$ with $\eta(t) \leq \frac{1}{2}$ and $i,m \in \{0,1\}$. In addition, with the aid of Proposition~\ref{prop:semiprinrec} and estimate~\eqref{e:nlest71} we establish
\begin{align} \label{e:nlest77}
\begin{split}
\left\|\partial_\xx^m \int_0^t \widetilde{S}_r^0(t-s) \mathcal Q_p(z(s),v(s),\gamma(s)) \de s\right\|_\infty &\lesssim \int_0^t \frac{\eta(s)^2 \log(2+s) }{(t-s)^{\frac{m}{2}}\sqrt{1+t-s}(1+s)^{\frac{3}{2}}} \de s\\ 
&\lesssim \frac{\eta(t)^2}{(1+t)^{\frac{1+m}{2}}},
\end{split}
\end{align}
and, analogously,
\begin{align} \label{e:nlest79}
\begin{split}
&\left\|\partial_\xx^m \int_0^t \widetilde{S}_r^1(t-s) \mathcal R_p(z(s),v(s),\gamma(s),\widetilde{\gamma}(s)) \de s\right\|_\infty, \left\|\partial_\xx^m \int_0^t \widetilde{S}_r^2(t-s) \mathcal S_p(z(s),v(s),\gamma(s)) \de s\right\|_\infty\\ &\qquad \lesssim \frac{\eta(t)^2}{(1+t)^{\frac{1+m}{2}}},
\end{split}
\end{align}
for $m = 0,1$ and all $t \in [0,\tau_{\max})$ with $\eta(t) \leq \frac{1}{2}$. Moreover, Proposition~\ref{prop:semiprinrec} yields
\begin{align} \label{e:nlest4}
\begin{split}
\left\|\partial_\xx^m \int_0^t \widetilde{S}_r^0(t-s)\left(f_p \gamma_\xx(s)^2\right) \de s\right\|_\infty &\lesssim \int_0^t \frac{\eta(s)^2}{\sqrt{1+t-s} (t-s)^{\frac{m}{2}} (1+s)} \de s \lesssim \frac{\eta(t)^2 \log(2+t)}{(1+t)^{\frac{1+m}{2}}},
\end{split}
\end{align}
for $m = 0,1$ and all $t \in [0,\tau_{\max})$. Next, we apply Propositions~\ref{prop:semiprinrec} and~\ref{prop:semiapp} to establish
\begin{align} \label{e:nlest5}
\begin{split}
\left\|\partial_\xx \int_0^t \re^{\left(d\partial_\xx^2 + a\partial_\xx\right) (t-s)}\left(A_h(f_p) \gamma_\xx(s)^2\right) \de s\right\|_\infty &\lesssim \int_0^t \frac{\eta(s)^2}{\sqrt{t-s} (1+s)} \de s \lesssim \frac{\eta(t)^2\log(2+t)}{\sqrt{1+t}}.
\end{split}
\end{align}
for all $t \in [0,\tau_{\max})$. Similarly, using that $\partial_\xx$ commutes with $\smash{\re^{\left(d\partial_\xx^2 + a\partial_\xx\right) (t-s)}}$, we obtain
\begin{align} \label{e:nlest6}
\begin{split}
&\left\|\partial_\xx^2 \int_0^t \re^{\left(d\partial_\xx^2 + a\partial_\xx\right) (t-s)}\left(A_h(f_p) \gamma_\xx(s)^2\right) \de s\right\|_\infty\\
&\qquad \lesssim \int_{0}^{\max\{0,t-1\}} \frac{\eta(s)^2}{(t-s)(1+s)} \de s + \int_{\max\{0,t-1\}}^t \frac{\eta(s)^2}{\sqrt{t-s}(1+s)} \de s \lesssim \frac{\eta(t)^2\log(2+t)}{1+t},
\end{split}
\end{align}
for all $t \in [0,\tau_{\max})$. Thus, using the estimates~\eqref{e:nlest7},~\eqref{e:nlest78},~\eqref{e:nlest77},~\eqref{e:nlest79},~\eqref{e:nlest4},~\eqref{e:nlest5} and~\eqref{e:nlest6} and Proposition~\ref{prop:semiprinrec}, we bound the right-hand side of~\eqref{e:intr} by
\begin{align} \label{e:nlest8}
\begin{split}
t^{\frac{m}{2}} \|\partial_\xx^m r(t)\|_\infty &\lesssim \left(E_0 + \eta(t)^2\right)\frac{\log(2+t)}{\sqrt{1+t}},
\end{split}
\end{align}
for $m = 0,1$ and all $t \in [0,\tau_{\max})$ with $\eta(t) \leq \frac{1}{2}$.
\bigskip\\
\noindent\textbf{Bounds on $y(t)$ and $y_\xx(t)$.} We use the identity~\eqref{e:ysmall}, the estimate~\eqref{e:nlest8} and $\eta(t) \leq \frac12$ to establish the short-time bound
\begin{align} \label{e:yshort}
t^{\frac{m}{2}}\|\partial_\xx^m y(t)\|_\infty \lesssim t^{\frac{m}{2}}\|\partial_\xx^m r(t)\|_\infty \lesssim E_0 + \eta(t)^2,
\end{align}
for $m = 0,1$ and all $t \in [0,\tau_{\max})$ with $t \leq 1$ and $\eta(t) \leq \frac{1}{2}$. 

Now let $t \in [0,\tau_{\max})$ with $t \geq 1$ and $\eta(t) \leq \frac{1}{2}$. We first apply Lemma~\ref{lem:nlboundsmod4} and estimate~\eqref{e:vbound} to establish
\begin{align*}
\| \Non_c(r(s),y(s),z(s),v(s),\gamma(s),\widetilde{\gamma}(s))\|_\infty \lesssim \frac{\eta(s)^2 \log(2+s)}{\left(1+s\right)^{\frac{3}{2}}},
\end{align*}
for $s \in [1,t]$, where we use $\eta(t) \leq \frac{1}{2}$. So, using Proposition~\ref{prop:semiprinrec} and~\eqref{e:yshort}, we bound
\begin{align} \label{e:nlest91} 
\begin{split}
\left\|\partial_\xx^m \re^{\left(d\partial_\xx^2 + a\partial_\xx\right) (t-1)} y(1)\right\|_\infty \lesssim \frac{\|y(1)\|_{W^{m,\infty}}}{(1+t)^{\frac{m}{2}}} \lesssim \frac{E_0 + \eta(t)^2}{(1+t)^{\frac{m}{2}}},
\end{split}
\end{align}
and
\begin{align} \label{e:nlest92}
\begin{split}
\left\|\partial_\xx^m \int_1^t \re^{\left(d\partial_\xx^2 + a\partial_\xx\right) (t-s)} \Non_c(r(s),y(s),z(s),v(s),\gamma(s),\widetilde{\gamma}(s)) \de s\right\|_\infty &\lesssim \int_1^t \frac{\eta(s)^2 \log(2+s)}{(t-s)^{\frac{m}{2}} (1+s)^{\frac{3}{2}}} \de s\\ &\lesssim \frac{\eta(t)^2}{(1+t)^{\frac{m}{2}}}.
\end{split}
\end{align}
for $m = 0,1$ and all $t \in [0,\tau_{\max})$ with $t \geq 1$ and $\eta(t) \leq \frac{1}{2}$. Combining~\eqref{e:nlest91} and~\eqref{e:nlest92} with the short-time bound~\eqref{e:yshort} yields
\begin{align} \label{e:nlest9}
\begin{split}
t^{\frac{m}{2}}\|\partial_\xx^m y(t)\|_\infty \lesssim E_0 + \eta(t)^2,
\end{split}
\end{align}
for $m = 0,1$ and all $t \in [0,\tau_{\max})$ with $\eta(t) \leq \frac{1}{2}$.
\bigskip\\
\noindent\textbf{Bounds on $\gamma(t)$ and $\gamma_\xx(t)$.} First, we consider the case $\nu \neq 0$. Taking the spatial derivative of
\begin{align*} \gamma(t) = r(t) + \frac{d}{\nu} \log(y(t) + 1),\end{align*}
yields
\begin{align*} \gamma_\xx(t) = r_\xx(t) + \frac{d y_\xx(t)}{\nu (1+y(t))},\end{align*}
for $t \in (0,\tau_{\max})$. We note that, since we have $\|y(t)\|_\infty \leq \eta(t) \leq \frac{1}{2}$ and $\nu \neq 0$, the above expressions are well-defined and it holds $\|\partial_\xx^m \gamma(t)\|_\infty \lesssim \|\partial_\xx^m r(t)\|_\infty + \|\partial_\xx^m y(t)\|_\infty$ for $t \in (0,\tau_{\max})$. Hence, using the estimates~\eqref{e:nlest8} and~\eqref{e:nlest9} and the fact that $\gamma(s) \equiv 0$ for $s \in [0,\tau_{\max})$ with $s \leq 1$ by Proposition~\ref{p:gamma}, we obtain
\begin{align} \label{e:nlest11} \|\partial_\xx^m \gamma(t)\|_\infty \lesssim \frac{E_0 + \eta(t)^2}{(1+t)^{\frac{m}{2}}},\end{align}
for $m = 0,1$ and $t \in [0,\tau_{\max})$ with $\eta(t) \leq \frac12$. 

Next, we consider the case $\nu = 0$. With the aid of Proposition~\ref{prop:semiprinrec} and estimate~\eqref{e:nlest71} we establish
\begin{align} \label{e:nlest70}
\begin{split}
\left\|\partial_\xx^m S_h^0(t)v_0\right\|_\infty &\lesssim \frac{E_0}{(1+t)^{\frac{m}{2}}},\\
\left\|\partial_\xx^m \int_0^t S_h^0(t-s) \mathcal Q_p(z(s),v(s),\gamma(s))\de s\right\|_\infty & \lesssim \int_0^t \frac{\eta(s)^2\log(2+s)}{(t-s)^{\frac{m}{2}} (1+s)^{\frac{3}{2}}} \de s \lesssim \frac{\eta(t)^2}{(1+t)^{\frac{m}{2}}},
\end{split}
\end{align}
and, analogously,
\begin{align} \label{e:nlest72}
\begin{split}
&\left\|\partial_\xx^m\! \int_0^t S_h^1(t-s) \mathcal R_p(z(s),v(s),\gamma(s),\widetilde{\gamma}(s))\de s\right\|_\infty, 
\left\|\partial_\xx^m\! \int_0^t S_h^2(t-s) \mathcal S_p(z(s),v(s),\gamma(s))\de s\right\|_\infty, \\
&\left\|\partial_\xx^m \int_0^t \re^{\left(d\partial_\xx^2 + a\partial_\xx\right) (t-s)} \left(A_h(f_p) \partial_\xx \left(\gamma_\xx(s)^2\right)\right) \de s\right\|_\infty \lesssim \frac{\eta(t)^2}{(1+t)^{\frac{m}{2}}},
\end{split}
\end{align}
for $m = 0,1$ and $t \in [0,\tau_{\max})$ with $t \geq 1$ and $\eta(t) \leq \frac12$. Thus, recalling that $\gamma(t)$ vanishes on $[0,1]$ by Proposition~\ref{p:gamma} and using estimates~\eqref{e:nlest8},~\eqref{e:nlest70} and~\eqref{e:nlest72}, we bound the right-hand side of~\eqref{e:intgamma2} and arrive at~\eqref{e:nlest11}.
\bigskip\\
\noindent\textbf{Bound on $z_{\xx\xx}(t)$.} Due to the fact that the cutoff function $\varrho$ vanishes on $[0,t_*]$, it suffices to focus on the case $t \geq t_*$. Thus, let $t \in [0,\tau_{\max})$ with $t \geq t_*$ and $\eta(t) \leq \frac12$. Thanks to~\eqref{e:shorttime} we find that~\eqref{e:defv2} is three times differentiable. Thus, we compute
\begin{align*}
v_{\xx\xx\xx}(\xx,s) &= \left(\vt_{\xx\xx\xx}(\xx-\gamma(\xx,s),s) + \phi_0'''(\xx-\gamma(\xx,s))\right)\left(1 - \gamma_\xx(\xx,s)\right)^3\\ 
&\qquad -\, 3\left(\vt_{\xx\xx}(\xx-\gamma(\xx,s),s)+\phi_0''(\xx-\gamma(\xx,s))\right)\gamma_{\xx\xx}(\xx,s)(1-\gamma_\xx(\xx,s))\\
&\qquad -\,\left(\vt_\xx(\xx-\gamma(\xx,s),s)+ \phi_0'(\xx-\gamma(\xx,s))\right)\gamma_{\xx\xx\xx}(\xx,s) - \phi_0'''(\xx),
\end{align*}
for $s \in [t_*,t]$ and $\xx \in \R$. With the aid of the mean value theorem we bound the latter as
\begin{align*}
\|v_{\xx\xx\xx}(s)\|_\infty \lesssim \|\vt_{\xx}(s)\|_{W^{2,\infty}} + \|\gamma(s)\|_{W^{3,\infty}} \lesssim  \sqrt{(1+s)\log(2+s)}\eta(s),
\end{align*}
for $s \in [2t_*,t]$, where we use $\eta(t) \leq \frac{1}{2}$. So, employing Lemma~\ref{lem:nlboundsmod} and estimate~\eqref{e:vbound}, we establish
\begin{align*}
\left\|\mathcal{Q}(v(s),\gamma(s))\right\|_{W^{1,\infty}} & \lesssim \frac{\eta(s)^2}{\sqrt{s}\sqrt{1+s}},\\
\left\| \mathcal{R}(v(s),\gamma(s),\partial_s\gamma(s))\right\|_{W^{2,\infty}} & \lesssim \frac{\log(2+s) \eta(s)^2}{\sqrt{1+s}},\\
\left\|\mathcal{S}(v(s),\gamma(s))\right\|_{W^{3,\infty}} &\lesssim \sqrt{\log(2+s)} \eta(s)^2,
\end{align*}
for $s \in (0,t]$, where we use $\eta(t) \leq \frac{1}{2}$ and the fact that, if $s \leq 2t_* \leq 1$, then we have $\gamma(s) \equiv 0$ by Proposition~\ref{p:gamma}. Hence, using~\eqref{e:nlest100} and applying the second estimate in Corollary~\ref{cor:der}, we find
\begin{align} \label{e:nlesta}
\begin{split}
&\left\|\int_0^t \partial_\xx^2 \re^{\El_0(t-s)} \mathcal{N}\left(v(s),\gamma(s),\partial_t \gamma(s)\right) \de s\right\|_{\infty}\\ 
&\qquad \lesssim \int_0^t \left(\left\|\mathcal{Q}(v(s),\gamma(s))\right\|_{\infty} + \left\|\mathcal{R}(v(s),\gamma(s),\partial_t \gamma(s))\right\|_{\infty} + \left\|\mathcal{S}(v(s),\gamma(s))\right\|_{\infty} \phantom{\left(\frac{1}{\sqrt{t}}\right)} \right.\\
&\qquad \qquad \left. + \,
\frac{\left\|\mathcal{Q}(v(s),\gamma(s))\right\|_{W^{1,\infty}} + \left\|\mathcal{R}(v(s),\gamma(s),\partial_t \gamma(s))\right\|_{W^{2,\infty}} + \left\|\mathcal{S}(v(s),\gamma(s))\right\|_{W^{3,\infty}}}{\left(1+t-s\right)^{\frac{11}{10}}}\right.\\
&\qquad \qquad \left. \cdot \, \left(1+\frac{1}{\sqrt{t-s}}\right)\right) \de s\\
&\qquad \lesssim \int_0^t \eta(s)^2\left(\frac{1}{1+s} + \frac{\sqrt{\log(2+s)}}{\left(1+t-s\right)^{\frac{11}{10}}}\left(1+\frac{1}{\sqrt{t-s}}\right)\left(1+\frac{1}{\sqrt{s}}\right)\right) \de s \lesssim \log(2+t)\eta(t)^2,
\end{split}
\end{align}
for $t \in [0,\tau_{\max})$ with $t \geq t_*$ and $\eta(t) \leq \frac12$. Moreover, the first estimate in Corollary~\ref{cor:der} yields
\begin{align} \label{e:nlestaa}
\left\|\partial_\xx^2 \re^{\El_0 t} v_0\right\|_\infty \lesssim E_0, 
\end{align}
for $t \in [0,\tau_{\max})$ with $t \geq t_*$. On the other hand, with the aid of Proposition~\ref{prop:semdif} and estimate~\eqref{e:nlest100} we bound
\begin{align}
\begin{split}
\left\|\partial_\xx^2 \left(\left(\phi_0' + k_0 \partial_k \phi(\cdot;k_0)\right) S_p^0(t) v_0\right)\right\|_\infty &\lesssim E_0,\\
\left\|\partial_\xx^2 \int_0^t S_p^0(t-s) \mathcal{N}\left(v(s),\gamma(s),\partial_t \gamma(s)\right) \de s\right\|_\infty &\lesssim  \int_0^t \frac{\eta(s)^2}{1+s} \de s \lesssim \log(2+t)\eta(t)^2,
\end{split} \label{e:nlestb}
\end{align}
for $t \in [0,\tau_{\max})$ with $t \geq t_*$ and $\eta(t) \leq \frac12$. Thus, recalling the semigroup decomposition~\eqref{e:semidecomp} and combining estimates~\eqref{e:vbound},~\eqref{e:nlesta},~\eqref{e:nlestaa} and~\eqref{e:nlestb}, we bound the right-hand side of~\eqref{e:intz} and obtain
\begin{align} \label{e:nlestc}
\begin{split}
\varrho(t)\|z_{\xx\xx}(t)\|_\infty &\lesssim \log(2+t)\left(E_0 + \eta(t)^2\right).
\end{split}
\end{align}
for $t \in [0,\tau_{\max})$ with $\eta(t) \leq \frac12$.
\bigskip\\
\noindent\textbf{Bound on $z_{\xx}(t)$.} Interpolating between estimates~\eqref{e:nlest3} and~\eqref{e:nlestc} yields
\begin{align} \varrho(t)\|z_\xx(t)\|_\infty &\lesssim \frac{\log(2+t)}{\sqrt{1+t}}\left(E_0 + \eta(t)^2\right). \label{e:nlestd}\end{align}
for $t \in [0,\tau_{\max})$ with $\eta(t) \leq \frac12$.
\bigskip\\
\noindent\textbf{Bounds on $\vt(t)$, $\vt_\xx(t)$ and $\vt_{\xx\xx}(t)$.} Since the cutoff function $\varrho$ vanishes on $[0,t_*]$ and we have established the short-time bound~\eqref{e:shorttime}, it suffices to focus on the case $t \geq t_*$. Thus, let $t \in [0,\tau_{\max})$ with $t \geq t_*$ and $\eta(t) \leq \frac12$. As in the proof of Corollary~\ref{cor:psit} we note that $\|\gamma_\xx(t)\|_\infty \leq \eta(t) \leq \frac{1}{2}$ implies that the map $\psi_t \colon \R \to \R$ given by $\psi_t(\xx) = \xx - \gamma(\xx,t)$ is invertible. We rewrite $\psi_t(\psi_t^{-1}(\xx)) = \xx$ as $\psi_t^{-1}(\xx) - \xx = \gamma(\psi_t^{-1}(\xx),t)$ and estimate
\begin{align} 
\label{e:psiest}
\sup_{\xx \in \R} \left|\psi_t^{-1}(\xx) - \xx\right| &\leq \|\gamma(t)\|_\infty.\end{align}
Substituting $\xx$ by $\psi_t^{-1}(\xx)$ in~\eqref{e:defv2} and its spatial derivatives
\begin{align*} 
v_{\xx}(\xx,t) &= \left(\vt_{\xx}(\xx-\gamma(\xx,t),t) + \phi_0'(\xx-\gamma(\xx,t))\right)\left(1 - \gamma_\xx(\xx,t)\right) - \phi_0'(x),\\
v_{\xx\xx}(\xx,t) &= \left(\vt_{\xx\xx}(\xx-\gamma(\xx,t),t) + \phi_0''(\xx-\gamma(\xx,t))\right)\left(1 - \gamma_\xx(\xx,t)\right)^2 \\ 
&\qquad - \, \left(\vt_{\xx}(\xx-\gamma(\xx,t),t)+\phi_0'(\xx - \gamma(\xx,t))\right)\gamma_{\xx\xx}(\xx,t) - \phi_0''(x),\end{align*}
yields, after rearranging terms,
\begin{align*} 
\vt(\xx,t) &= v\left(\psi_t^{-1}(\xx),t\right) + \phi_0\left(\psi_t^{-1}(\xx)\right) - \phi_0(\xx),\\
\vt_{\xx}(\xx,t) &= \frac{v_\xx\left(\psi_t^{-1}(\xx),t\right) + \phi_0'\left(\psi_t^{-1}(\xx)\right)}{1-\gamma_\xx\left(\psi_t^{-1}(\xx),t\right)} - \phi_0'(\xx),\\
\vt_{\xx\xx}(\xx,t) &= \frac{v_{\xx\xx}\left(\psi_t^{-1}(\xx),t\right) + \phi_0''\left(\psi_t^{-1}(\xx)\right) + \left(\vt_\xx(\xx,t) + \phi_0'(\xx)\right)\gamma_{\xx\xx}\left(\psi_t^{-1}(\xx),t\right)}{\left(1-\gamma_\xx\left(\psi_t^{-1}(\xx),t\right)\right)^2} - \phi_0''(\xx).\end{align*}
We apply the mean value theorem and estimates~\eqref{e:vbound},~\eqref{e:nlest3},~\eqref{e:nlest11},~\eqref{e:nlestc},~\eqref{e:nlestd} and~\eqref{e:psiest} to the latter and arrive at
\begin{align*} 
\begin{split}
\rho(t)\|\partial_\xx^j \vt(t)\|_\infty &\lesssim \rho(t)\|v(t)\|_{W^{j,\infty}} + \|\gamma(t)\|_{W^{j,\infty}} \lesssim E_0 + \eta(t)^2,\\
\rho(t)\left\|\vt_{\xx\xx}(t)\right\|_\infty &\lesssim \rho(t)\|v(t)\|_{W^{2,\infty}} + \|\gamma(t)\|_{W^{2,\infty}} \lesssim \log(2+t)\left(E_0 + \eta(t)^2\right),
\end{split}
\end{align*}
for $j = 0,1$ and $t \in [0,\tau_{\max})$ with $\eta(t) \leq \frac12$. Combining the latter with the short-time bound~\eqref{e:shorttime} we obtain
\begin{align} 
\begin{split}
\left(\frac{t}{1+t}\right)^{\frac{j}{2}}\|\partial_\xx^j \vt(t)\|_\infty \lesssim E_0 + \eta(t)^2, \qquad \varrho(t)\|\vt_{\xx\xx}(t)\|_\infty \lesssim \log(2+t)\left(E_0 + \eta(t)^2\right),
\end{split}
\label{e:nlest12} \end{align}
for $j = 0,1$ and $t \in [0,\tau_{\max})$ with $\eta(t) \leq \frac12$.
\bigskip\\
\noindent\textbf{Bounds on $\partial_\xx^4 \vt(t)$.} Since the cutoff function $\varrho$ vanishes on $[0,t_*]$, it suffices to focus on the case $t \geq t_*$. Thus, let $t \in [0,\tau_{\max})$ with $t \geq t_*$ and $\eta(t) \leq \frac12$. We use Lemma~\ref{lem:nlboundsunmod} to bound
\begin{align}\label{e:nlest13}
\begin{split}
\left\|\widetilde{\mathcal N}(\vt(s))\right\|_{\infty} &\lesssim \eta(s)^2, \qquad \left\|\widetilde{\mathcal N}(\vt(s))\right\|_{W^{3,\infty}} \lesssim \sqrt{(1+s)\log(2+s)} \eta(s)^2,
\end{split}
\end{align}
for $s \in [t_*,t]$. First, using the last estimate in Corollary~\ref{cor:der} and the short-time bound~\eqref{e:shorttime} we estimate the linear term on the right-hand side of the Duhamel formula
\begin{align}
\vt(t) = \re^{\El_0 (t-t_*)} \vt(t_*) + \int_{t_*}^t \re^{\El_0(t-s)}\widetilde{\mathcal{N}}(\vt(s))\de s, \label{e:intvt2}
\end{align}
as
\begin{align} \label{e:nlest16}
\left\|\partial_\xx^4 \re^{\El_0(t-t_*)} \vt(t_*)\right\|_\infty \lesssim \|\vt(t_*)\|_{W^{4,\infty}} \lesssim E_0,
\end{align}
for $t \in [0,\tau_{\max})$ with $t \geq t_*$ and $\eta(t) \leq \frac12$. Second, the nonlinear bounds~\eqref{e:nlest13} and the second estimate in Corollary~\ref{cor:der} yield
\begin{align} \label{e:nlest15}
\begin{split}
\left\|\partial_\xx^4\! \int_{t_*}^t \re^{\El_0 (t-s)} \widetilde{\mathcal{N}}(\vt(s))\de s\right\|_\infty &\lesssim \int_{t_*}^t \eta(s)^2\left(1 + \frac{\sqrt{(1+s)\log(2+s)}}{(1+t-s)^{\frac{11}{10}}} \left(1+\frac{1}{\sqrt{t-s}}\right) \right) \de s\\ 
&\lesssim \eta(t)^2 (1+t),
\end{split}
\end{align}
for $t \in [0,\tau_{\max})$ with $t \geq t_*$ and $\eta(t) \leq \frac12$. Thus, using estimates~\eqref{e:nlest16} and~\eqref{e:nlest15}, we bound the right-hand side of~\eqref{e:intvt2} and obtain
\begin{align} \label{e:nlest17}
\begin{split}
\varrho(t)\left\|\partial_\xx^4 \vt(t)\right\|_\infty &\lesssim (1+t)\left(E_0 + \eta(t)^2\right),
\end{split}
\end{align}
for $t \in [0,\tau_{\max})$ with $\eta(t) \leq \frac12$.
\bigskip\\
\noindent\textbf{Bound on $\partial_\xx^3 \vt(t)$.} Interpolating between~\eqref{e:nlest12} and~\eqref{e:nlest17} we readily arrive at
\begin{align} \label{e:nlest18}
\begin{split}
\varrho(t)\left\|\partial_\xx^3 \vt(t)\right\|_\infty \lesssim \sqrt{(1+t)\log(2+t)} \left(E_0 + \eta(t)^2\right) .
\end{split}
\end{align}
for $t \in [0,\tau_{\max})$ with $t \geq t_*$ and $\eta(t) \leq \frac12$.
\bigskip\\
\noindent\textbf{Proof of key inequality.} By the estimates~~\eqref{e:nlest3},~\eqref{e:nlest8},~\eqref{e:nlest9},~\eqref{e:nlest11},~\eqref{e:nlestc},~\eqref{e:nlestd},~\eqref{e:nlest12},~\eqref{e:nlest17} and~\eqref{e:nlest18} it follows that there exists a constant $C \geq 1$, which is independent of $E_0$ and $t$, such that the key inequality~\eqref{e:etaest} is satisfied.
\bigskip\\
\noindent\textbf{Approximation by the viscous Hamilton-Jacobi equation.} We distinguish between the cases $\nu = 0$ and $\nu \neq 0$. In case $\nu = 0$, we define $\breve \gamma(t) = S_h^0(t) v_0 = \smash{\re^{(d\partial_\xx^2 + a\partial_\xx)t}} \widetilde{\Phi}_0^* v_0$. Clearly, $\breve \gamma \in C\big([0,\infty),C_{\mathrm{ub}}(\R)\big) \cap C\big((0,\infty),C_{\mathrm{ub}}^2(\R)\big) \cap C^1\big((0,\infty),C_{\mathrm{ub}}(\R)\big)$ is a classical solution to~\eqref{e:HamJac} having initial condition $\breve \gamma(0) = \smash{\widetilde{\Phi}_0^* v_0} \in C_{\mathrm{ub}}(\R)$. Noting that $\gamma(t) \equiv 0$ for $t \in [0,1]$ by Proposition~\ref{p:gamma}, we obtain by Proposition~\ref{prop:semiprinrec} a constant $C_1 \geq 1$ such that
\begin{align} \label{e:mtest200}
\begin{split}
t^{\frac{m}{2}}\left\|\partial_\xx^m \left(\gamma(t) - \breve{\gamma}(t)\right)\right\|_\infty = t^{\frac{m}{2}}\left\|\partial_\xx^m \breve{\gamma}(t)\right\|_\infty \leq \frac{C_1E_0}{\sqrt{1+t}}, 
\end{split}
\end{align}
for $t \in [0,1]$ and $m = 0,1$. On the other hand, recalling $\eta(t) \leq M_0E_0$ and using identity~\eqref{e:intgamma2} and estimates~\eqref{e:nlest70} and~\eqref{e:nlest72} we establish constants $C_2,C_3 \geq 1$ such that
\begin{align} \label{e:mtest21}
\begin{split}
\left\|\partial_\xx^m\!\left(\gamma(t) - \breve{\gamma}(t)\right)\right\|_\infty \leq C_2\left(\frac{E_0^2 + \eta(t)^2}{(1+t)^{\frac{m}{2}}} +  \left\|\partial_\xx^m r(t)\right\|_\infty\!\right) \leq \frac{C_3 E_0}{(1+t)^{\frac{m}{2}}} \left(E_0 + \frac{\log(2+t)}{\sqrt{1+t}}\right), 
\end{split}
\end{align}
for $t \geq 1$ and $m = 0,1$. Thus, combining~\eqref{e:mtest200} and~\eqref{e:mtest21} yields~\eqref{e:mtest3} upon taking $M \geq \max\{C_1,C_3\}$.

Next, we consider the case $\nu \neq 0$. Taking $E_0 = \|v_0\|_{\infty}$ sufficiently small, we observe that the functions $\breve{\gamma}, \breve{y} \in C\big([0,\infty),C_{\mathrm{ub}}(\R)\big) \cap C\big((0,\infty),C_{\mathrm{ub}}^2(\R)\big) \cap C^1\big((0,\infty),C_{\mathrm{ub}}(\R)\big)$ given by
\begin{align*} \breve \gamma(t) = \frac{d}{\nu}\log\left(1 + \breve{y}(t)\right) \qquad \breve{y}(t) = \re^{(d\partial_\xx^2 + a\partial_\xx)t}\left(\re^{\frac{\nu}{d}\widetilde{\Phi}_0^* v_0} - 1\right),\end{align*}
are well-defined and it holds $\|\breve{y}(t)\|_\infty \leq \frac12$ for $t \geq 0$, where we use the standard bounds on the analytic semigroup $\smash{\re^{(d\partial_\xx^2 + a\partial_\xx)t}}$ established in~Propositions~\ref{prop:semiprinrec}. One readily verifies that $\breve{\gamma} \in C\big([0,\infty),C_{\mathrm{ub}}(\R)\big) \cap C\big((0,\infty),C_{\mathrm{ub}}^2(\R)\big) \cap C^1\big((0,\infty),C_{\mathrm{ub}}(\R)\big)$ is a classical solution to the viscous Hamilton-Jacobi equation~\eqref{e:HamJac} with initial condition $\breve \gamma(0) = \smash{\widetilde{\Phi}_0^* v_0} \in C_{\mathrm{ub}}(\R)$. Recalling that $\gamma$ vanishes on $[0,1]$, we obtain by Proposition~\ref{prop:semiprinrec} a constant $C_1 \geq 1$ such that~\eqref{e:mtest200} holds for $t \in [0,1]$. Next, Taylor's theorem, the fact that $0 = S_p^0(1) = S_h^0(1) + \widetilde{S}_r^0(1)$, identities~\eqref{e:intr} and~\eqref{e:ysmall} and estimates~\eqref{e:nlest7},~\eqref{e:nlest78},~\eqref{e:nlest77},~\eqref{e:nlest79},~\eqref{e:nlest4} and~\eqref{e:nlest5} yield a constant $C_3 > 0$ such that
\begin{align} \label{e:mtest23}
\begin{split}
\|y(1) - \breve{y}(1)\|_\infty &\leq \|y(1) + \tfrac{\nu}{d} r(1)\|_\infty + \|\breve{y}(1) - \tfrac{\nu}{d} S_h^0(1) v_0\|_\infty + \tfrac{|\nu|}{d} \|r(1) - \widetilde{S}_r^0(1)v_0\|_\infty\\
&\leq C_3E_0^2,
\end{split}
\end{align}
where we use that $\eta(t) \leq M_0 E_0$ for all $t \geq 0$. Finally, we use $\breve{y}(t) = \smash{\re^{(d \partial_\xx^2 + a \partial_\xx) (t-1)} \breve{y}(1)}$, recall the definition~\eqref{e:defy} of $y(t)$, and apply the mean value theorem, Proposition~\ref{prop:semiprinrec}, identity~\eqref{e:inty} and estimates~\eqref{e:nlest92} and~\eqref{e:mtest23} and $\eta(t) \leq M_0E_0$ to bound
\begin{align*} 
\begin{split}
\left\|\gamma(t) - \breve{\gamma}(t)\right\|_\infty &\lesssim \|r(t)\|_\infty + \left\|y(t) - \breve{y}(t)\right\|_\infty \lesssim \left\|r(t)\right\|_\infty + E_0^2 + \eta(t)^2,\\
\left\|\gamma_\xx(t) - \breve{\gamma}_\xx(t)\right\|_\infty &\lesssim \|r_\xx(t)\|_\infty + \left\|y_\xx(t) - \breve{y}_\xx(t)\right\|_\infty + \left\|y(t) - \breve{y}(t)\right\|_\infty\left\|y_\xx(t)\right\|_\infty
\\ &\lesssim \left\|r_\xx(t)\right\|_\infty + \frac{E_0^2 + \eta(t)^2}{\sqrt{1+t}}.
\end{split}
\end{align*}
for $t \geq 1$. Thus, noting that $\eta(t) \leq M_0E_0$ yields constants $C_2,C_3 \geq 1$ such that~\eqref{e:mtest21} is satisfied for all $t \geq 1$. Combining the latter with~\eqref{e:mtest200} proves~\eqref{e:mtest3} upon taking $M \geq \max\{C_1,C_3\}$.
\end{proof}

\begin{remark}\label{rem:choice_eta}
\upshape
We motivate the choice of temporal weights in the template function $\eta(t)$ in the proof of Theorem~\ref{t:mainresult}. First, the weights applied to $\smash{\|\vt(t)\|_{W^{1,\infty}}}$, $\|\gamma(t)\|_\infty$, $\|y(t)\|_\infty$, $\|\gamma_\zeta(t)\|_\infty$ and $\|y_\xx(t)\|_\infty$ reflect the decay of the linear terms in the respective Duhamel formulations~\eqref{e:intvt},~\eqref{e:intgamma} and~\eqref{e:inty}, cf.~Propositions~\ref{prop:full},~\ref{prop:semexp},~\ref{prop:semdif} and~\ref{prop:semiprinrec}. The same holds for the weights applied to $\|r(t)\|_\infty$, $\|z(t)\|_\infty$, $\|r_\xx(t)\|_\infty$ and $\|(\gamma_{\xx\xx}(t),\widetilde{\gamma}(t))\|_{W^{2,\infty} \times W^{2,\infty}}$ up to a logarithmic correction to accommodate integration of a factor $(1+s)^{-1}$ or a factor $(1+t-s)^{-1}$, when estimating certain nonlinear terms in their Duhamel formulations~\eqref{e:intgamma},~\eqref{e:intz} and~\eqref{e:intr}, cf.~estimates~\eqref{e:nlest22},~\eqref{e:nlest2},~\eqref{e:nlest4},~\eqref{e:nlest5} and~\eqref{e:nlest6}.   

In contrast, the weights applied to $\smash{\|\partial_\xx^4 \vt(t)\|_\infty}$ and $\|z_{\xx\xx}(t)\|_\infty$ do not reflect the decay of the linear terms, but are motivated by the bounds on certain \emph{nonlinear} terms in the Duhamel formulations~\eqref{e:intvt} and~\eqref{e:intz}. Indeed, as $\|\vt(t)\|_\infty$ is only bounded in $t$, the nonlinearity $\smash{\widetilde{\Non}(\vt(t))}$ does not decay and the same holds for the semigroup $\re^{\El_0 t}$ by Proposition~\ref{prop:full} and Corollary~\ref{cor:der}. Thus, integrating $\smash{\re^{\El(t-s)} \widetilde{\Non}(\vt(s))}$ in~\eqref{e:intvt} yields a bound of order $\mathcal{O}(t)$ on $\smash{\|\partial_\xx^4 \vt(t)\|_\infty}$, cf.~\eqref{e:nlest15}. Similarly, since $\|v(t)\|_\infty$ decays at rate $\smash{t^{-\frac{1}{2}}}$, the quadratic nonlinearity $\mathcal Q(v(s),\gamma(s))$ decays at rate $s^{-1}$ yielding a bound of order $\mathcal{O}(\log(t))$ on $\|z_{\xx\xx}(t)\|_\infty$ upon integrating $\re^{\El_0(t-s)}\mathcal Q(v(s),\gamma(s))$ in~\eqref{e:nlesta}.

Finally, the weights applied to $\|\partial_\xx^3 \vt(t)\|_\infty$ and $\|z_\xx(t)\|_\infty$ arise by interpolation, whereas the weight applied to $\|\vt_{\xx\xx}(t)\|_{L^2}$ is directly linked to that applied to $\|z(t)\|_{W^{2,\infty}}$ through the mean value theorem, see estimate~\eqref{e:nlest12}.
\end{remark}

\section{Discussion and open problems} \label{sec:discussion}

We compare our main result, Theorem~\ref{t:mainresult}, to previous results on the nonlinear stability of wave trains in reaction-diffusion systems and formulate open questions.

\subsection{Optimality of decay rates} \label{sec:optimal}

The lack of localization leads to a loss of a factor $\smash{t^{-\frac{1}{2}}}$ in the decay rates exhibited in Theorem~\ref{t:mainresult} in comparison with earlier results~\cite{JONZ,JUN,SAN3} considering localized perturbations. Nevertheless, we argue that the temporal decay rates in~Theorem~\ref{t:mainresult} are sharp (up to possibly a logarithm). Indeed, estimate~\eqref{e:mtest3} shows that the viscous Hamilton-Jacobi equation~\eqref{e:HamJac} captures the leading-order behavior of the perturbed wave train, cf.~Remark~\ref{rem:fronts}. Moreover,~\eqref{e:HamJac} reduces to the linear convective heat equation 
\begin{align} \partial_t \breve y = d\breve{y}_{\xx\xx} + a \breve{y}_{\xx}, \label{e:conv} \end{align}
upon applying the Cole-Hopf transform. It is well-known that bounded solutions to~\eqref{e:conv} stay bounded and any spatial or temporal derivative contributes a decay factor $\smash{t^{-\frac{1}{2}}}$. This diffusive behavior is reflected by the bounds in Theorem~\ref{t:mainresult}, showing that they are sharp (up to possibly a logarithm).

As far as the author is aware, the only other work concerning the nonlinear stability of wave trains against $C_{\mathrm{ub}}$-perturbations is~\cite{HDRS22}. In the special case of the real Ginzburg-Landau equation, considered in~\cite{HDRS22}, one find that the coefficient $\nu$ in front of the nonlinearity in the viscous Hamilton-Jacobi equation~\eqref{e:HamJac} vanishes due to reflection symmetry, cf.~\cite[Section~3.1]{DSSS}. The bounds obtained in~\cite{HDRS22} are also shown to be optimal and coincide with the ones in Theorem~\ref{t:mainresult}. Here, we note that the nonlinearity in~\eqref{e:HamJac} influences the profile of bounded solutions, but does not alter the rates at which the solutions and their derivatives decay, see Remark~\ref{rem:fronts} directly below.

\begin{remark}\label{rem:fronts}
{\upshape
For any $\breve \gamma_\pm \in \R$ the viscous Hamilton-Jacobi equation~\eqref{e:HamJac} admits monotone front solutions of the form 
\begin{align*} \breve \gamma_{\mathrm f}(\xx,t) = \breve \gamma_- + \left(\breve \gamma_+ - \breve \gamma_-\right)\mathrm{erf}\left(\frac{\xx + a(1+t)}{\sqrt{d(1+t)}}\right), \qquad \mathrm{erf}(x) := \frac{1}{\sqrt{4\pi}} \int_{-\infty}^x \re^{-\frac{y^2}{4}} \de y,
\end{align*}
for $\nu = 0$ and 
\begin{align*} \breve \gamma_{\mathrm f}(\xx,t) = \breve \gamma_- + \frac{d}{\nu} \log\left(1 + \beta\,\mathrm{erf}\left(\frac{\xx + a(1+t)}{\sqrt{d(1+t)}}\right)\right),
\end{align*}
for $\nu \neq 0$, where $\beta$ is defined through $d\log(1+\beta) = \nu\left(\breve \gamma_+ - \breve \gamma_-\right)$. For each $j,l \in \mathbb N_0$, there exists a constant $C \geq 1$, which is independent of $\breve \gamma_\pm$, such that we have
\begin{align}
\frac{\max\left\{|\breve \gamma_-|,|\breve \gamma_+|\right\}}{C(1+t)^{\frac{j}{2}+l}} \leq \left\|\partial_\xx^j \left(\partial_t - a\partial_\xx\right)^l \breve \gamma_{\mathrm{f}}(t)\right\|_\infty \leq \frac{C\max\left\{|\breve \gamma_-|,|\breve \gamma_+|\right\}}{(1+t)^{\frac{j}{2}+l}}, \label{e:mtest5}
\end{align}
for $t \geq 0$. Hence, if $v_0 \in C_{\mathrm{ub}}(\R)$ in Theorem~\ref{t:mainresult} is chosen such that $\smash{\widetilde{\Phi}}_0^* v_0 = \breve\gamma_{\mathrm{f}}(0)$,  then~\eqref{e:mtest3} and~\eqref{e:mtest5} imply that $\breve \gamma = \breve \gamma_{\mathrm f}$ is indeed a leading-order approximation of the phase $\gamma$.
}\end{remark}

\subsection{Modulation of the phase} \label{sec:modnonlocal}

Various works~\cite{IYSA,JONZNL,JONZW,JUNNL,SAN3} study the nonlinear stability and asymptotic behavior of \emph{modulated} wave trains. Here, one considers solutions to~\eqref{RD} with initial conditions of the form
\begin{align} \label{e:modulated} u(x,0) = \phi_0(k_0x + \gamma_0(x)) + v_0(x),\end{align}
where the initial perturbation $v_0$ is bounded and localized, and the initial phase off-set $\gamma_0(x)$ converges to asymptotic limits $\gamma_\pm$ as $x \to \pm \infty$. As these analyses crucially rely on localization-induced decay, $\gamma_0'$ is required to be sufficiently localized (as is $v_0$). We note that, so long as $\gamma_0$ is small in $C_{\mathrm{ub}}(\R)$, initial conditions of the form~\eqref{e:modulated} can be handled by Theorem~\ref{t:mainresult}. One finds that the estimates~\eqref{e:mtest11},~\eqref{e:mtest12} and~\eqref{e:mtest2} in Theorem~\ref{t:mainresult} are the same (or even sharper) than the ones obtained in~\cite{IYSA,JONZNL,JONZW,JUNNL,SAN3}.

It is an open question to describe the asymptotic dynamics of initial data of the form~\eqref{e:modulated}, where $\gamma_0$ is large, but bounded, as in~\cite{IYSA}, and only require $\gamma_0'$ and $v_0$ to be small in $C_{\mathrm{ub}}$-norm. This is currently under investigation by the author. One of the challenges is that the current analysis relies on optimal diffusive smoothing, i.e., it requires the optimal decay on $\|\partial_x \smash{\re^{\partial_{x}^2 t}} \gamma_0\|_\infty$, which is obtained by bounding it as $t^{-\frac{1}{2}} \|\gamma_0\|_\infty$, thus exploiting that $\|\gamma_0\|_\infty$ is small. 

\subsection{Modulation of the wavenumber}

Since a phase modulation is directly coupled to a modulation of the wavenumber, cf.~\cite{DSSS}, another natural class of initial conditions for~\eqref{RD} is
\begin{align}
u(x,0) = \phi\left(k_0x + \gamma_0(x);k_0 + \gamma_0'(x)\right) + v_0(x), \label{e:initcond2}
\end{align} 
where we recall that wave trains $u_k(x,t) = \phi(kx - \omega(k) t;k)$ exist in~\eqref{RD} for an open range of wavenumbers $k$ around $k_0$, see Proposition~\ref{prop:family}. 

The treatment of nonlocalized modulations of the wavenumber is a widely open problem. Although our main result permits initial data of the form~\eqref{e:initcond2} as long as $\|\gamma_0\|_\infty$ and $\|\gamma_0'\|_\infty$ are sufficiently small, it does not apply to the case where the wavenumber off-set $\gamma_0'(x)$ converges to different values $k_\pm$ as $x \to \pm \infty$. Indeed, in such a case $\gamma_0$ cannot be bounded. For $\gamma_0'(x) \to k_\pm$ as $x \to \pm \infty$, it was shown in~\cite{DSSS} that the solution to~\eqref{RD2} with initial condition~\eqref{e:initcond2} can, in the comoving frame $\zeta = k_0x-\omega_0t$, be written as
\begin{align*}  u(\zeta,t) = \phi\big(\zeta + \breve\gamma(\zeta,t);k_0 + \breve k(\zeta,t)\big) + v(\zeta,t), \end{align*}
where the phase variable $\breve\gamma(\zeta,t)$ satisfies the viscous Hamilton-Jacobi equation~\eqref{e:HamJac} with initial condition $\gamma_0$, the wavenumber modulation $\breve k(\zeta,t) = \partial_\zeta \breve\gamma(\zeta,t)$ satisfies the associated viscous Burgers' equation~\eqref{e:Burgers}, and the residual $v(\zeta,t)$ stays small on large, but bounded, time intervals. It is an open question whether this holds for \emph{all} times $t \geq 0$. So far, this has only been resolved for the special case of the real Ginzburg-Landau equation in~\cite{BK92,GALMI} by crucially exploiting its gauge and reflection symmetries. The analysis in this paper might provide a useful new opening to this problem. Indeed, since our $L^\infty$-framework does not rely on localization properties, it might be helpful to accommodate the nonlocalized wavenumber modulation.

\subsection{Approximation of phase and wavenumber} \label{sec:approx}

Estimate~\eqref{e:mtest3} in Theorem~\ref{t:mainresult} provides a global approximation of the phase by a solution to the viscous Hamilton-Jacobi equation~\eqref{e:HamJac}. In addition, the local wavenumber, represented by the derivative of the phase, is approximated by a solution to the viscous Burgers' equation~\eqref{e:Burgers}. Thereby, previous approximation results for modulated wave trains in~\cite{DSSS}, allowing for nonlocalized initial phase and wavenumber off-sets, are extended from local to global ones. Yet, the most general approximation result in~\cite{DSSS} only requires that the solution to the Burgers' equation has initial data in some uniformly local Sobolev space. It is still open whether such general initial data can be handled globally.

We note that global approximation results have previously been obtained in~\cite{JONZW,SAN3} for the case where the local wavenumber $\gamma_\xx(t)$ and the perturbation $v(t)$ are localized. This leads to slightly sharper bounds, where the term $E_0^2$ on the right-hand side of~\eqref{e:mtest3} does not appear. It is an open question whether the bound on the right-hand side of~\eqref{e:mtest3} can be sharpened for $C_{\mathrm{ub}}$-perturbations.

\begin{remark}{\upshape
We note that under the identification $1+ \breve \gamma_\xx = \Upsilon_\xx$ and $1+\breve k = \kappa$, the viscous Hamilton-Jacobi equation~\eqref{e:HamJac} and Burgers' equation~\eqref{e:Burgers} can be regarded as quadratic approximants of the Hamilton-Jacobi equation
\begin{align} \partial_t \Upsilon = \mathsf{d}\!\left(k_0\Upsilon_\xx\right)\Upsilon_{\xx\xx} + \omega(k_0)\Upsilon_\xx - \omega\left(k_0\Upsilon_\xx\right), \label{e:HamJac2}
\end{align}
and the associated Whitham equation
\begin{align}
\partial_t \kappa = \left(\mathsf{d}\!\left(k_0\kappa\right)\kappa_\xx\right)_{\xx} + \left(\omega(k_0)\kappa - \omega\left(k_0\kappa\right)\right)_\xx, \label{e:Whitham}
\end{align}
respectively, where we denote $\mathsf{d}(k) = \smash{k^2\big\langle \widetilde{\Phi}_0,D \phi_0' + 2kD \partial_{\zeta k} \phi(\cdot;k)\big\rangle_{L^2(0,1)}}$ so that $\mathsf{d}(k_0) = d$, cf.~\eqref{e:defad}. In the case where the local wavenumber $\gamma_\xx(t)$ and the perturbation $v(t)$ are localized, the equations~\eqref{e:HamJac} and~\eqref{e:HamJac2}, as well as~\eqref{e:Burgers} and~\eqref{e:Whitham}, are asymptotically equivalent, cf.~\cite[Appendix~A]{JONZW}. We expect that, by following the approach in~\cite[Appendix~A]{JONZW} and using the Cole-Hopf transform to eliminate quadratic nonlinearities, this asymptotic equivalence can also be established in our case of $C_{\mathrm{ub}}$-initial data. Thus, the approximating solution $\breve\gamma(t)$ in Theorem~\ref{t:mainresult} could be replaced by corresponding solutions to the Hamilton-Jacobi and Whitham equations~\eqref{e:HamJac2} and~\eqref{e:Whitham}.
}\end{remark}

\subsection{Higher spatial dimensions} 

We conjecture that the analysis in this paper generalizes effortlessly to wave trains $u_0(x,t) = \phi_0(k_0x_1-\omega_0 t)$ in reaction-diffusion systems 
\begin{align} \label{e:multiD} \partial_t u = D\Delta u + f(u), \qquad x \in \R^d, \,t \geq 0,\,u(x,t) \in \R^n,\end{align}
in higher spatial dimensions $d \geq 2$, where $\Delta$ denotes the $d$-dimensional Laplacian. In contrast to previous results in the literature considering the nonlinear stability of wave trains against localized perturbations in higher spatial dimensions, cf.~\cite{RS21,JONZVD12,OHZVD3,UECR}, we do not expect that the decay rates in Theorem~\ref{t:mainresult} improve upon increasing the spatial dimension. Indeed, our analysis fully relies on diffusive smoothing, while these earlier results exploit localization-induced decay, which is stronger in higher spatial dimensions, i.e., the heat semigroup $\smash{\re^{\Delta t}}$ decays at rate $\smash{t^{-\frac{d}{2}}}$ as an operator from $L^1(\R^d)$ into $L^\infty(\R^d)$, whereas $\partial_{x_1} \re^{\Delta t}$ exhibits decay at rate $\smash{t^{-\frac{1}{2}}}$ as an operator on $L^\infty(\R^d)$. In summary, we conjecture that, if the wave-train solution $u_0$ of~\eqref{e:multiD} is diffusively spectrally stable, then it is nonlinearly stable against perturbations in $\smash{C_{\mathrm{ub}}(\R^d)}$ and the perturbed solution decays, in the co-moving frame $\zeta_1 = k_0 x_1 - \omega_0 t$, at rate $\smash{t^{-\frac12}}$ towards the modulated wave train $u(\zeta_1 + \gamma(\zeta_1,x_2,\ldots,x_d,t),t)$ as in Corollary~\ref{cor:psit}. We expect that the phase modulation $\gamma(t)$ is approximated by a solution to (a higher-dimensional version of) the scalar Hamilton-Jacobi equation.

In fact, there is a trade-off between decay and localization in higher spatial dimensions, in the sense that partly localized perturbations, which are for instance only bounded in $d_1$ spatial dimensions and $L^1$-localized in the remaining $d_2$ transverse spatial dimensions, decay pointwise at rate $\smash{t^{-\frac{d_2}{2}}}$. This principle was used in~\cite{RS21}, where nonlinear stability of wave trains in planar reaction-diffusion systems was established against perturbations, which are only bounded along a line in $\R^2$, but decay exponentially in distance from that line. 

\subsection{Regularity control and extension of the class of reaction-diffusion systems}

A standard issue in the nonlinear stability analysis of wave trains is that the equation~\eqref{e:modpertbeq} for the modulated perturbation $v(t)$ is quasilinear. Consequently, an apparent loss of derivatives has to be addressed in order to close the nonlinear argument. In this work we rely on a method developed in~\cite{RS21} to control regularity. The idea is to establish a tame estimate on a sufficiently high derivative $\partial_\zeta^m \widetilde{v}(t)$ of the unmodulated perturbation $\widetilde{v}(t)$, which satisfies the semilinear equation~\eqref{e:umodpert} in which no derivatives are lost. Upon taking the integer $m$ large enough, sufficiently strong bounds on lower derivatives of $\widetilde{v}(t)$ can be obtained by interpolation. Using the mean value theorem, these bounds can then be transferred into bounds on derivatives of $v(t)$. In this work we need $m = 4$, but, due to parabolic smoothing, it suffices to work with bounded and uniformly continuous initial data, see also Remarks~\ref{rem:linfty} and~\ref{rem:choice_eta}. Nevertheless, the method of~\cite{RS21} itself does not rely on parabolic smoothing and we expect it to apply to general semilinear problems as long as initial perturbations are regular enough (so that one can accomodate sufficiently large $m$). For instance, in~\cite{HJPR21} the method of~\cite{RS21} has been employed in the nonlinear stability analysis of wave trains against localized perturbations in the Lugiato-Lefever equation, which is nonparabolic. We also point out that in certain cases it can be advantageous to instead consider the so-called \emph{forward-modulated} perturbation
\begin{align*}
u(\zeta,t) - \phi_0(\zeta + \gamma(\zeta,t)),
\end{align*}
to control regularity in the nonlinear argument. We refer to the recent survey~\cite{ZUM23} for more details.

The above observations suggest that parabolic smoothing of the underlying equation is not essential for the nonlinear stability analysis of periodic wave trains against $C_{\mathrm{ub}}$-perturbations. Thus, we expect that our approach could in principle be extended to general dissipative semilinear problems. Naturally, this extension hinges on the assumption that the linearization about the wave train generates a $C_0$-semigroup on $C_{\mathrm{ub}}(\R)$, which obeys a spectral mapping property.

Another natural question is how far we can broaden the class of equations \emph{within} the parabolic framework. Taking the previous considerations into account, we are confident that the current analysis can be extended without much effort to the semilinear reaction-diffusion system
\begin{align*}
\partial_t u = Du_{xx} + f(u,u_x).
\end{align*}
Yet, an extension to parabolic quasilinear equations such as the reaction-diffusion system
\begin{align} \label{e:quasi}
\partial_t u = (D(u) u_x)_x + f(u,u_x),
\end{align}
where $D \colon \R^n \to \R^{n \times n}$ is smooth and $D(u)$ is strongly elliptic for each $u \in \R^n$, requires an alternative approach to control regularity in the nonlinear stability argument, because the quasilinear nature of~\eqref{e:quasi} is naturally inherited by the equations governing the (unmodulated) perturbation. Such an approach has been developed in the context of traveling shocks in viscous conservation laws by Howard and Zumbrun in~\cite[Section~11.3]{ZUH}, see also~\cite{HOW15}. Their approach, which relies on the classical parametrix method of Friedman, Ladyzhenskaya and Levi, recasts the quasilinear equation for the perturbation as a linear parabolic problem with space- and time-dependent coefficient functions. Given sufficient control on the perturbation, one can bound the coefficient functions, which yields pointwise short-time estimates on the Green's function associated with the linear problem. Spatial or temporal derivatives of the perturbation can then be bounded in $L^\infty(\R)$ by the perturbation itself through its Green's function representation as a solution to the linear problem by placing derivatives on the Green's function. We refer to~\cite{HOW15,ZUH} for further details. It is an interesting question for future research whether the results in this paper can be extended to quasilinear problems such as~\eqref{e:quasi}. 

\appendix

\section{A pointwise estimate} \label{app:aux}

In order to establish $L^\infty$-bounds on the diffusive part of the semigroup $\re^{\El_0 t}$, we use the following slight adaptation of~\cite[Lemma~A.1]{HDRS22}.

\begin{lemma} \label{lem:semigroupEstimate1}
Let $n \in \mathbb N$, $m \in \mathbb N_0$ and $\xi_0 > 0$. Let $\lambda \in C^4\big((-\xi_0,\xi_0),\C\big)$ and $\chi \in C^4\big(\R^3,\C^{n \times n}\big)$. Suppose there exists constants $C, \mu > 0$ such that it holds
\begin{itemize}
\item[i)] $\lambda'(0) \in \ri \R$;
\item[ii)] $\Re \, \lambda(\xi) \leq -\mu \xi^2$ for all $\xi \in (-\xi_0,\xi_0)$;
\item[iii)] $\overline{\mathrm{supp}(\chi(\cdot,\xx,\xt))} \subset (-\xi_0,\xi_0)$ for all $\xx,\xt \in \R$;
\item[iv)] $\displaystyle\sup_{(\xx,\xt) \in \R^2} \left\|\chi(\cdot,\xx,\xt)\right\|_{W^{4,\infty}} \leq C$.
\end{itemize}
Let $a \in \R$ be such that $\lambda'(0) = a\ri$. Then, the pointwise estimate
\begin{align*}
\left|\int_\R \re^{t \lambda(\xi)} \xi^m \chi(\xi,\xx,\xt) \re^{\ri \xi(\xx-\xt)} \de \xi \right| &\lesssim t^{-\frac{m+1}{2}}\left(1 + \frac{(\xx-\xt + a t)^4}{t^2}\right)^{-1},
\end{align*}
is satisfied for each $t \geq 1$ and $\xx,\xt \in \R$. 
\end{lemma}
\begin{proof} 
We rewrite
\begin{align}
\begin{split}
\int_\R \re^{t \lambda(\xi)} \xi^m \chi(\xi,\xx,\xt) \re^{\ri \xi (\xx - \xt)} \de \xi
&= \left(1+\frac{(\xx-\xt+ a t)^4}{t^2}\right)^{-1} \left(\int_\R \re^{t \lambda(\xi)} \xi^m \chi(\xi,\xx,\xt) \re^{\ri \xi(\xx - \xt)} \de \xi\right.\\
&\quad \left. + \, \int_\R \re^{t \left(\lambda(\xi) - \lambda'(0) \xi\right)} \xi^m \chi(\xi,\xx,\xt) \frac{(\xx-\xt+a t)^4}{t^2} \re^{\ri \xi(\xx - \xt + a t)} \de \xi\right).
\end{split} \label{eq:split3}
\end{align}
for $t \geq 1$ and $\xx,\xt \in \R$. To bound the first integral in~\eqref{eq:split3} we use ii), iii) and iv) and obtain
\begin{align*}
\left|\int_\R \re^{t \lambda(\xi)} \xi^m \chi(\xi,\xx,\xt) \re^{\ri \xi(\xx - \xt)} \de \xi\right| \lesssim \int_{-\xi_0}^{\xi_0} \left|\xi^m \re^{t \lambda(\xi)}\right| \de \xi \lesssim \int_\R |\xi|^m \re^{-\mu \xi^2 t} \de \xi \lesssim t^{-\frac{m+1}{2}},
\end{align*}
for $t \geq 1$ and $\xx,\xt \in \R$. To bound the second integral in~\eqref{eq:split3}, we exploit $(\xx-\xt + at)^4 \re^{\ri \xi(\xx-\xt+at)} = \partial_\xi^4 \re^{\ri \xi(\xx-\xt + at)}$. Thus, using integration by parts we arrive at
\begin{align*}
\begin{split}
&\int_\R \re^{t \left(\lambda(\xi)-\lambda'(0)\xi\right)} \xi^m \chi(\xi,\xx,\xt) (\xx-\xt+at)^4 \re^{\ri \xi(\xx - \xt+at)} \de \xi\\ 
&\qquad = \int_\R \partial_\xi^4 \left(\re^{t \left(\lambda(\xi)-\lambda'(0)\xi\right)} \xi^m \chi(\xi,\xx,\xt)\right) \re^{\ri \xi(\xx - \xt+at)} \de \xi, 
\end{split}
\end{align*}
for $t \geq 1$ and $\xx,\xt \in \R$. We observe that by hypotheses i), ii) and iv) and the fact that $\lambda \in C^4\big((-\xi_0,\xi_0),\C\big)$ it holds 
\begin{align*}
\begin{split}
\left|\partial_\xi \left(\re^{t (\lambda(\xi)-\lambda'(0)\xi)} \chi(\xi,\xx,\xt)\right)\right| &\lesssim \left(1+t|\xi|\right)\re^{-\mu \xi^2 t},\\
\left|\partial_\xi^2 \left(\re^{t (\lambda(\xi)-\lambda'(0)\xi)} \chi(\xi,\xx,\xt)\right)\right| &\lesssim \left(t + \left(1 + t|\xi|\right)^2\right)\re^{-\mu \xi^2 t},\\
\left|\partial_\xi^3 \left(\re^{t (\lambda(\xi)-\lambda'(0)\xi)} \chi(\xi,\xx,\xt)\right)\right| &\lesssim \left(t + t^2|\xi| + \left(1 + t|\xi|\right)^3\right)\re^{-\mu \xi^2 t}\\
\left|\partial_\xi^4 \left(\re^{t (\lambda(\xi)-\lambda'(0)\xi)} \chi(\xi,\xx,\xt)\right)\right| &\lesssim \left(t + t^2 + t^2|\xi| + t^3|\xi|^2 +  \left(1 + t|\xi|\right)^4\right)\re^{-\mu \xi^2 t}
\end{split}
\end{align*}
for $\xi \in (-\xi_0,\xi_0)$, $t \geq 1$ and $\xx,\xt \in \R$. Therefore, using hypotheses ii), iii) and iv), we arrive at
\begin{align*}
\frac{1}{t^2} &\left|\int_\R \partial_\xi^4 \left(\re^{t \left(\lambda(\xi)-\lambda'(0)\xi\right)} \xi^m \chi(\xi,\xx,\xt)\right) \re^{\ri \xi(\xx - \xt+at)} \de \xi\right| \\
&\lesssim \int_\R \re^{-\mu \xi^2 t} \left(|\xi|^m \left(t^{-1} + 1 + |\xi| + t |\xi|^2 + \left(t^{-\frac{1}{2}} + t^{\frac{1}{2}}|\xi|\right)^4\right) \right.\\ 
&\left. \phantom{\left(t^{\frac{1}{2}}\right)^4} + m|\xi|^{m-1} \left(t^{-1} + |\xi| + t^{-\frac{1}{2}}\left(t^{-\frac{1}{2}} + t^{\frac{1}{2}}|\xi|\right)^3\right) + m(m-1)|\xi|^{m-2} \left(t^{-1} + \left(t^{-1} + |\xi|\right)^2\right)\right.\\
&\left. \phantom{\left(t^{\frac{1}{2}}\right)^4} + m(m-1)(m-2) |\xi|^{m-3}\left(t^{-2}+t^{-1}|\xi|\right) + m(m-1)(m-2)(m-3)|\xi|^{m-4}t^{-2}\right) \de \xi\\ 
&\lesssim t^{-\frac{m+1}{2}},
\end{align*}
for $t \geq 1$ and $\xx,\xt \in \R$, which completes the proof of the result.
\end{proof} 

\section{Local existence of the phase modulation} \label{app:B}

We prove the existence of a maximal solution to the integral equation~\eqref{e:intgamma} using a standard contraction mapping argument, cf.~\cite[Section~6]{Pazy} and~\cite[Section~4.3]{CA98}. 

\begin{proof}[Proof of Proposition~\ref{p:gamma}]
First, we note that $\gamma \colon [0,t_0] \to C_{\mathrm{ub}}(\R)$ given by $\gamma(t) = 0$ is a solution of~\eqref{e:intgamma} for each $t_0 \in (0,1]$ with $t_0 < T_{\max}$, because the propagator $S_p^0(t)$ vanishes on $[0,1]$. Next, take $t_0, \delta, \eta > 0$ such that $t_0 + \delta < T_{\max}$ and $\eta < \frac12$. Let $\check{\gamma} \in C\big([0,t_0],C_{\mathrm{ub}}^2(\R)\big) \cap C^1\big([0,t_0],C_{\mathrm{ub}}(\R)\big)$ be a solution to~\eqref{e:intgamma} with $\|(\check\gamma(t), \partial_t \check\gamma(t))\|_{W^{2,\infty} \times L^\infty} \leq \frac12 - \eta$ for all $t \in [0,t_0]$. We show that $\check{\gamma}$ can be extended to a solution $\gamma_{\mathrm{ext}} \in C\big([0,t_0+\delta],C_{\mathrm{ub}}^2(\R)\big) \cap C^1\big([0,t_0+\delta],C_{\mathrm{ub}}(\R)\big)$ to~\eqref{e:intgamma}, provided $\delta > 0$ is sufficiently small.

Since $S_p^0(t)$ vanishes on $[0,1]$, it follows from~\eqref{e:intgamma} that $\check{\gamma}(t) = 0$ for all $t \in [0,t_0]$ with $t \leq 1$. Let $B$ be the ball centered at the origin of radius $\frac12$ in $C_{\mathrm{ub}}^2(\R) \times C_{\mathrm{ub}}(\R)$. By Proposition~\ref{p:local_unmod} and the mean-value theorem the map $V \colon C_{\mathrm{ub}}(\R) \times [t_0,t_0+\delta] \to C_{\mathrm{ub}}(\R)$ given by
\begin{align*}V(\gamma,t)[\xx] &= \vt(\xx-\gamma(\xx),t) + \phi_0(\xx-\gamma(\xx)) - \phi_0(\xx),\end{align*}
is continuous in $t$ and satisfies the Lipschitz estimate
\begin{align*}
\left\|V(\gamma,t) - V(\tilde{\gamma},t)\right\|_\infty \leq \left(\|\vt_\xx(t)\|_\infty + \|\phi_0'\|_\infty\right) \|\gamma - \tilde{\gamma}\|_\infty
\end{align*}
for $\gamma,\tilde\gamma \in C_{\mathrm{ub}}(\R)$ and $t \in [t_0,t_0+\delta]$ with $t \geq 1$. Hence, the nonlinear maps $B \times [t_0,t_0+\delta] \to C_{\mathrm{ub}}(\R), \, (\gamma,\gamma_t,t) \mapsto \mathcal{Q}(V(\gamma,t),\gamma), \mathcal{R}(V(\gamma,t),\gamma,\gamma_t), \mathcal{S}(V(\gamma,t),\gamma)$ are bounded, continuous in $t$ and obey\begin{align*}
\left\|\mathcal{Q}(V(\gamma,t),\gamma) - \mathcal{Q}(V(\tilde\gamma,t),\tilde\gamma)\right\|_\infty &\lesssim \left\|\gamma - \tilde{\gamma}\right\|_{W^{1,\infty}},\\
\left\|\mathcal{R}(V(\gamma,t),\gamma,\gamma_t) - \mathcal{R}(V(\tilde\gamma,t),\tilde\gamma,\tilde\gamma_t)\right\|_\infty &\lesssim \left\|\gamma - \tilde{\gamma}\right\|_{W^{2,\infty}} + \left\|\gamma_t - \tilde{\gamma}_t\right\|_{\infty},\\
\left\|\mathcal{S}(V(\gamma,t),\gamma) - \mathcal{S}(V(\tilde\gamma,t),\tilde\gamma)\right\|_\infty &\lesssim \left\|\gamma - \tilde{\gamma}\right\|_{W^{1,\infty}},
\end{align*}
for $(\gamma,\gamma_t), (\tilde\gamma,\tilde\gamma_t) \in B$ and $t \in [t_0,t_0 + \delta]$ with $t \geq 1$. On the other hand, Proposition~\ref{prop:semdif} yields that the propagators $\smash{\partial_t^\ell S_p^0(t) \partial_\xx^i} \colon C_{\mathrm{ub}}(\R) \to C_{\mathrm{ub}}^l(\R)$ are $t$-uniformly bounded, strongly continuous on $[0,\infty)$ and vanish identically on $[0,1]$ for any $i,l,\ell \in \mathbb N_0$. A standard contraction mapping argument, cf.~\cite[Theorem~6.1.2]{Pazy}, on a closed ball $\mathcal{B}$ of radius $\tfrac12(1 - \eta)$ centered at the origin in the Banach space $\{(\gamma,\gamma_t) \in C\big([t_0,t_0 + \delta],C_{\mathrm{ub}}^2(\R) \times C_{\mathrm{ub}}(\R)\big) : (\gamma,\gamma_t)(s) = 0 \text{ for all } s \in [0,1]\}$ now yields a solution $(\gamma,\gamma_t) \in \mathcal{B}$ to the integral system
\begin{align*}
\gamma(t) &= S_p^0(t)v_0 + \int_0^{t_0} S_p^0(t-s) \mathcal{N}\big(V(\check\gamma(s),s),\check\gamma(s),\partial_s \check\gamma(s)\big) \de s\\ 
&\qquad + \, \int_{t_0}^t S_p^0(t-s) \mathcal{N}\big(V(\gamma(s),s),\gamma(s),\gamma_t(s)\big) \de s,\\
\gamma_t(t) &= \partial_t S_p^0(t)v_0 + \int_0^{t_0} \partial_t S_p^0(t-s) \mathcal{N}\big(V(\check\gamma(s),s),\check\gamma(s),\partial_s \check\gamma(s)\big) \de s\\
&\qquad + \, \int_{t_0}^t \partial_t S_p^0(t-s) \mathcal{N}\big(V(\gamma(s),s),\gamma(s),\gamma_t(s)\big) \de s,
\end{align*}
provided $\delta > 0$ is sufficiently small. By construction it holds $\gamma \in C^1\big([t_0,t_0+\delta],C_{\mathrm{ub}}(\R)\big)$ with $\partial_t \gamma(t) = \gamma_t(t)$ for all $t \in [t_0,t_0+\delta]$. It follows that
\begin{align*}
\gamma_{\mathrm{ext}}(t) = \begin{cases} \check{\gamma}(t), & t \in [0,t_0), \\ \gamma(t), & t \in [t_0,t_0+\delta],\end{cases}
\end{align*}
is a solution to~\eqref{e:intgamma} lying in $C\big([0,t_0+\delta],C_{\mathrm{ub}}^2(\R)\big) \cap C^1\big([0,t_0+\delta],C_{\mathrm{ub}}(\R)\big)$, which extends $\check{\gamma}$. Moreover, we have the estimate $\|(\gamma_{\mathrm{ext}}(t), \partial_t \gamma_{\mathrm{ext}}(t))\|_{W^{2,\infty} \times L^\infty} \leq \frac12(1 - \eta)$ for all $t \in [0,t_0+\delta]$. 

As argued in~\cite[Theorem~4.3.4]{CA98} and~\cite[Theorem~6.1.4]{Pazy}, this extension procedure yields the existence of the desired maximal solution, whose regularity properties follow by the fact that the propagators $\smash{\partial_t^\ell S_p^0(t) \partial_\xx^i} \colon C_{\mathrm{ub}}(\R) \to C_{\mathrm{ub}}^l(\R)$ are $t$-uniformly bounded for any $i,l,\ell \in \mathbb N_0$.
\end{proof}

\bibliographystyle{abbrv}
\bibliography{mybib}

\end{document}